\documentclass[11pt, reqno]{amsart}

\usepackage{amsmath}
\usepackage{amssymb}
\usepackage{amsthm}
\usepackage{biblatex}
\usepackage{enumerate}
\usepackage{fullpage}
\usepackage{graphicx} 
\usepackage{hyperref}
\usepackage{thmtools}
\usepackage{xcolor}

\numberwithin{equation}{section}

\newcommand{\B}{\mathbb{B}}
\newcommand{\C}{\mathbb{C}}
\newcommand{\disk}{\mathbb{D}}
\newcommand{\Half}{\mathbb{H}}
\newcommand{\N}{\mathbb{N}}
\newcommand{\R}{\mathbb{R}}
\renewcommand{\S}{\mathbb{S}}
\newcommand{\Y}{\mathbb{Y}}
\newcommand{\A}{\mathcal{A}}
\newcommand{\D}{\mathcal{D}}
\newcommand{\X}{\mathcal{X}}
\newcommand{\Hil}{\mathcal{H}}

\newcommand{\eps}{\varepsilon}
\renewcommand{\hat}{\widehat}
\newcommand{\pd}{\partial}
\renewcommand{\Tilde}{\widetilde}
\newcommand{\Vector}[2]{\begin{pmatrix}#1\\#2\end{pmatrix}}

\DeclareMathOperator{\dist}{dist}
\DeclareMathOperator{\Real}{Re}
\DeclareMathOperator{\rg}{rg}

\declaretheorem[
    name=Theorem,
	numberwithin=section,
	]{thm}

\declaretheorem[
    name=Proposition,
	numberwithin=section,
	]{prop}

\declaretheorem[
    name=Definition,
	style=definition,
	numberwithin=section,
	]{defin}

\declaretheorem[
	name=Lemma,
	numberwithin=section
	]{lem}

\declaretheorem[
	name=Corollary,
	numberwithin=section
	]{cor}

\title{Blowup stability of wave maps without symmetry}
\author{Roland Donninger}
\email{roland.donninger@univie.ac.at}
\address{Universit\"at Wien, Fakult\"at f\"ur Mathematik,
  Oskar-Morgenstern-Platz 1, 1090 Vienna, Austria}
\thanks{This research was funded in whole or in part by the
Austrian Science Fund (FWF) 10.55776/P34560,
10.55776/PIN2161424, and 10.55776/PAT9429324. For open access purposes, the authors have
applied a CC BY public copyright license to any author-accepted manuscript version arising from this submission.}

\author{Frederick Moscatelli}
\email{frederick.moscatelli@univie.ac.at}
\address{Universit\"at Wien, Fakult\"at f\"ur Mathematik,
  Oskar-Morgenstern-Platz 1, 1090 Vienna, Austria}

\begin{document}

\begin{abstract}
We study wave maps from $(1+d)$-dimensional Minkowski space into the $d$-sphere without any symmetry assumptions. There exists an explicit self-similar blowup solution and we prove that this solution is asymptotically stable under small perturbations of the initial data. The proof is fully rigorous and requires no numerical input whatsoever.
\end{abstract}

\maketitle
\tableofcontents
\section{Introduction}
\noindent Understanding large-data solutions of supercritical dispersive
equations remains one of the great challenges in contemporary
analysis. While the small-data theory is typically reasonably well understood, the phenomena one
encounters in large-data evolutions are often multifaceted and
unmanageably complex. Despite this, there are particular aspects that
one might hope to understand theoretically. A famous instance of this
concerns the formation of singularities for the wave maps equation with
spherical target in dimensions three and higher. It has been known for a long time that wave maps, starting from perfectly smooth initial data,
may develop singularities (or \emph{blow up}) in finite time \cite{Sha88}. The exciting point, however, is the conjectured \emph{universality} of the blowup profile. 
Indeed, numerical simulations suggest that the ``generic''
blowup is governed by a single explicit
self-similar solution $U_*$, at least under the assumption of corotational symmetry \cite{BizChmTab00}.
This indicates that $U_*$ is a
fundamental object and it is
worthwhile to devote a substantial amount of theoretical effort to studying $U_*$ and solutions nearby. This is the topic of
the present paper, where we prove the full nonlinear
asymptotic stability of $U_*$ without any symmetry assumptions. To our
knowledge, this is the first blowup stability result for a
supercritical geometric wave equation in full generality. The proof requires the entire power of the
functional-analytic machinery we
have been developing over the last 15 years or so and a number of
novel ideas on top. Before we go into more details about that, we give a precise formulation of our result.

We consider wave maps from $(1+d)$-dimensional Minkowski space  $\R^{1,d} = \R\times \R^d$ to the $d$-dimensional sphere $\S^d \subseteq \R^{d+1}$. The (extrinsic) wave maps equation for a function $U:\R^{1,d} \to \S^d$ is given by
\begin{align}\label{ex wm eq}
    \pd^\mu \pd_\mu U + (\pd^\mu U \cdot \pd_\mu U)U = 0,
\end{align}
where Einstein's summation convention is in effect\footnote{The derivative with respect to the $\mu$-th slot is denoted by $\pd_\mu$ with the conventions $\pd^0 := - \pd_0$ and $\pd^j := \pd_j$ for $j \in \{1,\ldots,d\}$. Greek indices like $\mu$ run from $0$ to $d$ and latin letters like $i,j,k$ run from $1 $ to $d$.}.
Eq.~\eqref{ex wm eq} has the explicit self-similar solution
\begin{align*}
    U_\ast(t,x) &:= V_\ast\left( \frac{x}{1 - t} \right),
\end{align*}
where
\begin{align*}
    V_\ast(\xi) := \frac{1}{d-2 + |\xi|^2}\Vector{2\sqrt{d-2}\, \xi}{d- 2- |\xi|^2}.
\end{align*}
This solution was first shown to exist by Shatah \cite{Sha88} for $d=3$ and its explicit form was found by Turok and Spergel \cite{SpeTur90}. For $d\geq 4$, the explicit expression is due to Bizo\'n and Biernat \cite{BizBie15}.
Even though the initial data $(U_\ast(0,\cdot),\pd_0 U_\ast(0,\cdot))$ are smooth, $U_\ast$ suffers a gradient blowup at the origin as $t \nearrow 1$ and is thus an explicit example for singularity formation in finite time from smooth initial data. Eq.~\eqref{ex wm eq} is invariant under spacetime translations, rotations, Lorentz boosts, and rotations on the target, see \autoref{symmetries}. By applying the various continuous symmetries of Eq.~\eqref{ex wm eq} to $U_*$, we obtain a family of blowup solutions $\left\{U_{\Theta}^{T,X}: T > 0, X \in \R^d, \Theta \in \R^{p(d)}\right\}$, where $p(d):=\frac{d(d+3)}{2}$ and $U_\Theta^{T,X}$ has a gradient blowup at $(t,x) = (T,X)$. For $T,R>0$ and $X\in \mathbb R^d$, we introduce the balls $\B^d_R(X):=\{x\in \R^d: |x-X|<R\}$, $\B^d_R:=\B^d_R(0)$, and the backward lightcone
\begin{align*}
    \mathcal{C}_{T,X} := \{ (t,x) \in [0,T)\times \R^d: |x- X| \leq T - t  \}.
\end{align*}
With this notation at hand, we can now formulate the main result of this paper.
\begin{thm}\label{main result}
For every $d \geq 3$ there exists $k_0 = k_0(d)$ such that the following holds for every $k \geq k_0$. There exist constants $0 < c_\ast,\delta_\ast \leq 1$ and $\eps_\ast > 0$ such that for all $(F,G) \in C^\infty(\overline{\B_2^d},\R^{d+1})^2$ with 
\begin{itemize}
    \item $F(x ) \in \S^{d}$ and $F(x) \cdot G(x) = 0$ for all $x \in \overline{\B_2^d}$
    \item $\|F - U_\ast(0,\cdot)\|_{H^k(\B_2^d,\R^{d+1})} + \|G - \pd_0 U_\ast(0,\cdot)\|_{H^{k-1}(\B_2^d,\R^{d+1})} \leq c_\ast \delta_\ast$ 
\end{itemize}
there exist $(T_\ast,X_\ast,\Theta_\ast) \in [1-\delta_\ast,1+\delta_\ast] \times \overline{\B_{\delta_\ast}^d} \times \overline{\B_{\delta_\ast}^{p(d)}}$ and a unique classical solution $U$ of 
\begin{align*}
    \begin{cases}
        \pd^\mu \pd_\mu U + (\pd^\mu U \cdot \pd_\mu U) U = 0 \quad \mbox{in} \quad
        \mathcal C_{T_\ast,X_\ast}\\
        U(0,\cdot) = F,\quad \pd_0 U(0,\cdot) = G \quad \mbox{on} \quad \overline{\B_{T_\ast}^d(X_\ast)}
    \end{cases}
\end{align*}
such that 
\begin{align}\label{U close to blowup}
    (T_\ast -t )^{s - \frac{d}{2}}\left\|U(t,\cdot) - U_{\Theta_\ast}^{T_\ast,X_\ast}(t,\cdot)\right\|_{H^s(\B_{T_\ast - t}^d(X_\ast),\R^{d+1})}    \lesssim\left( T_\ast - t\right)^{\eps_\ast}
\end{align}
for $0 \leq s \leq k$ and
\begin{align}\label{pd0U close to blowup}
    (T_\ast - t)^{s - \frac{d}{2}} \left\|\pd_t U(t,\cdot) - \pd_t  U_{\Theta_\ast}^{T_\ast,X_\ast}(t,\cdot)\right\|_{H^{s-1}(\B_{T_\ast - t}^d(X_\ast),\R^{d+1})}\lesssim \left( T_\ast - t\right)^{\eps_\ast}
\end{align}
for $1 \leq s \leq k$ hold for all $t \in [0,T_\ast)$.
\end{thm}
Some remarks are in order.
\begin{itemize}
    \item By classical solution $U$ we mean that $U$ is jointly $C^2$ in $\mathcal{C}_{T_\ast,X_\ast}$ and solves Eq.~\eqref{ex wm eq} pointwise there.
    \item Initially, we perturb the wave map $U_*$. Since a perturbation will in general trigger the symmetries of the equation, the evolution does not converge back to $U_*$ itself but to a slightly different member of the family of blowup solutions generated by applying the symmetries of the equation to $U_*$.
    \item The weights in \eqref{U close to blowup} and \eqref{pd0U close to blowup} arise since 
    \begin{align*}
        \left\|U_{\Theta_\ast}^{T_\ast,X_\ast}(t,\cdot)\right\|_{H^s(\B_{T_\ast - t}^d(X_\ast),\R^{d+1})} &\simeq (T_\ast - t)^{\frac{d}{2} - s}\\
        \left\|\pd_t U_{\Theta_\ast}^{T_\ast,X_\ast}(t,\cdot)\right\|_{H^s(\B_{T_\ast - t}^d(X_\ast),\R^{d+1})} &\simeq (T_\ast - t)^{\frac{d}{2} - s-1}
    \end{align*}
    as $t \nearrow T_\ast$ for $s \geq 0$. Hence $U(t,\cdot)$ converges to $U_{\Theta_\ast}^{T_\ast,X_\ast}(t,\cdot)$ as $t \nearrow T_\ast$ relative to $\left\|U_{\Theta_\ast}^{T_\ast,X_\ast}(t,\cdot)\right\|_{H^s(\B_{T_\ast - t}^d(X_\ast),\R^{d+1})}$.
    \item \autoref{main result} should be seen as a ``high regularity'' result in the sense that the minimal required regularity, $k_0(d)$, grows like $d^\frac{3}{2}$, see \eqref{k0}, and is in particular far away from the critical regularity $s_c = \frac{d}{2}$.
\end{itemize}
\subsection{Related results}
The wave maps equation attracted a lot of interest in the modern literature on partial differential equations. This is because it is a very elegant geometric model with a rich phenomenology but still simple enough to be amenable to rigorous analysis. Consequently, by thoroughly studying the wave maps equation one hopes to discover fundamental mechanisms that may be at work also in much more complicated problems like Einstein's equation of general relativity. Singularity formation for the model at hand was studied numerically in the influential paper \cite{BizChmTab00}. In the corotational case, the blowup stability problem was solved completely in the series of works \cite{Don11, DonSchAic12, CosDonXia16, CosDonGlo17, ChaDonGlo17, Glo25}, in dimensions $d=3,4$ even in the optimal topology \cite{DonWal23, DonWal25}. In the course of this, a canonical spectral-theoretic framework emerged that was further refined and successfully applied to many other self-similar blowup problems, e.g.~\cite{DonSch12, DonSch14, Don14, DonSch16, Don17, DonRao20, MerRapRodSze22, Ost24, Glo24, GhoLiuMas25, Gou25, LiuRae25, GloHofLuo25, GloHilWal25, GuoHadJanSch25}. In fact, we benefit a lot from novel insight in \cite{MerRapRodSze22} and \cite{GhoLiuMas25}.

A literature review on singularity formation for wave maps, even when restricted to self-similar blowup, would be grossly ignorant without at least mentioning some of the groundbreaking contributions in $d=2$. The blowup mechanism in $d=2$, which is the energy-critical case, is of an entirely different nature and not self-similar \cite{BizChmTab01, KriSchTat08, RodSte10, RapRod12, KriShu20, JenKri25}. In view of the topic of the present paper, the recent \cite{KriShuSch24} is interesting as it develops a stability theory of blowup outside of symmetry. Furthermore, due to the global control provided by the energy, remarkably general, nonperturbative results on large-data dynamics are available, see e.g.~\cite{SteTat10a, SteTat10b, KriSch12, CotKenLawSch15a, CotKenLawSch15b, JenLaw18, JenLaw25}.
\subsection{Symmetries}
\label{symmetries}
Eq.~\eqref{ex wm eq} has various symmetries that have to be incorporated into the analysis. Let $U$ be a solution of Eq.~\eqref{ex wm eq}. First, we can translate in space-time, i.e., $U^{T,X}(t,x) := U(t - T,x- X)$ is a solution for $T > 0$ and $X \in \R^d$. Next, scaling would also be a symmetry, however, since we are ultimately interested in the self-similar solution $U_\ast$, this corresponds to a space-time translation and can hence be neglected.

Next, we have rotations on the domain. For a matrix $R \in SO(d)$ we get a solution $(t,x) \mapsto U(t,R x)$. Similarly, for a Lorentz boost $L$ we get the solution $ (t,x) \mapsto U(L(t,x))$. It will be convenient to combine these symmetries with the space-time translations to guarantee that the point $(T,X)$ is fixed. Hence for fixed $(T,X)$ we define the new solutions
\begin{align*}
    U_R(t,x) &:= U(t, R(x - X) + X)\\
    U_L(t,x) &:= U(L(t-T,x-X) + (T,X))
\end{align*}
for a rotation $R$ and a Lorentz boost $L$.

Finally, we have rotations on the target $\S^d$. For $\mathcal{R} \in SO(d+1)$ we obtain a new solution $U_\mathcal{R}(t,x) := \mathcal{R} U(t,x)$. For $U_\ast$ this yields many redundancies. Namely, if $\mathcal{R}$ leaves the last component invariant, i.e., 
\begin{align*}
    \mathcal{R} \Vector{\Tilde{y}}{y^{d+1}} = \Vector{R \Tilde{y}}{y^{d+1}}
\end{align*}
for some $R \in SO(d)$, then one computes $\mathcal{R} U_\ast(t,x) = U_\ast(t,R x)$, which hence corresponds to a rotation on the domain.

To parametrize the symmetries, we define the infinitesimal generators $\{A_{i j}\}_{1 \leq i < j \leq d}$, $\{B_i\}_{1 \leq i \leq d}$, $\{C_i\}_{1 \leq i \leq d}$ as
\begin{align*}
    (A_{i j})_{a b} &= \begin{cases}
        1, &a=i, b=j\\
        -1, &a=j, b=i\\
        0,&\text{else}
    \end{cases}\\
    (B_i)_{\mu \nu} &= \begin{cases}
        1, &\mu=0, \nu=i\\
        1, &\mu=i, \nu=0\\
        0, &\text{else}
    \end{cases}\\
    (C_i)_{a b} &= \begin{cases}
        1, &a=i, b=d+1\\
        -1, &a=d+1, b=i\\
        0,&\text{else}
    \end{cases}.
\end{align*}
Then $\exp(\eps A_{i j})$ is a rotation on $\R^d$ in the $(x^i,x^j)$-plane with angle $\eps$, $\exp(\eps B_i)$ is a Lorentz boost in direction $e_i$ of rapidity $\eps$ and $\exp(\eps C_i)$ is a rotation on $\R^{d+1}$ in the $(y^i,y^{d+1})$-plane with angle $\eps$.

For $\alpha = (\alpha_{i j})_{1 \leq i < j \leq d} \in \R^{\frac{d(d-1)}{2}}$, $\beta = (\beta_i)_{1 \leq i  \leq d} \in \R^d$, $\gamma = (\gamma_i)_{1 \leq i \leq d} \in \R^d$ we define 
\begin{align*}
    R_\alpha &:= \exp(\alpha_{12} A_{12})\ldots \exp(\alpha_{1d}A_{1d}) \exp(\alpha_{23} A_{23})\ldots \exp(\alpha_{2d} A_{2d}) \ldots \exp(\alpha_{d-1,d} A_{d-1,d}) \\
    L_\beta &:= \exp(\beta_1 B_1)\ldots \exp(\beta_d B_d) \\
    \mathcal{R}_\gamma &:= \exp(\gamma_1 C_1) \ldots \exp(\gamma_d C_d).
\end{align*}
We set $\Theta = (\alpha,\beta,\gamma) \in \R^{\frac{d(d-1)}{2}} \times \R^d \times \R^d = \R^{p(d)}$, where
\begin{align*}
    p(d) := \frac{d(d-1)}{2} + d + d = \frac{d(d + 3)}{2}.
\end{align*}
With this we define 
\begin{align*}
    U_\Theta^{T,X} := (((U_\ast^{T-1,X})_{\mathcal{R}_\gamma})_{L_\beta})_{R_\alpha},
\end{align*}
which is explicitly given by 
\begin{align*}
    U_{\Theta}^{T,X}(t,x) = \mathcal{R}_\gamma U_\ast(L_\beta(t - T, R_\alpha(x - X)) + (1,0) ).
\end{align*}
$U_\Theta^{T,X}$ solves Eq.~\eqref{ex wm eq} and the blowup point is shifted from $(1,0)$ to $(T,X)$. Note in this notation $U_\ast = U_0^{1,0}$.
\subsection{Intrinsic formulation}
We will consider the wave maps equation in a chart. For a chart $\Psi:\S^{d} \supseteq W \to V \subseteq \R^d$, Eq.~\eqref{ex wm eq} transforms into
\begin{align}\label{int wm eq}
    \pd^\mu \pd_\mu (\Psi \circ U)^n + \Gamma_{i j}^n(\Psi \circ U) \pd^\mu (\Psi \circ U)^i \pd_\mu (\Psi\circ U)^j = 0,\quad n = 1,\ldots,d,
\end{align}
where $\Gamma_{i j}^n$ are the associated Christoffel symbols. We choose $\Psi$ to be the stereographic projection with respect to the south pole $S = (0,\ldots,0,-1)$. Concretely for $y = (\Tilde{y},y^{d+1})$
\begin{align*}
    \Psi(\Tilde{y},y^{d+1}) &= \frac{1}{1 + y^{d+1}} \Tilde{y}\\
    \Psi^{-1}(z) &= \frac{1}{1 + |z|^2}\Vector{2 z}{ 1 - |z|^2}
\end{align*}
which maps $\S^d \setminus \{S\}$ bijectively to $\R^d$ and one has 
\begin{align*}
    \Gamma_{i j}^n(z) = -\frac{2}{1 + |z|^2}(z_i \delta_j^n + z_j \delta_i^n - z^n \delta_{i j}).
\end{align*}
This choice of $\Psi$ is very convenient for us since $V_\ast$ is, up to a scaling factor of $\sqrt{d-2}$, the same as $\Psi^{-1}$.
\subsection{Similarity coordinates}
Since $U_{\Theta}^{T,X}$ is self-similar, it is standard to make the change of variables 
\begin{align*}
    (\tau,\xi) = \chi^{T,X}(t,x) = \left( -\log(T-t) + \log T, \frac{x-X}{T-t} \right)
\end{align*}
with inverse transformation
\begin{align*}
    (\chi^{T,X})^{-1}(\tau,\xi) = \left( T - T e^{-\tau}, X + T e^{-\tau} \xi  \right)
\end{align*}
for $T > 0$ and $X \in \R^d$.
This transformation maps the truncated backward light cone with vertex $(T,X)$ to an infinite cylinder with the blowup point shifted to infinity, i.e., 
\begin{align*}
    \chi^{T,X} : \mathcal{C}_{T,X} \to [0,\infty) \times \overline{\B^d}.
\end{align*}
Then, setting 
\begin{align*}
    v(\tau,\xi) := (\Psi \circ U)(T - T e^{-\tau},X + T e^{-\tau} \xi)
\end{align*}
transforms Eq.~\eqref{int wm eq} into
\begin{align}\label{int wm eq selfsim}
    0 &= \left(\pd_\tau^2 + \pd_\tau + 2 \xi^j \pd_{\xi^j}\pd_\tau - (\delta^{i j} - \xi^i \xi^j)\pd_{\xi^i} \pd_{\xi^j} + 2 \xi^j \pd_{\xi^j} \right)v^n(\tau,\xi)\\
    &\notag+ \Gamma_{i j}^n(v(\tau,\xi)) \left[(\pd_\tau + \xi^m \pd_{\xi^m})v^i(\tau,\xi) (\pd_\tau + \xi^m \pd_{\xi^m})v^j(\tau,\xi) - \pd^m v^i(\tau,\xi) \pd_m v^j(\tau,\xi) \right], \quad n=1,\ldots,d.
\end{align}
\subsection{Notation}
We denote by $|\cdot|$ the Euclidean norm on $\R^d$ and by $\cdot$ the Euclidean inner product on $\R^d$. We denote the open, resp.~closed, ball in $\R^d$ of radius $R > 0$ centered around $x \in \R^d$ by $\B_R^d(x), \overline{\B_R^d(x)}$, i.e., 
\begin{align*}
    \B_R^d(x) &= \{ y \in \R^d: |y - x| < R\}\\
    \overline{\B_R^d(x)} &= \{ y \in \R^d: |y - x| \leq R\}
\end{align*}
and we set $\B_R^d := \B_R^d(0)$ and $\B^d := \B_1^d$. For disks in $\C$ of radius $R$ around $z \in \C$ we use the notation $\disk_R(z),\overline{\disk_R(z)}$. We denote for $w \in \R$ the half-spaces
\begin{align*}
    \Half_w &= \{\lambda \in \C: \Real \lambda > w\}\\
    \overline{\Half_w} &= \{\lambda \in \C: \Real \lambda   \geq  w\},
\end{align*}
with $\Half := \Half_0$. We denote the $(d-1)$-dimensional unit sphere in $\R^d$ by $\S^{d-1}$ and its (not normalized) surface measure by $\sigma$.

For $T>0$ and $X \in \R^d$ we denote the truncated backward lightcone with vertex $(T,X)$ by $\mathcal{C}_{T,X}$, i.e.,
\begin{align*}
    \mathcal{C}_{T,X} = \{ (t,x) \in [0,T)\times \R^d: |x- X| \leq T - t  \}.
\end{align*}

The space of smooth, complex-valued functions on a ball $\B_R^d(x)$ are denoted by $C^\infty(\B_R^d)$. The space of smooth functions up to the boundary are denoted by $C^\infty(\overline{\B_R^d(x)})$, which can equivalently be defined as 
\begin{align*}
    C^\infty(\overline{\B_R^d(x)}) = \{f \in C^\infty(\B_R^d(x)): \pd^\alpha f \in L^\infty(\B_R^d) \quad \forall \alpha \in \N_0^d\}.
\end{align*}
For an integer $k \geq 0$, the $L^2$-based Sobolev space on $\B_R^d(x)$ measuring up to $k$ derivatives is denoted by $H^k(\B_R^d(x))$. The space $L^2(\S^{d-1})$ refers to the square-integrable functions with respect to the measure $\sigma$.

If we consider function spaces with values in $\C^m$ instead, we denote them as $C^\infty(\overline{\B_R^d(x)},\C^m)$ etc. All function spaces are by default for complex-valued functions, if we want to stress that something is real-valued, we denote this by $C^\infty(\overline{\B_R^d(x)},\R^m)$ etc.

For a function $f$ its Jacobi matrix is denoted by $D f $. If $f$ has values in $\C$, then we denote its gradient by $\nabla f$.

Lower-case bold face letters such as $\mathbf{f},\mathbf{g}$ denote 2-tuples of functions (that themselves may be vector-valued). Upper-case bold face letters such as $\mathbf{L},\mathbf{P}$ denote operators that act on 2-tuples of functions. For a closed operator $\mathbf{L}$ on a Hilbert space $\Hil$ with domain $\mathcal D(\mathbf L)$, its spectrum is denoted by $\sigma(\mathbf{L})$ and its resolvent set by $\rho(\mathbf{L})$. For $z \in \rho(\mathbf{L})$ the resolvent is denoted by $\mathbf{R}_{\mathbf{L}}(z) := (z \mathbf{I} - \mathbf{L})^{-1}$. Furthermore, for a (direct vector space) decomposition $\Hil = X \Dot{+} Y$, where the part of $\mathbf{L}$ in $X$ is well-defined (see \cite[p.~172-173]{Kat95}), we denote it by $\mathbf{L}_X$. On the other hand, if $X \subseteq \Hil$ is any subspace such that $X \subseteq \D(\mathbf{L})$, then the restriction $\mathbf{L}|_{X}:X \to \Hil$ is defined as $\D(\mathbf{L}|_X) = X$ and $\mathbf{L}|_X \mathbf{f} := \mathbf{L} \mathbf{f}$ for $\mathbf{f} \in X$.
For a bounded operator $\mathbf{L}:\Hil_1 \to \Hil_2$ between Hilbert spaces $\Hil_1,\Hil_2$ the operator norm is denoted by $\|\mathbf{L}\|_{\Hil_1 \to \Hil_2}$. The identity operator (on some implicitly understood space) is denoted by $\mathbf{I}$.

If we have two quantities $A$ and $B$, then $A \lesssim B$ means there exists some constant $C > 0$ such that $A \leq C B$ holds. The constant $C$ is only allowed to depend on parameters that do not matter for the estimate and/or were fixed in advance (for example the dimension $d$). If we want to stress that $C$ \emph{does not} depend on some particular parameter, we quantify it after the inequality. For example, if $A(d,k),B(d,k)$ depend on the parameters $d$ and $k$ and there exists a constant $C(d)$ that only depends on $d$, but not on $k$, such that $A(d,k) \leq C(d) B(d,k)$, then this statement would be: For all $d$ the inequality $A(d,k) \lesssim B(d,k)$ holds for all $k$. We have the usual conventions $A \gtrsim B :\iff B \lesssim A$ and $A \simeq B:\iff A \lesssim B \wedge A \gtrsim B$.
\subsection{Outline of the proof}
We build upon a functional-theoretic framework that was already used to prove stability results for supercritical wave maps in corotational symmetry. However, leaving the class of corotational functions introduces new obstacles that necessitate on the one hand a more intricate spectral analysis and on the other hand a more robust linear framework.
We give a rough outline of the proof.
\begin{itemize}
    \item After switching to the intrinsic formulation of the wave maps equation, we reformulate the equation in the similarity coordinates
    \begin{align*}
        \tau = - \log(T - t) + \log T,\quad \xi = \frac{x-X}{T-t},
    \end{align*}
    which transform the solutions $U_\Theta^{T,X}$ into solutions $v_\Theta$ that are independent of the new time variable $\tau$. In a first order formulation this gives rise to stationary solutions $\mathbf{v}_\Theta$ of 
    \begin{align}\label{outline first order unperturbed}
        \pd_\tau \mathbf{v}(\tau) = \mathbf{L} \mathbf{v}(\tau) + \mathbf{F}(\mathbf{v}(\tau)),
    \end{align}
    where $\mathbf{v}$ is a curve on $[0,\infty)$ into the Hilbert space $\Hil^k = H^k\times H^{k-1}(\B^d,\C^d)$, $\mathbf{L}$ is a spatial differential operator corresponding to the free wave operator (henceforth called \emph{free generator}) and $\mathbf{F}$ is the nonlinearity. We then make the perturbation ansatz $\mathbf{v}(\tau) = \mathbf{v}_\Theta + \mathbf{u}(\tau)$. Inserting this into Eq.~\eqref{outline first order unperturbed} yields
    \begin{align}\label{outline first order perturbed}
        \pd_\tau \mathbf{u}(\tau) = \mathbf{L}\mathbf{u}(\tau) + \mathbf{L}_\Theta \mathbf{u}(\tau) + \mathbf{N}_\Theta(\mathbf{u}(\tau)),
    \end{align}
    where $\mathbf{L}_\Theta$ is a linear operator and $\mathbf{N}_\Theta$ is a nonlinearity.
    \item We first analyze the behavior of the linear equation $\pd_\tau \mathbf{u}(\tau) = (\mathbf{L} + \mathbf{L}_0')\mathbf{u}(\tau)$.
    The standard approach would suggest to consider $\mathbf{L} + \mathbf{L}_0'$ as a perturbation of $\mathbf{L}$, in a way that essentially allows to carry over decay estimates of the semigroup generated by $\mathbf{L}$ to the one generated by $\mathbf{L} + \mathbf{L}_0'$. However, this fails due to two reasons.
    \begin{enumerate}
        \item $\mathbf{L}_0'$ is not compact. This would be needed to obtain information on the spectrum $\sigma(\mathbf{L} + \mathbf{L}_0')$.
        \item $\mathbf{L}_0' \Vector{f_1}{f_2}$ depends on first order derivatives of $f_1$ and on the second component $f_2$. This precludes the establishment of resolvent estimates in a standard way.
    \end{enumerate}
    Both of these issues arise since the nonlinearity of Eq.~\eqref{ex wm eq} depends on first order derivatives of the function, in contrast to the corotational case.
    To overcome these issues, we borrow from \cite{MerRapRodSze22} and \cite{GhoLiuMas25}.
    We rely on \emph{dissipative estimates} of the free generator, specifically the ``homogeneous'' one
    \begin{align}
        \Real (\mathbf{L} \mathbf{f} | \mathbf{f})_{\Dot{\Hil}^k} &\leq \left( \frac{d}{2} - k \right)\|\mathbf{f}\|_{\Dot{\Hil}^k}^2\label{outline free homog diss}.
    \end{align}
    Now the simple but crucial observation is that commuting $\mathbf{L}_0'$ with $k$ derivatives does not induce growth with respect to $k$ in the highest derivative. This yields an estimate of the form
    \begin{align*}
        \Real ((\mathbf{L} + \mathbf{L}_0')\mathbf{f} | \mathbf{f})_{\Dot{\Hil}^k} \leq \left(\frac{d}{2} + C - k\right) \|\mathbf{f}\|_{\Dot{\Hil}^k}^2 + c_k \|\mathbf{f}\|_{\Hil^{k-1}}^2,
    \end{align*}
    where $C$ does not depend on $k$. Hence, taking sufficiently many derivatives yields a negative constant for the highest-order derivative term. One can now trade one derivative for smallness, where the delicate part is to show that it suffices to pay with a finite-rank projection and not a full $L^2$-term. This allows one to upgrade the above bound to a dissipative estimate of the form
    \begin{align*}
        \Real ((\mathbf{L} + \mathbf{L}_0' - C_k \mathbf{P}) \mathbf{f} | \mathbf{f})_{\Hil^k} \leq - w_k \|\mathbf{f}\|_{\Hil^k}^2,
    \end{align*}
    where $C_k,w_k > 0$ and $\mathbf{P}$ is a bounded, finite-rank projection. Using the Lumer-Philips Theorem this shows that $\mathbf{L} + \mathbf{L}_0' - C_k \mathbf{P}$ generates a semigroup with growth bound at most $-w_k$. 
    \item The next step concerns the analysis of the point spectrum of  $\mathbf{L} + \mathbf{L}_0'$. Since eigenvalues of $\mathbf{L} + \mathbf{L}_0'$ that lie in the right-half plane $\overline{\Half}$ correspond to unstable directions of the linear evolution, one has to first exclude these. The points $0$, $1$ are in fact eigenvalues, but these simply reflect the presence of the continuous symmetries of the equation and can be handled. One thus has to prove that no other eigenvalues exist in $\overline{\Half}$ and that all eigenfunctions arise via the symmetries. This is exactly the assertion of \emph{mode stability}. For $d =3$ this was already established in \cite{WeiKocDon25} and in our companion paper \cite{DonMos26} we prove mode stability for the remaining dimensions $d \geq 4$. The eigenvalue equation $(\mathbf{L} +\mathbf{L}_0')\mathbf{f} = \lambda \mathbf{f}$ corresponds to a system of coupled PDEs, which is a major complication compared to the usual situation of non-coupled spectral ODEs. Decoupling these equations has a surprising connection to Lie algebra theory which we systematically exploit for all $d \geq 3$ in \cite{DonMos26}. This step includes a spherical harmonics decomposition, allowing one to reduce the problem to a system of coupled spectral ODEs. This system is then decoupled by revealing a connection to a Casimir operator of a certain irreducible representation of the Lie algebra $\mathfrak{so}(d)$. Even after these simplifications this is still a hard spectral problem due to the nonself-adjoint nature of the operator which precludes the application of standard Sturm-Liouville theory. In \cite{WeiKocDon25,DonMos26}, the mode stability problem is finally solved by employing the \emph{quasi-solution method} pioneered in \cite{CosDonGloHua16, CosDonGlo17}, see also \cite{Don24}.
    \item Next, by exploiting the compactness of $C_k \mathbf{P}$, we deduce that $\sigma(\mathbf{L} + \mathbf{L}_0') \cap \overline{\Half_{-\frac{w_k}{2}}}$ differs from $\sigma(\mathbf{L} + \mathbf{L}_0' - C_k \mathbf{P}) \cap \overline{\Half_{-\frac{w_k}{2}}}$ by at most finitely many eigenvalues and mode stability implies that $\sigma(\mathbf{L} + \mathbf{L}_0')\cap \overline{\Half} = \{0,1\}$. These spectral points can be projected away via the spectral projection $\mathbf{I} - \mathbf{P}_0$. At this point it is also necessary to prove that the spectral projection has the ``correct'' rank, which does not follow automatically due to the lack of self-adjointness. As in \cite{ChaDonGlo17, Glo25}, we use ODE logic to reduce this issue to mode stability. Furthermore, by exploiting \eqref{outline free homog diss} once again, we prove resolvent estimates for large $\lambda \in \overline{\Half}$ as in \cite{GhoLiuMas25}. On the stable subspace $\rg(\mathbf{I} - \mathbf{P}_0)$ we can handle compact regions of $\overline{\Half}$ due to the analyticity of the reduced resolvent on that set. Then the Gearhart-Prüss-Greiner Theorem shows that the semigroup generated by $\mathbf{L} + \mathbf{L}_0'$ restricted to its stable subspace decays exponentially. Finally, by a spectral perturbation argument as in \cite{DonSch16}, we extend these results to the case of small $|\Theta|$ which concludes the linear stability theory.
    \item Obtaining the full nonlinear stability is then standard, using fixed point arguments together with a Lyapunov-Perron argument in order to deal with the symmetries along the lines of \cite{Don11, DonSch16, DonOst24}.
\end{itemize}

\section{Linear stability}
\subsection{Definitions}
Since we are considering at times vector-valued functions, we introduce some notation. Also, we assume always $d \geq 3$. We set for integers $m \geq 1$ and $k \geq 0$
\begin{align*}
    (f|g)_{L^2(\B^d,\C^m)} &:= \int_{\B^d} f^j(\xi) \overline{g_j(\xi)} d\xi\\
    (f|g)_{H^k(\B^d,\C^m)} &:= \sum_{\substack{\alpha \in \N_0^d\\|\alpha|\leq k}}(\pd^\alpha  f | \pd^\alpha g )_{L^2(\B^d,\C^m)}\\
    (f|g)_{\Dot{H}^k(\B^d,\C^m)} &:= \sum_{\substack{\alpha \in \N_0^d\\|\alpha|= k}}(\pd^\alpha  f | \pd^\alpha g )_{L^2(\B^d,\C^m)}
\end{align*}
and similarly 
\begin{align*}
    (f|g)_{L^2(\S^{d-1},\C^m)} &:= \int_{\S^{d-1}} f^j(\omega) \overline{g_j(\omega)} d \sigma(\omega)\\
    (f|g)_{H^k(\S^{d-1},\C^m)} &:= \sum_{\substack{\alpha \in \N_0^d\\|\alpha|\leq k}}(\pd^\alpha  f | \pd^\alpha g )_{L^2(\S^{d-1},\C^m)}\\
    (f|g)_{\Dot{H}^k(\S^{d-1},\C^m)} &:= \sum_{\substack{\alpha \in \N_0^d\\|\alpha|= k}}(\pd^\alpha  f | \pd^\alpha g )_{L^2(\S^{d-1},\C^m)}
\end{align*}
for functions $f,g \in C^\infty(\overline{\B^d},\C^m)$, where $\sigma $ is the standard surface measure on $\S^{d-1}$. The induced (semi-)norms are denoted by $$\|f\|_{H^k(\B^d,\C^m)} = \sqrt{(f | f)_{H^k(\B^d,\C^m)}}$$ and similarly for the other inner products. If $m =1$, we just write $H^k(\B^d)$ instead of $H^k(\B^d,\C^m)$.

We also need inner products for tuple of functions. We define for $\mathbf{f} = \Vector{f_1}{f_2}$ and $\mathbf{g} = \Vector{g_1}{g_2}$
\begin{align*}
    ( \mathbf{f}|\mathbf{g})_{\Hil^1} := (f_1 | g_1)_{\Dot{H^1}(\B^d,\C^m)} + (f_2 | g_2)_{L^2(\B^d,\C^m)} + (f_1 | g_1)_{L^2(\S^{d-1},\C^m)}
\end{align*}
for $f_1,f_2,g_1,g_2 \in C^\infty(\overline{\B^d},\C^m)$. To incorporate higher order derivatives, we define for $i \in \{1,\ldots,d\}$
\begin{align*}
    \mathbf{D}_i \Vector{f_1}{f_2} :=\Vector{\pd_i f_1}{\pd_i f_2},
\end{align*}
where $\pd_i$ acts component-wise. Then we define for $k \geq 2$
\begin{align*}
    (\mathbf{f}|\mathbf{g})_{\Dot{\Hil}^k} &:= (\mathbf{D}^{i_1}\ldots \mathbf{D}^{i_{k-1}} \mathbf{f}|\mathbf{D}_{i_1}\ldots \mathbf{D}_{i_{k-1}}\mathbf{g})_{\Hil^1}\\
    (\mathbf{f}|\mathbf{g})_{\Hil^k} &:= (\mathbf{f}|\mathbf{g})_{\Dot{\Hil}^k} + (\mathbf{f}|\mathbf{g})_{\Hil^1}
\end{align*}
for $\mathbf{f},\mathbf{g} \in C^\infty(\overline{\B^d},\C^m)^2$ with respective induced (semi-)norms. These inner products are (somewhat crude) adaptations from \cite{Ost24,DonOst24}.

Based on the equivalence
\begin{align*}
    \|f\|_{H^1(\B^d)} \simeq \|\nabla f\|_{L^2(\B^d)} + \|f\|_{L^2(\S^{d-1})} \quad \forall f \in C^\infty(\overline{\B^d}),
\end{align*}
see for example \cite[Lemma 2.1]{Ost24}, which readily generalizes to the vector valued case, together with
\begin{align*}
    \|f\|_{H^k(\B^d,\C^m)} \simeq \|f\|_{\Dot{H}^k(\B^d,\C^m)} + \|f\|_{L^2(\B^d,\C^m)} \quad \forall f \in C^\infty(\overline{\B^d},\C^m),
\end{align*}
(see \autoref{standard results prop}) one has for all $k \geq 2$ the following equivalences
\begin{align*}
    \left\| \Vector{f_1}{f_2} \right\|_{\Hil^1} & \simeq \|f_1\|_{H^1(\B^d,\C^m)} + \|f_2\|_{L^{2}(\B^d,\C^m)}\\
    \left\| \Vector{f_1}{f_2} \right\|_{\Hil^k} & \simeq \|f_1\|_{H^k(\B^d,\C^m)} + \|f_2\|_{H^{k-1}(\B^d,\C^m)} \\
    \left\| \Vector{f_1}{f_2} \right\|_{\Dot{\Hil}^k} & \simeq  \|f_1\|_{\Dot{H}^k(\B^d,\C^m)} + \|f_1\|_{\Dot{H}^{k-1}(\B^d,\C^m)} +   \|f_2\|_{\Dot{H}^{k-1}(\B^d,\C^m)}
\end{align*}
for all $f_1,f_2 \in C^\infty(\overline{\B^d},\C^m)$.

Motivated by these equivalences, we define for $k \geq 1$ the space $\Hil^k$ as the completion of $C^\infty(\overline{\B^d},\C^m)^2$ under the norm $\|\cdot\|_{\Hil^k}$, which coincides with the space $H^k\times H^{k-1}(\B^d,\C^m)$.

\subsection{First order formulation}
By setting 
\begin{align*}
    \mathbf{v}(\tau)(\xi) := \Vector{v(\tau,\xi)}{(\pd_\tau + \xi^j \pd_{\xi^j})v(\tau,\xi)},
\end{align*}
we can rewrite Eq.~\eqref{int wm eq selfsim} as the first order system
\begin{align}\label{unperturbed first order sys}
    \pd_\tau \mathbf{v}(\tau) = \Tilde{\mathbf{L}} \mathbf{v}(\tau) + \mathbf{F}(\mathbf{v}(\tau)),
\end{align}
where
\begin{align*}
    \Tilde{\mathbf{L}}\Vector{f_1}{f_2} := \Vector{-\Lambda f_1 + f_2}{\Delta f_1  -\Lambda f_2 - f_2}
\end{align*}
is a (formal) differential operator and
\begin{align*}
    \mathbf{F}\left(  \Vector{f_1}{f_2} \right) := \Vector{0}{-F(f_1,D f_1,f_2)}
\end{align*}
where $(\Lambda f)(\xi) := \xi^j \pd_{\xi^j} f(\xi)$ and
\begin{align*}
    F^n(f_1,D f_1, f_2) := \Gamma_{i j}^n(f_1)(f_2^i f_2^j - \pd^m f_1^i \pd_m f_1^j)
\end{align*}
for $f_1,f_2 \in C^\infty(\overline{\B^d},\C^d)$.

Since the equation 
\begin{align*}
    \pd_\tau \mathbf{v}(\tau) = \Tilde{\mathbf{L}}\mathbf{v}(\tau)
\end{align*}
is formally equivalent to $v_1$ solving the free wave equation in similarity coordinates, we call $\Tilde{\mathbf{L}}$ the \emph{free generator}.
\subsection{Properties of the free generator}
We analyze some crucial properties of $\Tilde{\mathbf{L}}$. Since this requires not a lot of additional effort, we allow here for $\C^m$-valued functions for any integer $m \geq 1$.
\begin{lem}
Let $m \geq 1$ be an integer. We have
\begin{align}\label{free inhomog diss}
    \Real(\Tilde{\mathbf{L}}\mathbf{f}| \mathbf{f}  )_{\Hil^k} \leq \left( \frac{d}{2} - 1  \right)\|\mathbf{f}\|_{\Hil^k}^2, \quad \mathbf{f} \in C^\infty(\overline{\B^d},\C^m)^2
\end{align}
for $k \geq 1$ and
\begin{align}\label{free homog diss}
    \Real(\Tilde{\mathbf{L}}\mathbf{f}| \mathbf{f}  )_{\Dot{\Hil}^k} \leq \left( \frac{d}{2} - k  \right)\|\mathbf{f}\|_{\Dot{\Hil}^k}^2, \quad \mathbf{f} \in C^\infty(\overline{\B^d},\C^m)^2
\end{align}
for $k \geq 2$.
\end{lem}
\begin{proof}
For $\mathbf{f} = \Vector{f_1}{f_2} \in C^\infty(\overline{\B^d},\C^m)^2$ one computes with the divergence theorem
\begin{align*}
    \Real (\Tilde{\mathbf{L}} \mathbf{f} | \mathbf{f})_{\Hil^1} &= \left( \frac{d}{2} -1 \right)\left(\|f_1\|_{\Dot{H}^1(\B^d,\C^m)}^2 + \|f_2\|_{L^2(\B^d,\C^m)}^2 \right)\\
    &-\frac{1}{2} \|f_1\|_{\Dot{H}^1(\S^{d-1},\C^m)}^2 - \frac{1}{2}\|f_2\|_{L^2(\S^{d-1},\C^m)}^2\\
    &+ \Real \int_{\S^{d-1}} (\Lambda f_1^j \overline{f_{2 j}} - \Lambda f_1^j \overline{f_{1 j}} + f_2^j \overline{f_{1 j}}) d \sigma.
\end{align*}
We apply the inequality
\begin{align*}
    \Real (a \overline{b} - a \overline{c} + b \overline{c}) \leq \frac{1}{2} (|a|^2 + |b|^2 + |c|^2)
\end{align*}
to obtain
\begin{align*}
    \Real \int_{\S^{d-1}} (\Lambda f_1^j \overline{f_{2 j}} - \Lambda f_1^j \overline{f_{1 j}} + f_2^j \overline{f_{1 j}}) d \sigma \leq \frac{1}{2} \int_{\S^{d-1}} (|\Lambda f_1|^2 + |f_2|^2 + |f_1|^2) d \sigma.
\end{align*}
For $j \in \{1,\ldots,m\}$ we have 
\begin{align*}
    |\Lambda f_1^j(\omega)|^2 &= |\omega ^i \pd_i f_1^j(\omega)|^2 \leq \underbrace{|\omega|^2}_{=1} |\nabla f_1^j(\omega)|^2 = |\nabla f_1^j(\omega)|^2
\end{align*}
for all $\omega \in \S^{d-1}$. Summing over $j$ and integrating hence yields
\begin{align*}
    \frac{1}{2}\int_{\S^{d-1}} |\Lambda f_1|^2 d \sigma \leq \frac{1}{2} \int_{\S^{d-1}} \pd^i f_1^j \overline{\pd_i f_{1 j}} d \sigma = \frac{1}{2} \|f_1\|_{\Dot{H}^1(\S^{d-1},\C^m)}^2.
\end{align*}
Inserting these estimates gives in total
\begin{align*}
    \Real (\Tilde{\mathbf{L}} \mathbf{f} | \mathbf{f})_{\Hil^1} &\leq \left( \frac{d}{2} -1 \right)\left(\|f_1\|_{\Dot{H}^1(\B^d,\C^m)}^2 + \|f_2\|_{L^2(\B^d,\C^m)}^2 \right) + \frac{1}{2} \|f_1\|_{L^2(\S^{d-1},\C^m)}^2\\
    &\leq \left( \frac{d}{2} - 1 \right)\|\mathbf{f}\|_{\Hil^1}^2,
\end{align*}
since $d \geq 3$.

For the higher derivatives, we note the identity
\begin{align}\label{free comm rel}
    \mathbf{D}^\alpha \Tilde{\mathbf{L}} \mathbf{f} = \Tilde{\mathbf{L}}\mathbf{D}^\alpha  \mathbf{f} - |\alpha|\mathbf{D}^\alpha  \mathbf{f}
\end{align}
for all $\mathbf{f} \in C^\infty(\overline{\B^d},\C^m)^2$ and multi-indices $\alpha \in \N_0^d$, where
\begin{align*}
    \mathbf{D}^\alpha := (\mathbf{D}_1)^{\alpha_1} \ldots (\mathbf{D}_d)^{\alpha_d}.
\end{align*}
We compute
\begin{align*}
    \Real (\Tilde{\mathbf{L}} \mathbf{f}  | \mathbf{f})_{\Dot{\Hil}^k} &= \Real(\mathbf{D}^{i_1} \ldots \mathbf{D}^{i_{k-1}}\Tilde{\mathbf{L}}  \mathbf{f} | \mathbf{D}_{i_1} \ldots \mathbf{D}_{i_{k-1}} \mathbf{f}  )_{\Hil^1}\\
    &\overset{\eqref{free comm rel}}{=} \underbrace{\Real(\Tilde{\mathbf{L}}\mathbf{D}^{i_1} \ldots \mathbf{D}^{i_{k-1}}  \mathbf{f} | \mathbf{D}_{i_1} \ldots \mathbf{D}_{i_{k-1}} \mathbf{f}  )_{\Hil^1}}_{\leq (\frac{d}{2} - 1)(\mathbf{D}^{i_1} \ldots \mathbf{D}^{i_{k-1}} \mathbf{f} | \mathbf{D}_{i_1} \ldots \mathbf{D}_{i_{k-1}} \mathbf{f}  )_{\Hil^1}} \\
    &-(k-1) \Real(\mathbf{D}^{i_1} \ldots \mathbf{D}^{i_{k-1}} \mathbf{f} | \mathbf{D}_{i_1} \ldots \mathbf{D}_{i_{k-1}} \mathbf{f}  )_{\Hil^1}\\
    &\leq\left( \frac{d}{2} - k \right) (\mathbf{D}^{i_1} \ldots \mathbf{D}^{i_{k-1}} \mathbf{f} | \mathbf{D}_{i_1} \ldots \mathbf{D}_{i_{k-1}} \mathbf{f}  )_{\Hil^1}\\
    &= \left( \frac{d}{2} - k \right) \|\mathbf{f}\|_{\Dot{\Hil}^k}^2
\end{align*}
for $k \geq 2$, where we used the dissipative estimate for $\Hil^1$. Finally we get
\begin{align*}
    \Real ( \Tilde{\mathbf{L}}\mathbf{f} | \mathbf{f})_{\Hil^k} &= \Real ( \Tilde{\mathbf{L}}\mathbf{f} | \mathbf{f})_{\Dot{\Hil}^k} + \Real ( \Tilde{\mathbf{L}}\mathbf{f} | \mathbf{f})_{\Hil^1}\\
    &\leq \left(\frac{d}{2 } - k\right)\|\mathbf{f}\|_{\Dot{\Hil}^k}^2 +  \left(\frac{d}{2 } - 1\right) \|\mathbf{f}\|_{\Hil^1}^2\\
    &\leq \left( \frac{d}{2} - 1 \right)\|\mathbf{f}\|_{\Hil^k}^2
\end{align*}
for all $k \geq 2$.
\end{proof}
We can now conclude that the closure of $\Tilde{\mathbf{L}}$ generates a semigroup.
\begin{prop}\label{free semigroup}
Let $m,k \geq 1$ be integers. The operator $\Tilde{\mathbf{L}} : \D(\Tilde{\mathbf{L}}) \subseteq \Hil^k \to \Hil^k$ with domain 
\begin{align*}
    \D(\Tilde{\mathbf{L}}):= C^\infty(\overline{\B^d},\C^m)^2 
\end{align*}
is closable and its closure $\mathbf{L}$ generates a strongly-continuous semigroup $(\mathbf{S}(\tau))_{\tau \geq 0}$, which satisfies
\begin{align*}
    \|\mathbf{S}(\tau) \mathbf{f} \|_{\Hil^k} \leq e^{(\frac{d}{2} - 1)\tau} \|\mathbf{f}\|_{\Hil^k}
\end{align*}
for all $\mathbf{f} \in \Hil^k$.
\end{prop}
\begin{proof}
The operator $\Tilde{\mathbf{L}}$ is densely defined on $\Hil^k$. Also, by applying \cite[Proposition 2.2]{Ost24} to each component, we conclude that $\frac{d}{2} \mathbf{I} - \Tilde{\mathbf{L}}$ has dense range. Thus \eqref{free inhomog diss} together with the Lumer-Philips Theorem \cite[p.~83, Theorem 3.15]{EngNag00} then yields the claim.
\end{proof}

\subsection{Full linear problem}
We apply the change of variables to the solution $U_\Theta^{T,X}$. This yields 
\begin{align*}
    v_\Theta(\xi) := v_\Theta(\tau,\xi) := (\Psi \circ U_\Theta^{T,X})(T - T e^{-\tau},X + Te^{-\tau} \xi),
\end{align*}
which does not depend on $\tau$, since $U_\Theta^{T,X}$ is self-similar, and no longer depends on $(T,X)$, since the transformation itself does. Now we observe that since the range of $U_0^{T,X}$ on $\mathcal{C}_{T,X}$, which is the same as the range of $V_\ast$ on $\overline{\B^d}$, is contained in the upper hemisphere and that $U_0^{T,X}$ is continuous, we find some $M > 0$ sufficiently small such that the corresponding range of $U_\Theta^{T,X}$ is contained in
\begin{align*}
    \left\{(\Tilde{y},y^{d+1}) \in \S^d : y^{d+1} >- \frac{1}{2} \right\}.
\end{align*}
for all $\Theta \in \overline{\B_M^{p(d)}}$. We will restrict ourselves to such $\Theta$ from now on. Since the stereographic projection with respect to the south pole is a diffeomorphism on this part of the sphere, together with the fact that $ v_0$ depends smoothly on its argument and the symmetries depend smoothly on $\Theta$, we conclude that the map 
\begin{align*}
    (\xi,\Theta) \mapsto v_\Theta(\xi)
\end{align*}
is smooth on $\overline{\B^d} \times \overline{\B_M^{p(d)}}$.

Since $v_\Theta$ is a time-independent solution of Eq.~\eqref{int wm eq selfsim}, this gives rise to
\begin{align*}
    \mathbf{v}_\Theta := \Vector{v_\Theta}{\Lambda v_\Theta},
\end{align*}
which is a static solution of Eq.~\eqref{unperturbed first order sys}. We want to perturb this, so we set 
\begin{align*}
    \mathbf{v} = \mathbf{v}_\Theta + \mathbf{u}.
\end{align*}
Then $\mathbf{v}$ formally solves Eq.~\eqref{unperturbed first order sys} if and only if $\mathbf{u}$ solves
\begin{align}\label{perturb first order sys}
    \pd_\tau \mathbf{u}(\tau) = \Tilde{\mathbf{L}} \mathbf{u}(\tau) + \mathbf{L}_{\Theta}'\mathbf{u}(\tau) + \mathbf{N}_\Theta(\mathbf{u}(\tau)),
\end{align}
where
\begin{align*}
    \mathbf{L}_\Theta' \mathbf{f} := \mathbf{F}'(\mathbf{v}_\Theta)\mathbf{f}
\end{align*}
and
\begin{align*}
    \mathbf{N}_{\Theta}(\mathbf{f}) := \mathbf{F}(\mathbf{v}_\Theta  + \mathbf{f}) - \mathbf{F}(\mathbf{v}_\Theta) - \mathbf{F}'(\mathbf{v}_\Theta)\mathbf{f}.
\end{align*}
One computes that
\begin{align*}
    \mathbf{L}'_{\Theta} \Vector{f_1}{f_2} = \Vector{0}{V_\Theta f_1 + V_{\Theta i}\pd^i f_1 + W_\Theta f_2},
\end{align*}
where for $f \in C^\infty(\overline{\B^d},\C^d)$ and $j \in \{1,\ldots,d\}$
\begin{align*}
    (V_{\Theta} f)_j &:= V_{\Theta j n}f^n\\
    (V_{\Theta i} f)_j &:= V_{\Theta i j n} f^n\\
    (W_{\Theta} f)_j &:= W_{\Theta j n}f^n
\end{align*}
with smooth functions $V_{\Theta j n},V_{\Theta i j n},W_{\Theta j n} \in C^\infty(\overline{\B^d})$. In fact, since $v_\Theta(\xi)$ is jointly smooth in $\xi$ and $\Theta$, we conclude that also $V_{\Theta j n}(\xi),V_{\Theta i j n}(\xi),W_{\Theta j n}(\xi)$ are jointly smooth in $\xi$ and $\Theta$ on $\overline{\B^d}\times \overline{\B_M^{p(d)}}$. Note that $\mathbf{L}_\Theta'$ extends to a bounded operator on $\Hil^k$. We immediately get the following.
\begin{prop}
Let $k \geq 1$ and $\Theta \in \overline{\B_M^{p(d)}}$. The operator $\mathbf{L}_\Theta:\D(\mathbf{L}_\Theta) \subseteq \Hil^k \to \Hil^k$ defined by
\begin{align*}
    \mathbf{L}_\Theta = \mathbf{L} + \mathbf{L}_\Theta',\quad \D(\mathbf{L}_\Theta) = \D(\mathbf{L})
\end{align*}
generates a strongly continuous semigroup $(\mathbf{S}_\Theta(\tau))_{\tau \geq 0}$ that satisfies
\begin{align*}
    \|\mathbf{S}_\Theta(\tau) \mathbf{f}\|_{\Hil^k} \leq e^{\left( \frac{d}{2} - 1 + \|\mathbf{L}_\Theta'\|_{\Hil^k \to \Hil^k} \right)\tau} \|\mathbf{f}\|_{\Hil^k} , \quad \forall \mathbf{f} \in \Hil^k.
\end{align*}
\end{prop}
\begin{proof}
Since $\mathbf{L}$ generates a strongly continuous semigroup by \autoref{free semigroup} and $\mathbf{L}_\Theta'$ is bounded on $\Hil^k$, the statement follows from the bounded perturbation theorem \cite[p.~158]{EngNag00}.
\end{proof}
\subsection{Properties of the perturbation}
Here we will analyze some properties of $\mathbf{L}_\Theta'$. One can compute explicitly
\begin{align*}
    \mathbf{L}_0'\Vector{f_1}{f_2}(\xi) = \Vector{0}{\frac{2(d-2 - |\xi|^2) }{d- 2 + |\xi|^2}f_1(\xi) + \frac{4}{d- 2+ |\xi|^2}(|\xi|^2 f_2(\xi) - \Lambda f_1(\xi) +  K f_1(\xi) )  },
\end{align*}
where
\begin{align*}
    (K f)^n(\xi) := \xi^n \pd_j f^j(\xi)  - \xi_j \pd^n f^j(\xi).
\end{align*}
Observe that $\mathbf{L}_0'$ is not compact (only the first term would correspond to a compact operator). However, we would like to mention that it is possible to use a transformation that eliminates the term $\frac{4}{d-2 + |\xi|^2}(|\xi|^2 f_2(\xi) - \Lambda f_1(\xi))$. This idea comes from \cite{CheMcNSch23} and is possible due to the null structure of Eq.~\eqref{ex wm eq}. Unfortunately in our case the term $\frac{4}{d-2 + |\xi|^2} K f_1(\xi)$ would remain under this transformation, which prevents the resulting operator from being compact.

The following proposition is the basis of the heuristic idea that one should be able to deal with $\mathbf{L}_0'$ perturbatively assuming that the highest order derivatives do not induce growth in the regularity $k$.
\begin{prop}
Let $d \geq 3$. There exists a constant $C = C(d) > 0$ such that for all $k \geq 2$ there exists some $c_k > 0$ such that 
\begin{align}\label{Lprime highest derivative}
    |(\mathbf{L}_0' \mathbf{f}|\mathbf{f})_{\Dot{\Hil}^k}| \leq C \|\mathbf{f}\|_{\Dot{\Hil}^k}^2 + c_k \|\mathbf{f}\|_{\Hil^{k-1}}^2, \quad \forall \mathbf{f} \in \Hil^k
\end{align}
holds.
\end{prop}
\begin{proof}
We compute for $\mathbf{f} \in C^\infty(\overline{\B^d},\C^d)^2$, $\alpha \in \N_0^d$ and $j \in \{1,\ldots,d\}$ 
\begin{align*}
    ( \mathbf{D}^\alpha \mathbf{L}_0' \mathbf{f})_{2 j} &= \pd^\alpha(V_{0 j n} f_1^n + V_{0 i j n} \pd^i f_1^n + W_{0 j n}f_2^n )\\
    &= \sum_{\beta \leq \alpha} \binom{\alpha}{ \beta} \left(  \pd^\beta V_{0 j n} \pd^{\alpha - \beta} f_1^n +  \pd^\beta V_{0 i j n} \pd^{\alpha - \beta} \pd^i f_1^n + \pd^\beta W_{0 j n} \pd^{\alpha - \beta} f_2^n  \right)\\
    &=      V_{0 i j n} \pd^{\alpha } \pd^i f_1^n +  W_{0 j n} \pd^{\alpha } f_2^n\\
    &+V_{0 j n} \pd^{\alpha } f_1^n+\sum_{0 \not= \beta \leq \alpha} \binom{\alpha}{ \beta} \left(  \pd^\beta V_{0 j n} \pd^{\alpha - \beta} f_1^n +  \pd^\beta V_{0 i j n} \pd^{\alpha - \beta} \pd^i f_1^n + \pd^\beta W_{0 j n} \pd^{\alpha - \beta} f_2^n  \right).
\end{align*}
For $|\alpha| = k-1$ the last line only contains derivatives of order up to $k-1$ and $k-2$ of $f_1$ and $f_2$, respectively, hence we conclude from the smoothness of the potentials the existence of some constant $C_{k} $ such that
\begin{align*}
    \left\|V_{0 j n} \pd^{\alpha } f_1^n+ \sum_{0 \not= \beta \leq \alpha} \binom{\alpha}{ \beta} \left(  \pd^\beta V_{0 j n} \pd^{\alpha - \beta} f_1^n +  \pd^\beta V_{0 i j n} \pd^{\alpha - \beta} \pd^i f_1^n + \pd^\beta W_{0 j n} \pd^{\alpha - \beta} f_2^n  \right)\right\|_{L^2(\B^d)} \leq C_k \|\mathbf{f}\|_{\Hil^{k-1}}.
\end{align*}
Next we have
\begin{align*}
    \|V_{0 i j n} \pd^\alpha \pd^i f_1^n\|_{L^2(\B^d)} &\leq \sup_{i,j,n} \|V_{0 i j n}\|_{L^\infty(\B^d)} \sum_{i=1}^d \sum_{n=1}^d \|\pd^\alpha \pd^i f_1^n\|_{L^2(\B^d)}\\
    &\leq d \sup_{i,j,n} \|V_{0 i j n}\|_{L^\infty(\B^d)}  \left(\sum_{i=1}^d\sum_{n=1}^d \|\pd^\alpha \pd^i f_1^n\|_{L^2(\B^d)}^2\right)^\frac{1}{2}\\
    &= d \sup_{i,j,n} \|V_{0 i j n}\|_{L^\infty(\B^d)} \|\pd^\alpha f_1\|_{\Dot{H}^1(\B^d,\C^d)}\\
    \|W_{0 j n} \pd^\alpha f_2^n\|_{L^2(\B^d)} &\leq \sup_{j,n}\|W_{0 j n}\|_{L^\infty(\B^d)} \sum_{n=1}^d \|\pd^\alpha f_2^n\|_{L^2(\B^d)}\\
    &\leq d^\frac{1}{2} \sup_{j,n}\|W_{0 j n}\|_{L^\infty(\B^d)} \left( \sum_{n=1}^d \|\pd^\alpha f_2^n\|_{L^2(\B^d)}^2      \right)^\frac{1}{2}\\
    &= d^\frac{1}{2} \sup_{j,n}\|W_{0 j n}\|_{L^\infty(\B^d)} \|\pd^\alpha f_2\|_{L^2(\B^d,\C^d)}.
\end{align*}
Combining this with the previous yields
\begin{align*}
    \|(\mathbf{D}^\alpha \mathbf{L}_0' \mathbf{f})_{2 j}\|_{L^2(\B^d)} &\leq C'(\|\pd^\alpha f_1\|_{\Dot{H}^1(\B^d,\C^d)} + \|\pd^\alpha f_2\|_{L^2(\B^d,\C^d)}) + C_k \|\mathbf{f}\|_{\Hil^{k-1}}\\
    &\leq \sqrt{2} C' \underbrace{(\|\pd^\alpha f_1\|_{\Dot{H}^1(\B^d,\C^d)}^2 + \|\pd^\alpha f_2\|_{L^2(\B^d,\C^d)}^2)^\frac{1}{2}}_{\leq \| \mathbf{D}^\alpha\mathbf{f}\|_{\Hil^1}}+ C_k \|\mathbf{f}\|_{\Hil^{k-1}}\\
    &\leq \sqrt{2} C' \| \mathbf{D}^\alpha\mathbf{f}\|_{\Hil^1}+ C_k \|\mathbf{f}\|_{\Hil^{k-1}},
\end{align*}
where
\begin{align*}
    C' = d^\frac{1}{2} \max\left\{ d^\frac{1}{2}\sup_{i,j,n} \|V_{0 i j n}\|_{L^\infty(\B^d)} , \sup_{j,n}\|W_{0 j n}\|_{L^\infty(\B^d)} \right\}.
\end{align*}
Finally from this we get
\begin{align*}
    \|\mathbf{D}^\alpha \mathbf{L}_0' \mathbf{f}\|_{\Hil^1} &= \|(\mathbf{D}^\alpha \mathbf{L}_0' \mathbf{f})_2\|_{L^2(\B^d,\C^d)} = \left(\sum_{j=1}^d  \|(\mathbf{D}^\alpha \mathbf{L}_0' \mathbf{f})_{2 j}\|_{L^2(\B^d)}^2 \right)^\frac{1}{2} \leq \sum_{j=1}^d  \|(\mathbf{D}^\alpha \mathbf{L}_0' \mathbf{f})_{2 j}\|_{L^2(\B^d)}\\
    &\leq \sqrt{2} d C'\| \mathbf{D}^\alpha\mathbf{f}\|_{\Hil^1}+ d C_k \|\mathbf{f}\|_{\Hil^{k-1}}.
\end{align*}
Since this holds for any $\alpha \in \N_0^d$ with $|\alpha| = k-1$, we have
\begin{align*}
    |(\mathbf{L}_0' \mathbf{f} |\mathbf{f})_{\Dot{\Hil}^k}| &= \left|\sum_{i_1,\ldots,i_{k-1} = 1}^d (\mathbf{D}_{i_1} \ldots \mathbf{D}_{i_{k-1}} \mathbf{L}_0' \mathbf{f} | \mathbf{D}_{i_1} \ldots \mathbf{D}_{i_{k-1}} \mathbf{f} )_{\Hil^1}   \right|\\
    &\leq \sum_{i_1,\ldots,i_{k-1} = 1}^d \|\mathbf{D}_{i_1} \ldots \mathbf{D}_{i_{k-1}} \mathbf{L}_0' \mathbf{f}\|_{\Hil^1} \|\mathbf{D}_{i_1} \ldots \mathbf{D}_{i_{k-1}} \mathbf{f}\|_{\Hil^1}\\
    &\leq \sqrt{2} d C'\underbrace{\sum_{i_1,\ldots,i_{k-1} = 1}^d \|\mathbf{D}_{i_1} \ldots \mathbf{D}_{i_{k-1}}  \mathbf{f}\|_{\Hil^1} \|\mathbf{D}_{i_1} \ldots \mathbf{D}_{i_{k-1}} \mathbf{f}\|_{\Hil^1}}_{= \|\mathbf{f}\|_{\Dot{\Hil}^k}^2}\\
    &+ d C_k \sum_{i_1,\ldots,i_{k-1} = 1}^d \|\mathbf{f}\|_{\Hil^{k-1}} \underbrace{\|\mathbf{D}_{i_1} \ldots \mathbf{D}_{i_{k-1}} \mathbf{f}\|_{\Hil^1}}_{\leq \|\mathbf{f}\|_{\Dot{\Hil}^k}}\\
    &\leq \sqrt{ 2 } d C' \|\mathbf{f}\|_{\Dot{\Hil}^k}^2 + d^k C_k \|\mathbf{f}\|_{\Hil^{k-1}} \|\mathbf{f}\|_{\Dot{\Hil}^k}\\
    &\leq \sqrt{ 2 } d C' \|\mathbf{f}\|_{\Dot{\Hil}^k}^2 + \frac{1}{2} \|\mathbf{f}\|_{\Dot{\Hil}^k}^2 + \frac{d^{2 k} C_k^2}{2} \|\mathbf{f}\|_{\Hil^{k-1}}^2\\
    &= C \|\mathbf{f}\|_{\Dot{\Hil}^k}^2 + c_k \|\mathbf{f}\|_{\Hil^{k-1}}^2
\end{align*}
with 
\begin{align*}
    C &= 2 d C' + \frac{1}{2}\\
    c_k &= \frac{d^{2k}C_k^2}{2},
\end{align*}
for all $\mathbf{f} \in C^\infty(\overline{\B^d},\C^d)^2$ and $k  \geq 2$.
\end{proof}
Since the potentials $V_{0 i j n}, W_{0 j n}$ are explicit, one computes that 
\begin{align}\label{commutation constant}
    C = C(d) \simeq d^\frac{3}{2}.
\end{align}
With this preparation we can prove the next result which allows us to view $\mathbf{L}_0$ as a compact perturbation of a dissipative operator, at least for sufficiently large $k$.
\begin{thm}\label{full gen is diss up to per}
Let $d \geq 3$. Define 
\begin{align*}
    k_0 := k_0(d)  := \left \lceil \frac{d}{2} + 1  + C \right \rceil,
\end{align*}
where $C $ is as in \eqref{Lprime highest derivative}.
Then for each $k \geq k_0$ there exist $C_k,w_k > 0$ and a, with respect to $(\cdot|\cdot)_{\Hil^k}$ orthogonal, finite-rank projection $\mathbf{P}$ on $\Hil^k$ such that
\begin{align*}
    \Real( (\mathbf{L}_0 - C_k \mathbf{P})\mathbf{f} |\mathbf{f}  )_{\Hil^k} \leq -w_k \|\mathbf{f}\|_{\Hil^k}^2
\end{align*}
for all $ \mathbf{f} \in C^\infty(\overline{\B^d},\C^d)^2$.
\end{thm}
\begin{proof}
We have from \eqref{free homog diss} and \eqref{Lprime highest derivative} for all $\mathbf{f} \in C^\infty(\overline{\B^d},\C^d)^2$
\begin{align*}
    \Real((\mathbf{L} + \mathbf{L}_0')\mathbf{f} | \mathbf{f})_{\Dot{\Hil}^k} \leq \left( \frac{d}{2} + C - k   \right) \|\mathbf{f}\|_{\Dot{\Hil}^k}^2 + c_k \|\mathbf{f}\|_{\Hil^{k-1}}^2
\end{align*}
for some $C > 0$ independent of $k$ and $c_k > 0$. Since $k \geq k_0 \geq 3$, the operator $\mathbf{L} + \mathbf{L}_0'$ is bounded from $\Hil^{k-1}$ to $\Hil^1$. Combining this with the previous yields 
\begin{align*}
    \Real(\mathbf{L}_0 \mathbf{f }  |\mathbf{f})_{\Hil^k} \leq \left( \frac{d}{2} + C - k   \right) \|\mathbf{f}\|_{\Dot{\Hil}^k}^2 + c_k' \|\mathbf{f}\|_{\Hil^{k-1}}^2
\end{align*}
for some $c_k' > 0$. By applying \autoref{subcoer prop} to both components, we infer the existence of an orthonormal set $(\mathbf{f}_n)_{n=1}^N \subseteq \Hil^k$ and a constant $c_k'' > 0$ such that
\begin{align*}
    \|\mathbf{f}\|_{\Hil^{k-1}}^2 \leq \frac{1}{2 c_k'} \|\mathbf{f}\|_{\Dot{\Hil}^k}^2 + c_k'' \sum_{n=1}^N |(\mathbf{f} | \mathbf{f}_n)_{\Hil^k}|^2
\end{align*}
for all $\mathbf{f} \in \Hil^k$. Inserting this into the above yields
\begin{align*}
    \Real (\mathbf{L}_0 \mathbf{f} | \mathbf{f})_{\Hil^k} \leq \left( \frac{d}{2} + C + \frac{1}{2} - k  \right) \|\mathbf{f}\|_{\Dot{\Hil}^k}^2 + c_k' c_k'' \sum_{n=1}^N |(\mathbf{f} | \mathbf{f}_n)_{\Hil^k}|^2.
\end{align*}
Since $k \geq k_0$, we have in particular
\begin{align}\label{in proof bound full by homog}
    \Real (\mathbf{L}_0 \mathbf{f} | \mathbf{f})_{\Hil^k} \leq -\frac{1}{2} \|\mathbf{f}\|_{\Dot{\Hil}^k}^2 + c_k' c_k'' \sum_{n=1}^N |(\mathbf{f} | \mathbf{f}_n)_{\Hil^k}|^2.
\end{align}
Since we want a full $\Hil^k$-norm on the right-hand side of this inequality, we introduce the following space
\begin{align*}
    Y_k := \left\{ \Vector{f_1}{f_2} : f_1 \in \mathcal{P}_{k,d,d}, f_2 \in \mathcal{P}_{k-1,d,d}  \right\},
\end{align*}
where $\mathcal{P}_{k,d,d},\mathcal{P}_{k-1,d,d}$ are defined in \autoref{prop control inhomog by homog}.
Let $\mathbf{Q}$ denote the, with respect to $(\cdot|\cdot)_{\Hil^k}$ orthogonal, projection onto $Y_k$. Then using \autoref{prop control inhomog by homog} we infer the existence of some $c_k'''>0$ such that
\begin{align*}
    \|\mathbf{f}\|_{\Hil^k}^2 \leq c_k''' \|\mathbf{f}\|_{\Dot{\Hil}^k}^2
\end{align*}
for all $\mathbf{f} \in \ker(\mathbf{Q})$. For general $\mathbf{f} \in  \Hil^k$ we have
\begin{align*}
    \|\mathbf{f}\|_{\Hil^k}^2 &= \|\mathbf{Q} \mathbf{f}\|_{\Hil^k}^2 + \|\mathbf{f } - \mathbf{Q} \mathbf{f}\|_{\Hil^k}^2\leq \|\mathbf{Q} \mathbf{f}\|_{\Hil^k}^2 + c_k''' \|\mathbf{f } - \mathbf{Q} \mathbf{f}\|_{\Dot{\Hil}^k}^2 \\
    &\leq  \|\mathbf{Q} \mathbf{f}\|_{\Hil^k}^2 + 2 c_k'''\|\mathbf{f}\|_{\Dot{\Hil}^k}^2 +  2 c_k'''\|\mathbf{Q}\mathbf{f}\|_{\Dot{\Hil}^k}^2\\
    &\leq (1 + 2 c_k''')\|\mathbf{Q} \mathbf{f}\|_{\Hil^k}^2 + 2 c_k'''\|\mathbf{f}\|_{\Dot{\Hil}^k}^2,
\end{align*}
which we rewrite as
\begin{align*}
    - \|\mathbf{f}\|_{\Dot{\Hil}^k}^2 \leq - \frac{1 }{2c_k'''} \|\mathbf{f}\|_{\Hil^k} + \frac{1 + 2 c_k'''}{2c_k'''} \|\mathbf{Q} \mathbf{f}\|_{\Hil^k}^2.
\end{align*}
Inserting this into \eqref{in proof bound full by homog} yields
\begin{align*}
    \Real (\mathbf{L}_0 \mathbf{f} | \mathbf{f})_{\Hil^k} \leq -\frac{1}{4 c_k'''} \|\mathbf{f}\|_{\Hil^k}^2 + \frac{1 + 2 c_k'''}{4c_k'''} \|\mathbf{  Q} \mathbf{f}\|_{\Hil^k}^2 + c_k' c_k'' \sum_{n=1}^N |(\mathbf{f} | \mathbf{f}_n)_{\Hil^k}|^2.
\end{align*}
Now define $\mathbf{P}$ to be the, with respect to $(\cdot |\cdot)_{\Hil^k}$ orthogonal, projection onto
\begin{align*}
    \text{span}(Y_k \cup \{\mathbf{f}_n\}_{n=1}^N).
\end{align*}
By construction
\begin{align*}
    \sum_{n=1}^N |(\mathbf{f} | \mathbf{f}_n)_{\Hil^k}|^2 &\leq \|\mathbf{P} \mathbf{f}\|_{\Hil^k}^2\\
    \|\mathbf{Q} \mathbf{f}\|_{\Hil^k}^2&\leq \|\mathbf{P} \mathbf{f}\|_{\Hil^k}^2
\end{align*}
and hence
\begin{align*}
    \Real (\mathbf{L}_0 \mathbf{f} | \mathbf{f})_{\Hil^k} \leq -\frac{1}{4 c_k'''} \|\mathbf{f}\|_{\Hil^k}^2 + \left(\frac{1 + 2 c_k'''}{4c_k'''} + c_k' c_k''\right) \| \mathbf{P} \mathbf{f}\|_{\Hil^k}^2.
\end{align*}
By orthogonality $\Real (\mathbf{P} \mathbf{f} | \mathbf{f})_{\Hil^k} = \|\mathbf{P} \mathbf{f}\|_{\Hil^k}^2$ and we conclude
\begin{align*}
    \Real((\mathbf{L}_0 - C_k\mathbf{P}) \mathbf{f}| \mathbf{f})_{\Hil^k} \leq - w_k \|\mathbf{f}\|_{\Hil^k}^2
\end{align*}
for all $\mathbf{f} \in C^\infty(\overline{\B^d},\C^d)^2$, where
\begin{align*}
    w_k &= \frac{1}{4 c_k'''}\\
    C_k &= \frac{1 + 2 c_k'''}{4c_k'''} + c_k' c_k'',
\end{align*}
which is the claim.
\end{proof}
From \eqref{commutation constant} we conclude 
\begin{align}\label{k0}
    k_0 = k_0(d) \simeq d^\frac{3}{2}.
\end{align}

In order to treat $\mathbf{L}_\Theta'$ for $\Theta$ small perturbatively, we need Lipschitz-continuous dependence on $\Theta$.
\begin{lem}\label{perturb is Lipschitz}
For each $k \geq 1$ there exists some $L_k > 0$ such that
\begin{align*}
    \|\mathbf{L}_{\Theta_1}' - \mathbf{L}_{\Theta_2}'\|_{\Hil^k \to \Hil^k} \leq L_k |\Theta_1 - \Theta_2|
\end{align*}
for all $\Theta_1,\Theta_2 \in \overline{\B_M^{p(d)}}$.
\end{lem}
\begin{proof}
Recall that all of $V_{\Theta j n}(\xi),V_{\Theta i j n}(\xi), W_{\Theta j n}(\xi)$ are jointly smooth for $(\xi,\Theta) \in \overline{\B^d} \times \overline{\B_M^{p(d)}}$. The fundamental theorem of calculus yields for $F \in C^\infty\left(\overline{\B^d} \times \overline{\B_M^{p(d)}}\right)$
\begin{align*}
    F(\xi,\Theta_1) - F(\xi,\Theta_2) &= \int_{0}^1 \frac{d}{dt} F(\xi,\Theta_2 + t(\Theta_1 - \Theta_2)) d t\\
    &= \int_0^1 (\Theta_1^j - \Theta_2^j ) \pd_{\Theta^j} F(\xi,\Theta)\Big|_{\Theta = \Theta_2 + t(\Theta_1 - \Theta_2)} d t
\end{align*}
and hence
\begin{align*}
    |F(\xi,\Theta_1) - F(\xi,\Theta_2)| \leq L_F |\Theta_1 - \Theta_2|
\end{align*}
for all $\xi \in \overline{\B^d}$ and $\Theta_1,\Theta_2 \in \overline{\B_M^{p(d)}}$. Since this also holds for any $\xi$-partial derivative, the claimed Lipschitz-continuity of $\mathbf{L}_\Theta'$ on $\Hil^k$ then follows straightforwardly.
\end{proof}

\subsection{Spectrum of the unperturbed generator}
\label{subsect spec of L0}
In the following we will call $\mathbf{L}_0$ the unperturbed generator and $\mathbf{L}_\Theta$ for $\Theta \in \overline{\B_M^{p(d)}}$ the perturbed generator.

We need information on the spectrum of $\mathbf{L}_0$. For this we first have to take into consideration the effect of symmetries. Namely if one applies one of the continuous symmetries of Eq.~\eqref{int wm eq selfsim} (which are induced by the symmetries of Eq.~\eqref{ex wm eq}, see \autoref{symmetries}) to $v_0$ parametrized by a real parameter $\eps$, then differentiates the resulting expression with respect to $\eps$ and evaluates this expression at $\eps =0$, one obtains a solution of the linearization of Eq.~\eqref{int wm eq selfsim}, which is given by 
\begin{align}\label{lin int wm eq selsim}
    0 &= [\pd_\tau^2 + \pd_\tau + 2 \xi^j \pd_{\xi^j} \pd_\tau -(\delta^{i j} - \xi^i \xi^j)\pd_{\xi^i}\pd_{\xi^j} + 2 \xi^j \pd_{\xi^j}]w^n(\tau,\xi)\\
    &\notag-\frac{4}{d-2 + |\xi|^2}( |\xi|^2(\pd_\tau + \xi^j \pd_{\xi^j})w^n(\tau,\xi) - \xi^j \pd_{\xi^j} w^n(\tau,\xi)+ \xi^n \pd_j w^j(\tau,\xi) - \xi_j \pd^n w^j(\tau,\xi)  )\\
    &\notag- \frac{2(d-2 - |\xi|^2)}{d-2 + |\xi|^2}w^n(\tau,\xi), \quad n=1,\ldots,d.
\end{align}
The symmetries then generate the following solutions of this equation.
\begin{align*}
    (d - |\xi|^2) e_j, &\quad 1 \leq j \leq d\\
    \xi_i e_j - \xi_j e_i, & \quad 1 \leq i < j \leq d\\
    |\xi|^2 e_j - d \xi_j \xi, & \quad 1 \leq j \leq d\\
    e^\tau e_j, & \quad 1 \leq j \leq d\\
    e^\tau \xi &.
\end{align*}
For the first order formulation, we define the following functions
\begin{align*}
    \mathbf{f}_{0 ,j}^0(\xi) &:= \Vector{(d-|\xi|^2) e_j}{  - 2 |\xi|^2 e_j}, \quad 1 \leq j \leq d\\
    \mathbf{f}_{1,1 ,i, j}^0(\xi) &:= \Vector{\xi_i e_j - \xi_j e_i}{\xi_i e_j - \xi_j e_i}, \quad 1 \leq i < j \leq d\\
    \mathbf{f}_{2, d, j}^0(\xi) &:= \Vector{|\xi|^2 e_j - d \xi_j \xi}{2(|\xi|^2 e_j - d \xi_j \xi)},\quad 1 \leq j \leq d\\
    \mathbf{f}_{0, j}^1(\xi) &:= \Vector{e_j}{e_j}, \quad 1 \leq j \leq d\\
    \mathbf{f}_{1 ,d-1}^1(\xi) &:= \Vector{\xi}{2 \xi}.
\end{align*}
The notation is for the following. Eq.~\eqref{lin int wm eq selsim} is coupled due to the term
\begin{align*}
    \xi^n \pd_j w^j(\tau,\xi) - \xi_j \pd^n w^j(\tau,\xi) = (K_\xi w(\tau,\xi))^n.
\end{align*}
The operator $K$ is angular in the sense that it maps any radial function to 0. Hence it still makes sense to attempt the standard approach of decomposing into spherical harmonics. But then one has to find the eigenvalues of $K$ considered as an operator acting on functions in $L^2(\S^{d-1},\C^d)$. In \cite{WeiKocDon25} and \cite{DonMos26} the orthogonal decomposition
\begin{align*}
    L^2(\S^{d-1},\C^d) = \Y_0^d \oplus \bigoplus_{ \ell \geq 1} \bigoplus_{m \in \{-\ell,1,\ell+d-2\}} W_{\ell,m} 
\end{align*}
is obtained, where
\begin{align*}
    K|_{\Y_0^d} &\equiv 0\\
    K|_{W_{\ell,m}} &\equiv m.
\end{align*}
Then one chooses $\{Y_{0,a}\}$ to be an orthonormal basis of $\Y_0^d$ and $\{Y_{\ell,m,a}\}$ to be an orthonormal basis of $W_{\ell,m}$ (where $a$ runs through some finite index set). In this notation we write (dropping renormalization constants)
\begin{align*}
    \mathbf{f}_{0 ,j}^0(\xi) &= \Vector{f_{0}^0(|\xi|) Y_{0,j}\left( \frac{\xi}{|\xi|} \right)  }{|\xi| (f_0^0)'(|\xi|) Y_{0,j}\left( \frac{\xi}{|\xi|} \right)}, \quad 1 \leq j \leq d\\
    \mathbf{f}_{1, 1 ,i, j}^0(\xi) &= \Vector{f_{1, 1}^0(|\xi|)Y_{1,1,i,j}\left( \frac{\xi}{|\xi|} \right)  }{|\xi|(f_{1 ,1}^0)'(|\xi|) Y_{1,1,i,j}\left( \frac{\xi}{|\xi|} \right)},\quad 1 \leq i < j \leq d\\
    \mathbf{f}_{2,d,j}^0(\xi) &= \Vector{f_{2,d}^0(|\xi|) Y_{2,d,j}\left( \frac{\xi}{|\xi|} \right) }{|\xi| (f_{2,d}^0)'(|\xi|)Y_{2,d,j}\left( \frac{\xi}{|\xi|} \right)}, \quad 1 \leq j \leq d\\
    \mathbf{f}_{0,j}^1(\xi) &= \Vector{f_0^1(|\xi|) Y_{0,j}\left( \frac{\xi}{|\xi|} \right) }{( f_0^1(|\xi|) + |\xi|(f_0^1)'(|\xi|)) Y_{0,j}\left( \frac{\xi}{|\xi|} \right)}, \quad 1 \leq j \leq d\\
    \mathbf{f}_{1,d-1}^1(\xi) &= \Vector{f_{1,d-1}^1(|\xi|) Y_{1,d-1}\left( \frac{\xi}{|\xi|} \right) }{(f_{1,d-1}^1(|\xi|)  + |\xi| (f_{1,d-1}^1)'(|\xi|)  )Y_{1,d-1}\left( \frac{\xi}{|\xi|} \right)},
\end{align*}
where
\begin{align*}
    f_0^0(\rho) &:= d-\rho^2\\
    f_{1,1}^0(\rho) &:= \rho\\
    f_{2,d}^0(\rho) &:= \rho^2\\
    f_{0}^1(\rho) &:= 1\\
    f_{1,d-1}^1(\rho) &:= \rho.
\end{align*}
We can extract all the spectral information needed from this.
\begin{prop}\label{eigenvalues of L0}
Let $d \geq 3$ and $k \geq k_0$. Let $C_k,w_k > 0$ as in \autoref{full gen is diss up to per}. Then $\overline{\Half_{-\frac{w_k}{2}}} \cap \sigma(\mathbf{L}_0)$ is finite and $\overline{\Half} \cap \sigma(\mathbf{L}_0) = \{0,1\}$. Both $0,1$ are eigenvalues and 
\begin{align*}
    \{\mathbf{f}_{0,j}^0\}_{1 \leq j \leq d},\quad \{\mathbf{f}^0_{1,1,i,j}\}_{1 \leq i < j \leq d},\quad \{  \mathbf{f}_{2,d,j}^0\}_{1 \leq j \leq d}
\end{align*}
span $\ker(\mathbf{L}_0)$ and
\begin{align*}
    \{\mathbf{f}_{0,j}^1\}_{1 \leq j \leq d},\quad \{\mathbf{f}_{1,d-1}^1\}
\end{align*}
span $\ker(\mathbf{I} - \mathbf{L}_0)$. In particular
\begin{align*}
    \dim(\ker(\mathbf{L}_0)) = d + \binom{d}{2} + d = p(d),\quad \dim(\ker(\mathbf{I} - \mathbf{L}_0)) = d+1.
\end{align*}
\end{prop}
\begin{proof}
Let $\mathbf{P}$ be the projection from \autoref{full gen is diss up to per}. Since it is bounded, the bounded perturbation theorem yields that $\mathbf{L}_0 - C_k \mathbf{P}$ is the generator of a strongly continuous semigroup. Thus we can again apply the Lumer-Philips Theorem to conclude from \autoref{full gen is diss up to per} that $\mathbf{L}_0 - C_k \mathbf{P}$ generates a semigroup with growth bound at most $-w_k$. Since $C_k \mathbf{P}$ is bounded and has finite rank, it is compact. Hence applying \cite[Theorem B.1]{Glo22} to $\mathbf{L}_0 - C_k \mathbf{P}$ and $C_k \mathbf{P}$ implies that $ \overline{\Half_{-\frac{w_k}{2}}} \cap \sigma(\mathbf{L}_0) $ is finite and consists of eigenvalues with finite algebraic multiplicity.

Now let $\lambda \in \overline{\Half} \cap \sigma(\mathbf{L}_0)$ and $\mathbf{f} \in \D(\mathbf{L}_0)\setminus \{0\}$ with $(\lambda \mathbf{I} - \mathbf{L}_0)\mathbf{f} = 0$. One checks along the lines of \cite[Lemma 4.8]{Don24} that 
\begin{align*}
      \lambda f_1(\xi) + \Lambda f_1(\xi)  -f_2(\xi)=&0\\
      \Delta f_1(\xi) - \Lambda f_2(\xi) - f_2(\xi) + \frac{2(d-2 - |\xi|^2)}{d-2 + |\xi|^2}f_1(\xi) + \frac{4}{d- 2 + |\xi|^2}(|\xi|^2 f_2(\xi) - \Lambda f_1(\xi)  + K f_1(\xi)  )   =&0
\end{align*}
holds for $\xi  \in \B^d$ in the sense of distributions. Inserting the first equation into the second one yields
\begin{align*}
    0 &= (\delta^{i j} - \xi^i \xi^j )\pd_{\xi^i} \pd_{\xi^j} f_1^n(\xi) - 2(\lambda+ 1)\xi^j \pd_{\xi^j}f_1^n(\xi) - \lambda(\lambda+1) f_1^n(\xi) + \frac{2(d-2 - |\xi|^2)}{d-2 + |\xi|^2}f_1^n(\xi)\\
    &+ \frac{4}{d-2 + |\xi|^2}(\lambda |\xi|^2 f_1^n(\xi) - (1-|\xi|^2) \xi^j \pd_{\xi^j} f_1^n(\xi) + (\xi^n \pd_{\xi^j} - \xi_j \pd_{\xi_n}) f_1^j(\xi)   ) ,\quad n=1,\ldots,d
\end{align*}
in the distributional sense. Elliptic regularity then implies $f_1 \in C^\infty(\B^d,\C^d) \cap H^k(\B^d,\C^d)$.

Next, we make the transformation
\begin{align*}
    f(\xi) := \frac{1}{d - 2 + |\xi|^2} f_1(\xi).
\end{align*}
Then $f \in C^\infty(\B^d,\C^d) \cap H^k(\B^d,\C^d)$ solves
\begin{align}\label{full transformed spectral eq}
    0 &=\notag (\delta^{i j} - \xi^i \xi^j)\pd_{\xi^i} \pd_{\xi^j} f^n(\xi) - 2 (\lambda  + 1)\xi^j \pd_{\xi^j} f^n(\xi) + \left( -\lambda(\lambda + 1) + \frac{4(d-1)(d-2  - |\xi|^2)}{(d-2 +  |\xi|^2)^2}  \right)f^n(\xi)\\
    &+ \frac{4}{d-2 + |\xi|^2}(\xi^n \pd_{\xi^j} - \xi_j \pd_{\xi_n})f^j(\xi), \quad n=1,\ldots,d.
\end{align}
Next, we define $f_{0,a},f_{\ell,m,a},:[0,1) \to \C$ by
\begin{align*}
    f_{0,a}(\rho) &:= \int_{\S^{d-1}} f^j(\rho \omega) \overline{(Y_{0,a})_j(\omega)} d\sigma(\omega) = (f(\rho\cdot)|Y_{0,a})_{L^2(\S^{d-1},\C^d)}\\
    f_{\ell,m,a}(\rho) &:= \int_{\S^{d-1}} f^j(\rho \omega) \overline{(Y_{\ell,m,a})_j(\omega)} d\sigma(\omega) = (f(\rho\cdot)|Y_{\ell,m,a})_{L^2(\S^{d-1},\C^d)},
\end{align*}
which gives the decomposition
\begin{align}\label{sph harm exp}
    f(\rho\cdot ) = \sum_a f_{0,a}(\rho) Y_{0,a} + \sum_{\ell \geq 1} \sum_{m \in \{-\ell,1,\ell+d-2\}} \sum_a f_{\ell,m,a}(\rho) Y_{\ell,m,a},
\end{align}
which converges in $L^2(\S^{d-1},\C^d)$ for each $\rho \in [0,1)$. Using dominated convergence, one verifies that $f_{0,a},f_{\ell,m,a} \in C^\infty([0,1)) \cap H^k((0,1))$. Smoothness in the interior allows one to insert the expansion \eqref{sph harm exp} into Eq.~\eqref{full transformed spectral eq}, then test with $Y_{0,a},Y_{\ell,m,a}$, respectively, to obtain the equations
\begin{align}\label{mode ode zero}
    (1 - \rho^2)f_{0,a}''(\rho) + \left( \frac{d-1}{\rho} - 2 (\lambda + 1)\rho \right)f_{0,a}'(\rho) + \left(-\lambda(\lambda + 1)  + \frac{4(d-1)(d-2 - \rho^2)}{(d-2  + \rho^2)^2} \right)f_{0,a}(\rho)  &= 0
\end{align}
and
\begin{align}\label{mode ode l m}
    (1 - \rho^2)f_{\ell,m,a}''(\rho) + \left( \frac{d-1}{\rho}- 2 (\lambda + 1)\rho  \right)f_{\ell,m,a}'(\rho)&\\
    +\left( -\lambda(\lambda + 1) -\frac{\ell(\ell+d-2)}{\rho^2} + \frac{4(d-1)(d-2 - \rho^2)}{(d-2 + \rho^2)^2} + \frac{4m}{d-2 + \rho^2}     \right) f_{\ell,m,a}(\rho)& = 0.\notag
\end{align}
We claim that we even have $f_{0,a},f_{\ell,m,a} \in C^\infty([0,1])$. To see this, one computes that both equations have the Frobenius indices $\{0, \frac{d-1}{2} - \lambda\}$ at $\rho=1$.

Assume first that $\frac{d-1}{2} - \lambda \notin \N_0$. Then a function that behaves like $(1- \rho)^{\frac{d-1}{2} - \lambda}$ around $\rho=1$ does not belong to $H^{\lceil \frac{d}{2} \rceil } \left(( \frac{1}{2},1)  \right)  $ and since $k \geq k_0 \geq \lceil \frac{d}{2} \rceil$, it in particular does not belong to $H^{k} \left( (\frac{1}{2},1 ) \right)  $. Hence Frobenius analysis yields that the functions $f_{0,a},f_{\ell,m,a}$ must correspond to the index $0$ which implies, since $\frac{d-1}{2} - \lambda \notin \N_0$, that they must be analytic and in particular smooth around $\rho=1$.

In the case where $\frac{d-1}{2} - \lambda \in \N_0$ one has potentially two situations.
\begin{enumerate}
    \item The fundamental system of solutions of Eq.~\eqref{mode ode zero}, respectively Eq.~\eqref{mode ode l m}, consists of two functions that are both analytic around $\rho = 1$.
    \item Let $\phi_1(\rho) = (1 - \rho)^{\frac{d-1}{2} - \lambda} h_1(\rho)$ for a function $h_1$ analytic around $\rho=1$ with $h_1(1) \not=0$ be a solution of Eq.~\eqref{mode ode zero}, respectively Eq.~\eqref{mode ode l m}, then any solution $\phi_2$ that is not a multiple of $\phi_1$ can be written as
    \begin{align*}
        \phi_2(\rho) = h_2(\rho) + c(1 - \rho)^{\frac{d-1}{2} - \lambda} \log(1-  \rho) h_1(\rho),
    \end{align*}
    for some $c \not=0$ and where $h_2$ is analytic around $\rho=1$.
\end{enumerate}
In the first case we again have $f_{0,a},f_{\ell,m,a} \in C^\infty([0,1])$. In the second case one computes that $\phi_2 \notin H^{\lceil \frac{d+1}{2} \rceil}\left( (\frac{1}{2},1) \right)$ and since $k \geq k_0 \geq \lceil \frac{d+1}{2} \rceil$, we conclude that $f_{0,a}$, respectively $f_{\ell,m,a}$, must be a multiple of $\phi_1$, which is smooth around $\rho=1$.

By mode stability, see \cite{WeiKocDon25} and \cite{DonMos26}, Eqs.~\eqref{mode ode zero} and \eqref{mode ode l m} have no nontrivial solutions in $C^\infty([0,1])$ for $\lambda \in \overline{\Half} \setminus\{0,1\}$. For $\lambda \in \overline{\Half} \setminus \{0,1\}$ hence all $f_{0,a},f_{\ell,m,a}$ must be zero and hence $f = 0$. This means $\mathbf{f} = 0$, a contradiction to the assumption.

Hence we have $\overline{\Half} \cap  \sigma(\mathbf{L}_0) \subseteq \{0,1\} $. One computes that $\{\mathbf{f}_{0,j}^0\}_{1 \leq j \leq d}$, $\{\mathbf{f}_{1,1,i,j}^0\}_{1 \leq i < j \leq d}$ and $\{\mathbf{f}_{2,d,j}^0\}_{1 \leq j \leq d}$ belong to $\ker(\mathbf{L}_0)$ and $\{ \mathbf{f}_{0,j}^1\}_{1 \leq j \leq d}$ and $\mathbf{f}_{1,d-1}^1$ belong to $\ker(\mathbf{I} - \mathbf{L}_0)$ and hence $\overline{\Half} \cap  \sigma(\mathbf{L}_0) = \{0,1\}$. Clearly these functions form linearly independent sets, hence we have $\dim(\ker(\mathbf{L}_0)) \geq p(d)$ and $\dim(\ker(\mathbf{I} - \mathbf{L}_0)) \geq d+1$. The fact that these functions exhaust these spaces is also a consequence of mode stability, which finishes the proof.
\end{proof}
Since we are in a nonself-adjoint setting, we still have to show that the algebraic multiplicities equal the geometric multiplicities, which is the content of the next lemma.
\begin{lem}
Under the assumptions of \autoref{eigenvalues of L0} let 
\begin{align*}
    \mathbf{P}_{\lambda,0} := \frac{1}{2 \pi i} \int_{\pd \disk_r(\lambda)} \mathbf{R}_{\mathbf{L}_0}(z) dz
\end{align*}
for $\lambda \in \{0,1\}$, where $r = \frac{1}{2}\dist(0,\sigma(\mathbf{L}_0 ) \setminus \{0\})$, be the Riesz projection. Then $\mathbf{P}_{\lambda,0}$ has finite-dimensional range and
\begin{align*}
    \rg(\mathbf{P}_{\lambda,0}) = \ker(\lambda \mathbf{I} - \mathbf{L}_0),
\end{align*}
i.e., the algebraic multiplicity is finite and equals the geometric multiplicity.
\end{lem}
\begin{proof}
The results from spectral theory we use here can be found in \cite{Kat95}. First of all note that $r > 0$, since $ \overline{\Half_{- \frac{w_k}{2}}}  \cap \sigma(\mathbf{L}_0)$ is finite. Hence by construction we have $\pd \disk_r(0),\pd \disk_r(1) \subseteq \rho(\mathbf{L}_0)$, so $\mathbf{P}_{0,0},\mathbf{P}_{1,0}$ are well-defined bounded projections. We already observed in the proof of \autoref{eigenvalues of L0} that \cite[Theorem B.1]{Glo22} implies that $0,1$ have finite algebraic multiplicity, i.e., $\mathbf{P}_{0,0},\mathbf{P}_{1,0}$ have finite rank.

We first claim that $\rg(\mathbf{P}_{\lambda , 0}) \subseteq \D(\mathbf{L}_0)$ for $\lambda \in \{0,1\}$. To see this, let $\mathbf{f} \in \rg(\mathbf{P}_{\lambda,0})$. By density of $\D(\mathbf{L}_0)$ in $\Hil^k$ there exists a sequence $(\mathbf{g}_n)_{n \geq 1} \subseteq \D(\mathbf{L}_0)$ with $\mathbf{g}_n \to \mathbf{f}$. Then \cite[p.~178, Theorem 6.17]{Kat95} implies $\mathbf{P}_{\lambda, 0} \D(\mathbf{L}_0) \subseteq \D(\mathbf{L}_0)$. So by setting $\mathbf{f}_n := \mathbf{P}_{\lambda,0} \mathbf{g}_n$ we have $(\mathbf{f}_n)_{n \geq 1} \subseteq \D(\mathbf{L}_0) \cap \rg(\mathbf{P}_{\lambda,0})$ and since $\mathbf{P}_{\lambda,0}$ is bounded, we have $\mathbf{f}_n \to \mathbf{P}_{\lambda,0} \mathbf{f} = \mathbf{f}$. Since $\mathbf{L}_0 |_{\D(\mathbf{L}_0) \cap \rg(\mathbf{P}_{\lambda,0})}$ is bounded, $\mathbf{L}_0 \mathbf{f}_n \to \Tilde{\mathbf{f}}$ for some $\Tilde{\mathbf{f}} \in \rg(\mathbf{P}_{\lambda,0})$. Hence closedness of $\mathbf{L}_0$ implies $\mathbf{f} \in \D(\mathbf{L}_0)$ and hence $\rg(\mathbf{P}_{\lambda,0}) \subseteq \D(\mathbf{L}_0)$.

Since one always has $\ker(\lambda \mathbf{I} - \mathbf{L}_0) \subseteq \rg(\mathbf{P}_{\lambda,0})$, we aim to prove the other inclusion. Since the spectrum of $ (\mathbf{L}_0)_{\rg(\mathbf{P}_{\lambda,0})}=\mathbf{L}_0|_{\rg(\mathbf{P}_{\lambda,0})}$ is exactly $\{\lambda\}$, the operator $(\lambda \mathbf{I} - \mathbf{L}_0)|_{\rg(\mathbf{P}_{\lambda,0})}$ must be nilpotent. If $(\lambda \mathbf{I} - \mathbf{L}_0)|_{\rg (\mathbf{P}_{\lambda,0})} = 0$, then we have $\rg(\mathbf{P}_{\lambda,0}) \subseteq \ker(\lambda \mathbf{I} - \mathbf{L}_0)$ and we are done.

So assume $(\lambda \mathbf{I} - \mathbf{L}_0)|_{\rg (\mathbf{P}_{\lambda,0})} \not= 0$. From nilpotency we conclude $\rg((\lambda \mathbf{I} - \mathbf{L}_0)|_{\rg (\mathbf{P}_{\lambda,0})}) \cap \ker((\lambda \mathbf{I} - \mathbf{L}_0)|_{\rg (\mathbf{P}_{\lambda,0})}) \not=0$ and hence there must exist $0 \not= \mathbf{g} \in \ker((\lambda \mathbf{I} - \mathbf{L}_0)|_{\rg (\mathbf{P}_{\lambda,0})}) \subseteq \ker(\lambda\mathbf{I} - \mathbf{L}_0 ) $ and some $\mathbf{f} \in \rg(\mathbf{P}_{\lambda , 0})$ such that
\begin{align*}
    (\lambda \mathbf{I} - \mathbf{L}_0)\mathbf{f} = \mathbf{g}.
\end{align*}
As in the proof of \autoref{eigenvalues of L0} the equations
\begin{align*}
    \lambda f_1(\xi) + \Lambda f_1(\xi) - f_2(\xi) &= g_1(\xi)\\
    -\Delta f_1(\xi) + \Lambda f_2(\xi) + (\lambda + 1)f_2(\xi) - \frac{2(d-2 - |\xi|^2)}{d-2 + |\xi|^2} f_1(\xi) &\\
    - \frac{4}{d-2 + |\xi|^2}(|\xi|^2 f_2(\xi) - \Lambda f_1(\xi) + K f_1(\xi) ) &= g_2(\xi) 
\end{align*}
hold for $\xi \in \B^d$ in the sense of distributions. Inserting the first equation in the second one yields
\begin{align*}
    (\delta^{i j} - \xi^i \xi^j) \pd_i \pd_j f_1(\xi) + \left[ - 2(\lambda + 1) - \frac{4(1 - |\xi|^2)}{d-2 + |\xi|^2} \right]\Lambda f_1(\xi)\\
    +\left[ -\lambda(\lambda + 1) + \frac{2(d-2 - |\xi|^2)}{d-2 + |\xi|^2} + \frac{4 \lambda |\xi|^2}{d-2 + |\xi|^2} \right]f_1(\xi) + \frac{4}{d-2+  |\xi|^2} Kf_1(\xi) &= G(\xi),
\end{align*}
where $G(\xi) = \frac{4 |\xi|^2}{d-2 + |\xi|^2} g_1(\xi) - (\lambda + 1)g_1(\xi) - \Lambda g_1(\xi) - g_2(\xi)$. Since $\mathbf{g} \in \ker(\lambda \mathbf{I} - \mathbf{L}_0)$, it is a linear combination of functions in $C^\infty(\overline{\B^d},\C^d)^2$ and hence also belongs to $C^\infty(\overline{\B^d},\C^d)^2$. Hence $G \in C^\infty(\overline{\B^d},\C^d)$ and with elliptic regularity we conclude that $f_1 \in C^\infty(\B^d,\C^d) \cap H^k(\B^d,\C^d)$ and the above equation is satisfied classically on $\B^d$.

Now we again employ a spherical harmonics decomposition as in \eqref{sph harm exp} for both $f_1$ and $G$ to obtain the equations
\begin{align*}
    (1 - \rho^2)f_0''(\rho) + \left( \frac{d-1}{\rho} - 2 (\lambda + 1)\rho - \frac{4 \rho(1 - \rho^2)}{d-2 + \rho^2} \right) f_0'(\rho)&\\
    +\left[-\lambda(\lambda + 1) + \frac{2(d-2 - \rho^2)}{d-2 + \rho^2} +  \frac{4 \lambda \rho^2}{d-2 + \rho^2}  \right]f_0(\rho) &= G_0(\rho)
\end{align*}
and
\begin{align}\label{inhomog lm eq}
    (1 -\rho^2)f_{\ell,m}''(\rho) + \left( \frac{d-1 }{\rho} - 2(\lambda + 1)\rho - \frac{4\rho(1 - \rho^2)}{d-2 + \rho^2} \right) f_{\ell,m}'(\rho)&\\
    +\left[ -\frac{\ell(\ell+d-2)}{\rho^2} - \lambda(\lambda +1) + \frac{2(d-2 - \rho^2)}{d-2 + \rho^2} + \frac{4 \lambda \rho^2}{d-2 + \rho^2} + \frac{4 m}{d-2 + \rho^2} \right]f_{\ell,m}(\rho) &= G_{\ell,m}(\rho),\notag
\end{align}
where the additional index $a$ has been omitted. One checks $G_{\ell,m} = 0$ except for $\lambda = 0,1$, $ \ell=0$ or $(\lambda,\ell,m) \in \{(0,1,1),(1,1,d-1),(0,2,d)\}$. Except for these special cases we conclude $f_{\ell,m} = 0$ similar as in the proof of \autoref{eigenvalues of L0}.

We know that $f_0,f_{\ell,m} \in C^\infty([0,1))\cap H^k((0,1))$ and that $f_0, f_{\ell,m}$ satisfy the equations classically. We claim that in all of these special cases whenever $G_{\ell,m}\not=0$ and $f_{\ell,m} \not=0$ then $f_{\ell,m} \notin C^\infty([0,1))\cap H^k((0,1))$ follows, a contradiction, similarly for $f_0$. Hence also in these cases we must have $f_{\ell,m}=0$ which in total yields $f_1 = 0$ and hence also $\mathbf{f} = 0$, a contradiction to the fact that $(\lambda \mathbf{I} - \mathbf{L}_0)\mathbf{f}=\mathbf{g} \not=0$.

Since the equations decouple as above, it suffices to go through each of the special cases separately. As in \cite{ChaDonGlo17,Glo25}, the idea is to use mode stability here again. We will carry out the argument for $(\lambda,\ell,m) = (1,1,d-1) $ and for $\ell = 0$, the two remaining cases are handled similarly as the case $(\lambda,\ell,m) = (1,1,d-1)$.

\begin{itemize}
    \item $(\lambda,\ell,m) = (1,1,d-1)$: We define
    \begin{align*}
        f(\rho) := \frac{1}{d-2 + \rho^2}f_{1,d-1}(\rho)
    \end{align*}
    which solves 
    \begin{align}\label{inhom corot eq}
        f''(\rho) + \frac{d-1 - 4 \rho^2}{\rho(1 - \rho^2)} f'(\rho) +\left(-\frac{2}{1 - \rho^2} + V(\rho)\right)f(\rho) = \frac{g(\rho)}{1-\rho^2},
    \end{align}
    where (cf.~Eq.~\eqref{mode ode l m})
    \begin{align*}
        V(\rho) &:= (1-\rho^2)^{-1}\left( -\frac{d-1}{\rho^2} + \frac{4(d-1)(d-2-\rho^2)}{(d-2 + \rho^2)^2} + \frac{4(d-1)}{d-2 + \rho^2} \right)\\
        &= \frac{-(d-1)((d-2)^2 - 6(d-2)\rho^2 + \rho^4)}{\rho^2(1 - \rho^2)(d-2 + \rho^2)^2}
    \end{align*}
    and
    \begin{align*}
        g(\rho) := \frac{G_{1,d-1}(\rho)}{d-2 + \rho^2} = -\frac{\rho(\rho^2 + 5(d-2))}{(d- 2 + \rho^2)^2}.
    \end{align*}
    The homogeneous version of Eq.~\eqref{inhom corot eq} has the solution
    \begin{align*}
        h_0(\rho) := \frac{1}{d-2 + \rho^2} f_{1,d-1}^1(\rho) = \frac{\rho}{d-2 + \rho^2}.
    \end{align*}
    One computes that the Wronskian of Eq.~\eqref{inhom corot eq} is
    \begin{align*}
        W(\rho) = \rho^{-(d-1)}(1 - \rho^2)^\frac{d-5}{2}.
    \end{align*}
    Hence one obtains a second solution of the homogeneous version of Eq.~\eqref{inhom corot eq} by defining
    \begin{align}\label{h1 def corot}
        h_1(\rho) := h_0(\rho) \int_{\rho_1}^\rho W(s) h_0(s)^{-2} ds = h_0(\rho) \int_{\rho_1}^\rho s^{-(d+1)} (1-s^2)^\frac{d-5}{2} (d-2+s^2)^2 ds
    \end{align}
    for some $\rho_1 \in (0,1]$.

    We first consider the case $d = 3$. Here we set $\rho_1 = \frac{1}{2}$. In this case one just explicitly computes
    \begin{align*}
        h_1(\rho) &= h_0(\rho) \left( -\frac{1}{3\rho^3} -\frac{3}{\rho} - 2 \log(1 -\rho) + 2 \log(1+ \rho) + C \right)\\
        &= (1+\rho^2)^{-1}\left( -\frac{1}{3\rho^2} - 3 - 2 \rho \log(1 - \rho) + 2\rho \log(1 + \rho) + C\rho \right)
    \end{align*}
    for some constant $C$. Now we can employ the variation of constants formula to see that we have
    \begin{align*}
        f(\rho) = c_0 h_0(\rho) + c_1 h_1(\rho) - h_0(\rho) \int_0^\rho \frac{h_1(s)}{W(s)} \frac{g(s)}{1-s^2} ds + h_1(\rho) \int_0^\rho \frac{h_0(s)}{W(s)} \frac{g(s)}{1-s^2} ds
    \end{align*}
    for some constants $c_0,c_1$. Since $\frac{h_1(s)}{W(s)} \frac{g(s)}{1-s^2}$ is integrable close to $s = 0$ and since $\frac{h_0(s)}{W(s)} \frac{g(s)}{1-s^2}$ behaves as $s^4$ as $s \searrow 0$, we see that the last two terms in the above formula are bounded close to $\rho = 0$. Clearly the first term also stays bounded, hence the only unbounded term would be $c_1 h_1(\rho)$ if $c_1 \not=0$. But since we know that $f \in H^k((0,1)) \subseteq L^\infty((0,1))$ is bounded, we must have $c_1 = 0$. Hence we have
    \begin{align*}
        f(\rho) = c_0 h_0(\rho) - h_0(\rho) \int_0^\rho \frac{h_1(s)}{W(s)} \frac{g(s)}{1-s^2} ds + h_1(\rho) \int_0^\rho \frac{h_0(s)}{W(s)} \frac{g(s)}{1-s^2} ds.
    \end{align*}
    We observe that $\frac{h_1(s)}{W(s)} \frac{g(s)}{1-s^2}$ is integrable close to $s = 1$ as well and hence the first integral term above is bounded also close to $\rho = 1$. Hence the only potentially unbounded term is the second integral term. Since $h_1(\rho)$ is unbounded close to $\rho =1$, the last term can only be bounded, if $\int_0^\rho \frac{h_0(s)}{W(s)} \frac{g(s)}{1-s^2} ds$ vanishes at $\rho = 1$, i.e.,
    \begin{align*}
        \int_0^1 \frac{h_0(s)}{W(s)} \frac{g(s)}{1-s^2} ds = 0.
    \end{align*}
    But this is not the case, as one confirms by direct computation or by observing that the integrand is strictly negative on $(0,1)$. Hence $f$ is unbounded and in particular $f \notin H^k((0,1))$.

    For the general case $d \geq 4$ we set $\rho_1 = 1$ and we have the expansions
    \begin{align*}
        h_1(\rho) = \frac{A_1(\rho)}{\rho^{d-1}} + C\rho \log \rho A_2(\rho) = (1 - \rho^2)^\frac{d-3}{2} B(\rho),
    \end{align*}
    where $A_j$ are analytic at $\rho = 0$ with $A_j(0) \not=0$, $B$ is analytic at $\rho = 1$ with $B(1) \not=0$ and $C$ is a constant. Note that $C = 0$ for odd $d$, since the integrand in \eqref{h1 def corot} is even and hence there cannot be a logarithmic term. For even $d$ however, we actually have $C \not=0$ and hence encounter a logarithmic term. This was overlooked in \cite{Glo25}, but a small adaptation of the argument goes through nonetheless, as we show below.

    We again have from variation of constants
    \begin{align*}
        f(\rho) = c_0 h_0(\rho) + c_1 h_1(\rho) - h_0(\rho) \int_0^\rho \frac{h_1(s)}{W(s)} \frac{g(s)}{1-s^2} ds + h_1(\rho) \int_0^\rho \frac{h_0(s)}{W(s)} \frac{g(s)}{1-s^2} ds.
    \end{align*}
    Boundedness of $f$ at $\rho = 0$ again yields $c_1 = 0$. Now observe that
    \begin{align*}
        \frac{h_1(s)}{W(s)} \frac{g(s)}{1-s^2} = B(s) s^{d-1} g(s)
    \end{align*}
    is evidently analytic at $s = 1$. Hence $h_0(\rho)\int_0^\rho \frac{h_1(s)}{W(s)} \frac{g(s)}{1-s^2} ds$ is analytic at $\rho = 1$. One computes
    \begin{align*}
        \int_0^\rho \frac{h_0(s)}{W(s)} \frac{g(s)}{1-s^2} ds = \begin{cases}
            (1-\rho^2)^{-\frac{d-5}{2}} B_0(\rho) + c \log(1 - \rho^2), &d\text{ is odd}\\
            (1 - \rho^2)^{-\frac{d-5}{2}} B_0(\rho) + c, &d \text{ is even}
        \end{cases},
    \end{align*}
    where $B_0$ is analytic at $\rho = 1$ with $B_0(1) \not=0$ and $c$ is a constant. We hence have
    \begin{align*}
        h_1(\rho)\int_0^\rho \frac{h_0(s)}{W(s)} \frac{g(s)}{1-s^2} ds = \begin{cases}
            (1-\rho^2) B_0(\rho)B(\rho) + c (1-\rho^2)^\frac{d-3}{2} \log(1 - \rho^2)B(\rho), &d\text{ is odd}\\
            (1 - \rho^2) B_0(\rho)B(\rho) + c (1-\rho^2)^\frac{d-3}{2} B(\rho) , &d \text{ is even}
        \end{cases}.
    \end{align*}
    In both cases this yields an analytic function at $\rho = 1$ if and only if $c = 0$. We claim that this function is in fact not analytic at $\rho = 1$ and hence $c \not=0$. Since this term dictates the behavior of $f$ at $\rho = 1$, from this we will see that the $\left\lceil \frac{d-2}{2}\right\rceil$-th derivative of $f$ does not belong to $L^2((\frac{1}{2},1))$ and hence $f \notin H^k((0,1))$.

    To see this, we first note the relation $g(s) = -(3 h_0(s) - 2 s h_0'(s))$ which allows us to perform the following integration by parts
    \begin{align*}
        \int_0^\rho \frac{h_0(s)}{W(s)} \frac{g(s)}{1-s^2} ds &= -\int_0^\rho s^{d-1}(1-s^2)^{-\frac{d-3}{2}} (3 h_0(s)^2 + 2 s h_0(s)h_0'(s)) ds\\
        &= -\int_0^\rho s^{d-3}(1-s^2)^{-\frac{d-3}{2}} (3s^2 h_0(s)^2 + s^3  2 h_0(s)h_0'(s)) ds\\
        &= -\int_0^\rho s^{d-3}(1-s^2)^{-\frac{d-3}{2}} \pd_s[s^3 h_0(s)^2] ds\\
        &= - \rho^d (1-\rho^2)^{-\frac{d-3}{2}} h_0(\rho)^2 + (d-3)\int_0^\rho s^{d-1} (1-s^2)^{-\frac{d-1}{2}} h_0(s)^2 ds.
    \end{align*}
    From this we have
    \begin{align*}
        h_1(\rho) \int_0^\rho \frac{h_0(s)}{W(s)} \frac{g(s)}{1-s^2} ds &= - \rho^d h_0(\rho)^2 B(\rho)\\
        &+ (d-3) B(\rho) (1-\rho^2)^\frac{d-3}{2}\int_0^\rho s^{d-1} (1-s^2)^{-\frac{d-1}{2}} h_0(s)^2 ds,
    \end{align*}
    where the first term is analytic at $\rho = 1$. Hence we focus on the last term. We now define the function
    \begin{align*}
        \Tilde{h}_0(\rho) = \rho^{-(d-1)}(1-\rho^2)^\frac{d-3}{2} h_0(\rho)^{-1}\int_0^\rho s^{d-1} (1-s^2)^{-\frac{d-1}{2}} h_0(s)^2 ds.
    \end{align*}
    This function satisfies the equation
    \begin{align*}
        \Tilde{h}_0''(\rho) + \frac{d-1 - 4 \rho^2}{\rho(1 - \rho^2)}\Tilde{h}_0'(\rho) + \left(-\frac{2}{1-\rho^2} - \frac{2(d-2)(d-\rho^2)}{\rho^2( 1 - \rho^2)(d-2 + \rho^2)}\right)\Tilde{h}_0(\rho) = 0,
    \end{align*}
    see \cite{ChaDonGlo17} for a derivation of this fact. But this equation is exactly the supersymmetric analogue of the homogeneous version of Eq.~\eqref{inhom corot eq} (see Eq.~(2.10) in \cite{DonMos26}). Since $\Tilde{h}_0$ is evidently analytic at $\rho = 0$, by mode stability it cannot be also analytic at $\rho = 1$. Finally we can write
    \begin{align*}
        (d-3) B(\rho) (1-\rho^2)^\frac{d-3}{2}\int_0^\rho s^{d-1} (1-s^2)^{-\frac{d-1}{2}} h_0(s)^2 ds = (d-3)B(\rho)\rho^{d-1}h_0(\rho) \Tilde{h}_0(\rho)
    \end{align*}
    and since $(d-3) B(\rho)\rho^{d-1}h_0(\rho)  $ is analytic at $\rho = 1$ and does not vanish there, the function $(d-3)B(\rho)\rho^{d-1}h_0(\rho) \Tilde{h}_0(\rho)$ is not analytic at $\rho = 1$, which is what we wanted to prove.
    \item $(\lambda,\ell) = (0,0)$: We follow the same strategy as in the previous case. We set
    \begin{align*}
        f(\rho) := \frac{1}{d-2 + \rho^2} f_0(\rho)
    \end{align*}
    which solves
    \begin{align}\label{inhom lam0 l0 eq}
        f''(\rho) + \frac{d-1- 2 \rho^2}{\rho(1 - \rho^2)}f'(\rho) + V(\rho) f(\rho) = \frac{g(\rho)}{1 - \rho^2},
    \end{align}
    where (cf.~Eq.~\eqref{mode ode zero})
    \begin{align*}
        V(\rho) := \frac{4(d-1)(d-2 - \rho^2)}{(1 - \rho^2)(d-2 + \rho^2)^2}
    \end{align*}
    and 
    \begin{align*}
        g(\rho) := \frac{G_0(\rho)}{d-2+  \rho^2} = \frac{-d(d-2) + (8d - 10)\rho^2 + \rho^4}{(d-2 + \rho^2)^2}.
    \end{align*}
    The homogeneous version of Eq.~\eqref{inhom lam0 l0 eq} has the solution 
    \begin{align*}
        h_0(\rho) := \frac{1}{d-2 + \rho^2} f_0^0(\rho) = \frac{d- \rho^2}{d-2 + \rho^2}.
    \end{align*}
    The Wronskian of Eq.~\eqref{inhom lam0 l0 eq} is
    \begin{align*}
        W(\rho) = \rho^{-(d-1)} (1 - \rho^2)^{\frac{d-3}{2}}
    \end{align*}
    and hence we find the second solution
    \begin{align*}
        h_1(\rho) = h_0(\rho)\int_1^\rho W(s) h_0(s)^{-2} ds = h_0(\rho) \int_1^\rho s^{-(d-1)}(1 -s^2)^\frac{d-3}{2} (d-s^2)^{-2} (d-2 + s^2)^2 ds,
    \end{align*}
    which we integrate in this case from $1$ for all $d \geq 3$. We can again write down the expansions
    \begin{align*}
        h_1(\rho) = \frac{A_1(\rho)}{\rho^{d-2}} + C \log \rho A_2(\rho) = (1-\rho^2)^\frac{d-1}{2} B(\rho).
    \end{align*}
    Variation of constants yields
    \begin{align*}
        f(\rho) = c_0 h_0(\rho) + c_1 h_1(\rho) - h_0(\rho) \int_0^\rho \frac{h_1(s)}{W(s)} \frac{g(s)}{1-s^2} ds + h_1(\rho) \int_0^\rho \frac{h_0(s)}{W(s)} \frac{g(s)}{1-s^2} ds 
    \end{align*}
    for constants $c_0,c_1$. Boundedness of $f$ implies that $c_1 = 0$. Again the first integral term is analytic at $\rho = 1$ and we can write
    \begin{align*}
        h_1(\rho) \int_0^\rho \frac{h_0(s)}{W(s)} \frac{g(s)}{1-s^2} ds = \begin{cases}
            (1-\rho^2) B_1(\rho) B(\rho) + c (1  - \rho^2)^\frac{d-1}{2}\log(1 - \rho^2)B(\rho),&d\not=9\text{ is odd}\\
            (1-\rho^2) B_1(\rho) B(\rho) + c (1  - \rho^2)^\frac{d-1}{2}B(\rho),&d\text{ is even}\\
            (1-\rho^2)^2 B_1(\rho) B(\rho) + c (1-\rho^2)^4 \log(1 - \rho^2)B(\rho), &d=9
        \end{cases},
    \end{align*}
    where $B$ is analytic and non-vanishing at $\rho = 1$ and $c$ is a constant. For $d = 9$ one has $g(1)=0$, which slightly changes the expansion, but crucially the logarithmic term is still potentially present. We will prove that this function is not analytic at $\rho = 1$, which implies $c \not= 0$ and from this we conclude again $f \notin H^k((0,1))$. 

    We write $g(s) = -(h_0(s) + 2s h_0'(s))$ which again allows an integration by parts to yield
    \begin{align*}
        h_1(\rho)\int_0^\rho \frac{h_0(s)}{W(s)} \frac{g(s)}{1-s^2} ds &= - \rho^d B(\rho) h_0(\rho)^2 \\
        &+ (d-1) B(\rho)(1 - \rho^2)^\frac{d-1}{2} \int_0^\rho s^{d-1} (1-s^2)^{-\frac{d+1}{2}} h_0(s)^2ds.
    \end{align*}
    Similar to before it suffices to prove that the function defined by 
    \begin{align*}
        \Tilde{h}_0(\rho) = \rho^{-(d-1) } (1-\rho^2)^\frac{d-1}{2} h_0(\rho)^{-1}\int_0^\rho s^{d-1}(1-s^2)^{-\frac{d+1}{2}} h_0(s)^2 ds
    \end{align*}
    is not analytic at $\rho = 1$. This function satisfies the equation
    \begin{align*}
        \Tilde{h}_0''(\rho) + \frac{d-1 - 2 \rho^2}{\rho(1 - \rho^2)}\Tilde{h}_0'(\rho) -\frac{(d-1)(d^2 - 6d\rho^2 - 3 \rho^4)}{\rho^2(1 - \rho^2)(d- \rho^2)^2}\Tilde{h}_0(\rho) = 0.
    \end{align*}
    This is again the supersymmetric analogue of the homogeneous version of Eq.~\eqref{inhom lam0 l0 eq}. However, we cannot use mode stability just yet, since in the case $\ell = 0$ there also exists a mode solution for $\lambda = 1$ that has to be removed as well. To get to the final equation, we have to set
    \begin{align}\label{h hat from h tilde}
        \hat{h}_0(\rho) := \rho^{-\frac{d-1}{2}} (1 - \rho^2)^\frac{d+1}{4} (\pd_\rho - w(\rho)) \rho^\frac{d-1}{2} (1 - \rho^2)^{-\frac{d-3}{4}} \Tilde{h}_0(\rho),
    \end{align}
    where
    \begin{align*}
        w(\rho) = \frac{d(d+1) + (3-7d)\rho^2 + 2 \rho^4}{2\rho(1-\rho^2)(d - \rho^2)}.
    \end{align*}
    This function will now finally satisfy the ``right'' equation, namely
    \begin{align*}
        \hat{h}_0''(\rho) + \frac{d-1 - 2 \rho^2}{\rho(1 - \rho^2)}\hat{h}_0'(\rho) - \frac{2d}{\rho^2(1 - \rho^2)}\hat{h}_0(\rho) = 0,
    \end{align*}
    see Eq.~(2.10) in \cite{DonMos26}. Mode stability implies that this equation does not have a nontrivial solution that is analytic at both $\rho = 0$ and $\rho = 1$. One checks that indeed $\hat{h}_0$ is nontrivial and analytic at $\rho = 0$. Furthermore, one can compute that the transformation \eqref{h hat from h tilde} preserves analyticity at $\rho = 1$. Hence if $\Tilde{h}_0$ were analytic at $\rho = 1$, then also $\hat{h}_0$ would be analytic at $\rho = 1$ which in the end would contradict mode stability.
    \item $(\lambda,\ell) = (1,0)$: In this case, one has to treat $d = 3$ as a special case analogous to $(\lambda,\ell,m) = (1,1,d-1)$.
    
    For $d \geq 4$ one would like to reproduce the steps from before. This requires to again reduce the problem to the analyticity of some integral term, connect this term to a solution of a supersymmetric problem and then connect that solution in turn to an equation covered by mode stability. For this to work nicely, one has to apply the supersymmetric removal in reverse order, first removing the $ \lambda = 1$ solution, then the $\lambda = 0$ one. Since this was not done in \cite{DonMos26}, we provide the transformation here explicitly. The transformation
    \begin{align*}
        \Tilde{f}(\rho) = \rho^{-\frac{d-1}{2}} (1 - \rho^2)^{\frac{d+1-2\lambda}{4}} \left(\pd_\rho -\frac{(\Tilde{g}_0^0)'(\rho)}{\Tilde{g}_0^0(\rho)} \right)(1 - \rho^2)\left( \pd_\rho - \frac{(g_0^1)'(\rho)}{g_0^1(\rho)} \right) \rho^\frac{d-1}{2} (1 - \rho^2)^\frac{2\lambda - (d-3)}{4}f(\rho)
    \end{align*}
    transforms the equation
    \begin{align*}
        f''(\rho) + \frac{d-1 - 2(\lambda + 1)\rho^2}{\rho(1- \rho^2)} f'(\rho) + \left( -\frac{\lambda(\lambda + 1)}{1-\rho^2} + \frac{4(d-1)(d-2 - \rho^2)}{(1 - \rho^2)(d-2 + \rho^2)^2} \right) f(\rho) = 0
    \end{align*}
    into
    \begin{align}\label{double susy}
        \Tilde{f}''(\rho) + \frac{d-1 - 2(\lambda + 1)\rho^2}{\rho(1 - \rho^2)} \Tilde{f}'(\rho)+ \left(-\frac{\lambda(\lambda + 1)}{1-\rho^2} - \frac{2d}{\rho^2(1 - \rho^2)} \right)\Tilde{f}(\rho) = 0,
    \end{align}
    where
    \begin{align*}
        g_0^1(\rho) &:= \rho^\frac{d-1}{2} (1 - \rho^2)^{-\frac{d-5}{4}}(d-2 + \rho^2)^{-1}\\
        \Tilde{g}_0^0(\rho) &:= \rho^\frac{d+1}{2} (1 - \rho^2)^{-\frac{d-3}{4}}.
    \end{align*}
    Conveniently, the final equation Eq.~\eqref{double susy} is the exact same as when doing the removal in the reverse order. In particular, we still know for $\lambda = 1$ that Eq.~\eqref{double susy} has no nontrivial solutions that are analytic at both $\rho = 0$ and $\rho = 1$. With this, one does the same steps as before, which finishes the proof.
\end{itemize}

\end{proof}
\subsection{Growth bound for the unperturbed generator}
We aim to use the spectral properties of $\mathbf{L}_0$ in order to get growth bounds for $\mathbf{S}_0(\tau)$. For this we additionally need the following resolvent bound.
\begin{prop}
Let $d \geq 3$ and $k \geq k_0$. There exists some $R > 1$ such that
\begin{align}\label{unpert resol bd for large lambda}
    \sup_{\substack{\lambda \in \overline{\Half}\\|\lambda| \geq R}}\|\mathbf{R}_{\mathbf{L}_0}(\lambda)\|_{\Hil^k \to \Hil^k} < \infty.
\end{align}
\end{prop}
\begin{proof}
We follow the idea from \cite{GhoLiuMas25} of bounding inner products from below by analyzing real and imaginary parts separately.
It suffices to prove the existence of $R > 1$ such that
\begin{align}\label{inner prod bd from below}
    |((\lambda \mathbf{I} - \mathbf{L}_0)\mathbf{f}|\mathbf{f})_{\Dot{\Hil}^k}| + |((\lambda \mathbf{I} - \mathbf{L}_0)\mathbf{f}|\mathbf{f})_{\Hil^1} |   \gtrsim \|\mathbf{f}\|_{\Hil^k}^2
\end{align}
holds for all $\mathbf{f} \in C^\infty(\overline{\B^d},\C^d)^2$ and $\lambda \in \overline{\Half}$ with $|\lambda| \geq R$. Indeed, if Eq.~\eqref{inner prod bd from below} holds, then we have by Cauchy-Schwarz
\begin{align*}
    \|\mathbf{f}\|_{\Hil^k}^2 &\lesssim |((\lambda \mathbf{I} - \mathbf{L}_0)\mathbf{f}|\mathbf{f})_{\Dot{\Hil}^k}| + |((\lambda \mathbf{I} - \mathbf{L}_0)\mathbf{f}|\mathbf{f})_{\Hil^1} |  \\
    &\leq \|(\lambda \mathbf{I} - \mathbf{L}_0)\mathbf{f}\|_{\Dot{\Hil}^k} \|\mathbf{f}\|_{\Dot{\Hil}^k} + \|(\lambda \mathbf{I} - \mathbf{L}_0)\mathbf{f}\|_{\Hil^1} \|\mathbf{f}\|_{\Hil^1}\\
    &\lesssim \|(\lambda \mathbf{I} - \mathbf{L}_0)\mathbf{f}\|_{\Hil^k} \|\mathbf{f}\|_{\Hil^k}
\end{align*}
and hence
\begin{align*}
    \|(\lambda \mathbf{I} - \mathbf{L}_0)\mathbf{f} \|_{\Hil^k} \gtrsim \|\mathbf{f}\|_{\Hil^k}.
\end{align*}
This bound together with the fact that for $R > 1$
\begin{align*}
    \{ \lambda \in \overline{\Half} : |\lambda| \geq R \} \subseteq \rho(\mathbf{L}_0)
\end{align*}
holds, by \autoref{eigenvalues of L0}, then yields \eqref{unpert resol bd for large lambda}.

In order to establish \eqref{inner prod bd from below}, we recall from \eqref{free homog diss} and \eqref{Lprime highest derivative} 
\begin{align*}
    \Real (\mathbf{L} \mathbf{f}|\mathbf{f})_{\Dot{\Hil}^k} &\leq \left(\frac{d}{2} - k  \right)\|\mathbf{f}\|_{\Dot{\Hil}^k}^2\\
    |(\mathbf{L}_0' \mathbf{f}|\mathbf{f})_{\Dot{\Hil}^k}| &\leq C \|\mathbf{f}\|_{\Dot{\Hil}^k}^2 + c_k \|\mathbf{f}\|_{\Hil^{k-1}}^2
\end{align*}
for all $\mathbf{f} \in C^\infty(\overline{\B^d},\C^d)^2$. Hence we have
\begin{align*}
    |((\lambda \mathbf{I} - (\mathbf{L} + \mathbf{L}_0'))\mathbf{f}|\mathbf{f})_{\Dot{\Hil}^k}| &\geq \Real ((\lambda \mathbf{I} - \mathbf{L})\mathbf{f}|\mathbf{f})_{\Dot{\Hil}^k} - |(\mathbf{L}_0' \mathbf{f}|\mathbf{f})_{\Dot{\Hil}^k}|\\
    &= \underbrace{\Real (\lambda \mathbf{f}|\mathbf{f})_{\Dot{\Hil}^k}}_{= \Real (\lambda) \|\mathbf{f}\|_{\Dot{\Hil}^k}^2 } - \Real (\mathbf{L} \mathbf{f}|\mathbf{f})_{\Dot{\Hil}^k} - |(\mathbf{L}_0' \mathbf{f}|\mathbf{f})_{\Dot{\Hil}^k}|\\
    &\geq \Real (\lambda) \|\mathbf{f}\|_{\Dot{\Hil}^k}^2 - \left(\frac{d}{2} - k\right)\|\mathbf{f}\|_{\Dot{\Hil}^k}^2 - (C \|\mathbf{f}\|_{\Dot{\Hil}^k}^2 + c_k \|\mathbf{f}\|_{\Hil^{k-1}}^2)\\
    &= \left( \Real \lambda + k - \frac{d}{2} - C  \right)\|\mathbf{f}\|_{\Dot{\Hil}^k}^2 - c_k \|\mathbf{f}\|_{\Hil^{k-1}}^2
\end{align*}
for all $\lambda \in \C$ and $\mathbf{f} \in C^\infty(\overline{\B^d},\C^d)^2$. On the other hand, since $k \geq k_0 \geq 3$, there exists $c_k' >0$ such that
\begin{align*}
    |(\mathbf{L}_0 \mathbf{f} | \mathbf{f})_{\Hil^1}| \leq c_k' \|\mathbf{f}\|_{\Hil^{k-1}}^2
\end{align*}
for all $\mathbf{f} \in C^\infty(\overline{\B^d},\C^d)^2$. Hence we have 
\begin{align*}
    |((\lambda \mathbf{I} - \mathbf{L}_0)\mathbf{f}|\mathbf{f})_{\Hil^1}| \geq |\lambda| \|\mathbf{f}\|_{\Hil^1}^2 - c_k' \|\mathbf{f}\|_{\Hil^{k-1}}^2
\end{align*}
for all $\mathbf{f} \in C^\infty(\overline{\B^d},\C^d)^2$. Together this yields
\begin{align*}
    &|((\lambda \mathbf{I} - \mathbf{L}_0)\mathbf{f}|\mathbf{f})_{\Dot{\Hil}^k}| + |((\lambda \mathbf{I} - \mathbf{L}_0)\mathbf{f}|\mathbf{f})_{\Hil^1} |\\
    &\geq \left( \Real \lambda + k - \frac{d}{2} - C  \right)\|\mathbf{f}\|_{\Dot{\Hil}^k}^2 + |\lambda| \|\mathbf{f}\|_{\Hil^1}^2  - (c_k + c_k')\|\mathbf{f}\|_{\Hil^{k-1}}^2.
\end{align*}
Using \eqref{strong Ehrling ineq}, we find some $c_k'' > 0$ such that
\begin{align*}
    \|\mathbf{f}\|_{\Hil^{k-1}}^2 \leq \frac{1}{2 (c_k+  c_k')}\|\mathbf{f}\|_{\Dot{\Hil}^k}^2 + c_k'' \|\mathbf{f}\|_{\Hil^1}^2
\end{align*}
for all $\mathbf{f} \in C^\infty(\overline{\B^d},\C^d)^2$. Inserting in the previous inequality yields
\begin{align*}
    &|((\lambda \mathbf{I} - \mathbf{L}_0)\mathbf{f}|\mathbf{f})_{\Dot{\Hil}^k}| + |((\lambda \mathbf{I} - \mathbf{L}_0)\mathbf{f}|\mathbf{f})_{\Hil^1} |\\
    &\geq \left( \Real \lambda + k - \frac{d}{2} - C  - \frac{1}{2} \right)\|\mathbf{f}\|_{\Dot{\Hil}^k}^2 + (|\lambda| - c_k''(c_k+  c_k') ) \|\mathbf{f}\|_{\Hil^1}^2\\
    &\overset{k \geq k_0}{\geq} \left( \Real \lambda + \frac{1}{2}\right) \|\mathbf{f}\|_{\Dot{\Hil}^k}^2 + (|\lambda| - c_k''(c_k+  c_k') ) \|\mathbf{f}\|_{\Hil^1}^2
\end{align*}
for all $\mathbf{f} \in C^\infty(\overline{\B^d},\C^d)^2$. For $\lambda \in \overline{\Half}$ with $|\lambda| \geq R := c_k''(c_k+  c_k') + \frac{1}{2}$ we conclude
\begin{align*}
    |((\lambda \mathbf{I} - \mathbf{L}_0)\mathbf{f}|\mathbf{f})_{\Dot{\Hil}^k}| + |((\lambda \mathbf{I} - \mathbf{L}_0)\mathbf{f}|\mathbf{f})_{\Hil^1}| \geq \frac{1}{2} (\|\mathbf{f}\|_{\Dot{\Hil}^k}^2  + \|\mathbf{f}\|_{\Hil^1}^2 ) = \frac{1}{2} \|\mathbf{f}\|_{\Hil^k}^2
\end{align*}
for all $\mathbf{f}\in C^\infty(\overline{\B^d},\C^d)^2$, which finishes the proof.
\end{proof}
Using this resolvent bound, we can prove exponential decay of $(\mathbf{S}_0(\tau))_{\tau\geq0}$ on its stable subspace.
\begin{thm}\label{stable unperturbed evolution}
Let $d \geq 3$ and $k \geq k_0$. There exist $\eps_0> 0$ and $C_0 > 0$ such that 
\begin{align*}
    \|\mathbf{S}_0(\tau) (\mathbf{I} - \mathbf{P}_0)\mathbf{f}\|_{\Hil^k} &\leq C_0 e^{-\eps_0 \tau} \|\mathbf{f}\|_{\Hil^k}\\
    \mathbf{S}_0(\tau) \mathbf{P}_{0,0}=\mathbf{P}_{0,0} \mathbf{S}_0(\tau) &=\mathbf{P}_{0,0} \\
    \mathbf{S}_0(\tau) \mathbf{P}_{1,0}=\mathbf{P}_{1,0} \mathbf{S}_0(\tau) &=e^\tau\mathbf{P}_{1,0}
\end{align*}
for all $\tau \geq 0$ and $\mathbf{f} \in \Hil^k$, where $\mathbf{P}_0 := \mathbf{P}_{0,0} + \mathbf{P}_{1,0}$.
\end{thm}
\begin{proof}
Since the Riesz projections $\mathbf{P}_{0,0},\mathbf{P}_{1,0}$ commute with the resolvents of $\mathbf{L}_0$, one sees, using for example the Post-Widder Inversion Formula \cite[p.~223, Corollary 5.5]{EngNag00}, that they also commute with $\mathbf{S}_0(\tau)$ for all $\tau \geq 0$. Since $\rg(\mathbf{P}_{\lambda,0}) = \ker(\lambda \mathbf{I}  - \mathbf{L}_0)$ this immediately yields $\mathbf{S}_0(\tau) \mathbf{P}_{\lambda, 0} = e^{\lambda \tau} \mathbf{P}_{\lambda, 0}$.

Next, we choose $R > 1$ large enough such that \eqref{unpert resol bd for large lambda} holds. Since
$\sigma((\mathbf{L}_0)_{\rg(\mathbf{I} - \mathbf{P}_0)}) = \sigma(\mathbf{L}_0)\setminus\{0,1\}$, see \cite[p.~178, Theorem 6.17]{Kat95}, it is by \autoref{eigenvalues of L0} contained in $\{\lambda \in \C: \Real \lambda < 0 \}$ and further
\begin{align*}
    \left\{\lambda \in \C: -\frac{w_k}{2}\leq\Real \lambda < 0 \right\} \cap \sigma(\mathbf{L}_0)
\end{align*}
is finite. Hence $\overline{\Half}$ is contained in $\rho((\mathbf{L}_0)_{\rg(\mathbf{I - \mathbf{P}})})$ and thus the resolvent is bounded on the compact half-disk $\{\lambda \in \overline{\Half} : |\lambda| \leq R\}$, i.e.,
\begin{align}\label{unpert resol bounded on compact set}
    \sup_{\substack{\lambda \in \overline{\Half}\\|\lambda|\leq R}} \|\mathbf{R}_{(\mathbf{L}_0)_{\rg(\mathbf{I - \mathbf{P}})}}(\lambda)\|_{\rg(\mathbf{I} - \mathbf{P}_0) \to \rg(\mathbf{I} - \mathbf{P}_0)} < \infty.
\end{align}
Now observe that whenever $\lambda \in \rho(\mathbf{L}_0)$, then the resolvent of the part of $\mathbf{L}_0$ in $\rg(\mathbf{I} - \mathbf{P}_0)$ coincides with the restriction of $\mathbf{R}_{\mathbf{L}_0}(\lambda)$ to $\rg(\mathbf{I} - \mathbf{P}_0)$, i.e.,
\begin{align*}
     \mathbf{R}_{(\mathbf{L}_0)_{\rg(\mathbf{I} - \mathbf{P}_0)}}(\lambda)= \mathbf{R}_{\mathbf{L}_0}(\lambda) \Big|_{\rg(\mathbf{I} - \mathbf{P}_0)}.
\end{align*}
From this, we immediately deduce a bound on the operator norms
\begin{align*}
    \|\mathbf{R}_{(\mathbf{L}_0)_{\rg(\mathbf{I} - \mathbf{P}_0)}}(\lambda)\|_{\rg(\mathbf{I} - \mathbf{P}_0) \to \rg(\mathbf{I} - \mathbf{P}_0)} \leq \|\mathbf{R}_{\mathbf{L}_0}(\lambda)\|_{\Hil^k \to \Hil^k}.
\end{align*}
Since $R > 1$, we can apply \eqref{unpert resol bd for large lambda} to infer
\begin{align*}
    \sup_{\substack{\lambda \in \overline{\Half}\\|\lambda|\geq R }} \|\mathbf{R}_{(\mathbf{L}_0)_{\rg(\mathbf{I} - \mathbf{P}_0)}}(\lambda)\|_{\rg(\mathbf{I} - \mathbf{P}_0) \to \rg(\mathbf{I} - \mathbf{P}_0)}\leq \sup_{\substack{\lambda \in \overline{\Half}\\|\lambda|\geq R }}\|\mathbf{R}_{\mathbf{L}_0}(\lambda)\|_{\Hil^k \to \Hil^k} < \infty.
\end{align*}
Together with \eqref{unpert resol bounded on compact set} this yields
\begin{align*}
    \sup_{\lambda \in \overline{\Half}} \|\mathbf{R}_{(\mathbf{L}_0)_{\rg(\mathbf{I} - \mathbf{P}_0)}}(\lambda)\|_{\rg(\mathbf{I} - \mathbf{P}_0) \to \rg(\mathbf{I} - \mathbf{P}_0)} < \infty.
\end{align*}
Then the Gearthart-Prüss-Greiner Theorem \cite[p.~302, Theorem 1.11]{EngNag00} implies that there exist $\eps_0 > 0$ and $C_0 > 0$ such that 
\begin{align*}
    \|\mathbf{S}_0(\tau)_{\rg(\mathbf{I}  - \mathbf{P}_0)} \mathbf{f}\|_{\Hil^k} \leq C_0 e^{-\eps_0 \tau}\|\mathbf{f}\|_{\Hil^k}
\end{align*} 
for all $\mathbf{f} \in \rg(\mathbf{I}  - \mathbf{P}_0)$, which readily implies the claim.
\end{proof}
\subsection{Spectrum of the perturbed generator}
We can now also fully understand the (unstable) spectrum of $\mathbf{L}_\Theta$, at least for sufficiently small $\Theta$.
\begin{prop}\label{perturbed spectrum}
Let $d \geq 3$ and $k \geq k_0$ and $0 < \eps_0 \leq 1$ as in \autoref{stable unperturbed evolution}. There exists some $M_0 \leq M$ such that the following holds for all $\Theta \in \overline{\B_{M_0}^{p(d)}}$.

The spectrum satisfies $\sigma(\mathbf{L}_\Theta) \cap \overline{\Half_{-\frac{\eps_0}{2}}} = \{0,1\}$ and the points $0,1$ are eigenvalues of $\mathbf{L}_\Theta$. The eigenspaces $\ker(\mathbf{L}_\Theta)$ and $\ker(\mathbf{I} - \mathbf{L}_\Theta) $ are $p(d)$- and $(d+1)$-dimensional, respectively. The algebraic multiplicity of both $0,1$ is finite and equal to the respective geometric multiplicity. Further, there exists some $\Tilde{C} > 0$ such that
\begin{align}\label{pert resol bounds away from spectrum}
    \|\mathbf{R}_{\mathbf{L}_\Theta}(\lambda)\|_{\Hil^k \to \Hil^k} \leq \Tilde{C}
\end{align}
for all $\lambda \in \overline{\Half_{-\frac{\eps_0}{2}}} \setminus (\disk_{\frac{\eps_0}{4}}(0) \cup \disk_{\frac{\eps_0}{4}}(1)  )$, where $\Tilde{C}$ is independent of $\Theta \in \overline{\B_{M_0}^{p(d)}}$.
\end{prop}
\begin{proof}
Let $\lambda \in \overline{\Half_{-\frac{\eps_0}{2}}} \setminus \{0,1\}$. We have $\lambda \in \rho(\mathbf{L}_0)$ and the Birman-Schwinger principle yields
\begin{align*}
    \lambda \mathbf{I} - \mathbf{L}_\Theta &= (\mathbf{I} - (\mathbf{L}_\Theta - \mathbf{L}_0) \mathbf{R}_{\mathbf{L}_0}(\lambda))(\lambda \mathbf{I} - \mathbf{L}_0) \\
    &= (\mathbf{I} - (\mathbf{L}_\Theta' - \mathbf{L}_0') \mathbf{R}_{\mathbf{L}_0}(\lambda))(\lambda \mathbf{I} - \mathbf{L}_0).
\end{align*}
Hence a Neumann series argument shows that also $\lambda \in \rho(\mathbf{L}_\Theta)$ whenever $\|(\mathbf{L}_\Theta' - \mathbf{L}_0')\mathbf{R}_{\mathbf{L}_0}(\lambda)\|_{\Hil^k \to \Hil^k}$ is small enough. To this end, we decompose
\begin{align*}
    \mathbf{R}_{\mathbf{L}_0}(\lambda ) &= \mathbf{R}_{\mathbf{L}_0}(\lambda )(\mathbf{I} - \mathbf{P}_0) + \mathbf{R}_{\mathbf{L}_0}(\lambda )\mathbf{P}_0.
\end{align*}
For the second term we compute
\begin{align*}
    \mathbf{R}_{\mathbf{L}_0}(\lambda )\mathbf{P}_0 = \mathbf{R}_{\mathbf{L}_0}(\lambda )\mathbf{P}_{0,0} + \mathbf{R}_{\mathbf{L}_0}(\lambda )\mathbf{P}_{1,0} = \lambda^{-1} \mathbf{P}_{0,0} + (\lambda - 1)^{-1}\mathbf{P}_{1,0}.
\end{align*}
Hence
\begin{align*}
    \|\mathbf{R}_{\mathbf{L}_0}(\lambda )\mathbf{P}_0\|_{\Hil^k \to \Hil^k} \lesssim 1
\end{align*}
for all $ \lambda \in \overline{\Half_{-\frac{\eps_0}{2}}} \setminus (\disk_{\frac{\eps_0}{4}}(0) \cup \disk_{\frac{\eps_0}{4}}(1)  )$. On the other hand, \autoref{stable unperturbed evolution} together with standard semigroup theory yields the bound
\begin{align*}
    \|\mathbf{R}_{\mathbf{L}_0}(\lambda)(\mathbf{I} - \mathbf{P}_0)\|_{\Hil^k \to \Hil^k} \leq \frac{C_0}{\Real \lambda + \eps_0} \lesssim 1
\end{align*}
for all $ \lambda \in \overline{\Half_{-\frac{\eps_0}{2}}} \setminus (\disk_{\frac{\eps_0}{4}}(0) \cup \disk_{\frac{\eps_0}{4}}(1)  )$. Hence there exists some $c > 0$ such that
\begin{align*}
     \|\mathbf{R}_{\mathbf{L}_0}(\lambda)\|_{\Hil^k \to \Hil^k} \leq c.
\end{align*}
for all $ \lambda \in \overline{\Half_{-\frac{\eps_0}{2}}} \setminus (\disk_{\frac{\eps_0}{4}}(0) \cup \disk_{\frac{\eps_0}{4}}(1)  )$. By \autoref{perturb is Lipschitz} there exists $L_k > 0$ such that
\begin{align*}
    \|\mathbf{L}_\Theta' - \mathbf{L}_0\|_{\Hil^k \to \Hil^k} \leq L_k |\Theta|
\end{align*}
for all $\Theta \in \overline{\B_{M}^{p(d)}}$. We set $M_0:= \min \{M,\frac{1}{2 L_k c}\}$ and see that
\begin{align*}
    \|(\mathbf{L}_\Theta' - \mathbf{L}_0)\mathbf{R}_{\mathbf{L}_0}(\lambda)\|_{\Hil^k \to \Hil^k} \leq \frac{1}{2}
\end{align*}
for all $ \lambda \in \overline{\Half_{-\frac{\eps_0}{2}}} \setminus (\disk_{\frac{\eps_0}{4}}(0) \cup \disk_{\frac{\eps_0}{4}}(1)  )$ and $\Theta \in \overline{\B_{M_0}^{p(d)}}$. From this we deduce that 
\begin{align}\label{pert resol set away from 01}
    \overline{\Half_{-\frac{\eps_0}{2}}} \setminus (\disk_{\frac{\eps_0}{4}}(0) \cup \disk_{\frac{\eps_0}{4}}(1)  ) \subseteq \rho(\mathbf{L}_\Theta)
\end{align}
and the resolvent map is bounded uniformly in $\Theta$ on this set, i.e., there exists $\Tilde{C} > 0$ such that
\begin{align*}
    \|\mathbf{R}_{\mathbf{L}_\Theta}(\lambda)\|_{\Hil^k \to \Hil^k} \leq \Tilde{C}
\end{align*}
for all $\Theta \in \overline{\B_{M_0}^{p(d)}}$, which is exactly \eqref{pert resol bounds away from spectrum}.

We now consider the Riesz projections 
\begin{align*}
    \mathbf{P}_{\lambda,\Theta} := \frac{1}{2 \pi i} \int_{\pd \disk_{\frac{\eps_0}{2}}(\lambda)} \mathbf{R}_{\mathbf{L}_{\Theta}}(z) dz.
\end{align*}
The projections $\mathbf{P}_{\lambda,\Theta}$ depend Lipschitz-continuously on $ \Theta$, which is checked using \autoref{perturb is Lipschitz} together with the second resolvent formula. We apply \cite[p.~34, Lemma 4.10]{Kat95} to see that $\dim(\rg \mathbf{P}_{\lambda,\Theta})$ remains constant for all $\Theta \in \overline{\B_{M_0}^{p(d)}}$, i.e.,
\begin{align*}
    \dim(\rg \mathbf{P}_{0,\Theta}) &= \dim(\rg \mathbf{P}_{0,0}) = p(d)\\
    \dim(\rg \mathbf{P}_{1,\Theta}) &= \dim(\rg \mathbf{P}_{1,0}) = d+1.
\end{align*}
Now we observe that, analogous to the case $\Theta=0$ in \autoref{subsect spec of L0}, applying each of the $p(d)$ and $d+1$ continuous symmetries of Eq.~\eqref{int wm eq selfsim} to the solution $v_\Theta$ will give rise to eigenfunctions of $\mathbf{L}_\Theta$ corresponding to the eigenvalues $0$ and 1, respectively. At this point it is not clear whether the family of functions obtained this way will again be linearly independent. However, since this family of functions still depends continuously on $\Theta$ and we know from the explicit case $\Theta =0 $ that they form a linearly independent set, we can assume, up to potentially shrinking $M_0$, that they are indeed linearly independent again for all $\Theta \in \overline{\B_{M_0}^{p(d)}}$. 

This argument shows that 
\begin{align*}
    \dim(\ker(\mathbf{L}_\Theta)) &\geq p(d)\\
    \dim(\ker(\mathbf{I} - \mathbf{L}_\Theta)) &\geq d+1.
\end{align*}
Since we always have the inclusions
\begin{align*}
    \ker(\mathbf{L}_\Theta) &\subseteq \rg(\mathbf{P}_{0,\Theta})\\
    \ker(\mathbf{I} - \mathbf{L}_\Theta) &\subseteq \rg(\mathbf{P}_{1,\Theta})
\end{align*}
we must have equality in both cases, which exactly means that the algebraic multiplicity is equal to the geometric multiplicity (and in particular finite).

These considerations yield
\begin{align*}
    \sigma(\mathbf{L}_\Theta) \cap (\disk_{\frac{\eps_0}{2}}(0) \cup \disk_{\frac{\eps_0}{2}}(1)) = \{0,1\}
\end{align*}
and together with \eqref{pert resol set away from 01} that
\begin{align*}
    \sigma(\mathbf{L}_\Theta) \cap \overline{\Half_{-\frac{\eps_0}{2}}} = \{0,1\}
\end{align*}
for all $\Theta \in \overline{\B_{M_0}^{p(d)}}$, which finishes the proof.
\end{proof}
\subsection{Growth bound for the perturbed generator}
The following theorem concludes this section and yields a sufficiently good characterization of the linear dynamics.
\begin{thm}\label{stable perturbed evolution}
Let $d \geq 3$ and $k \geq k_0$. There exist $\eps_\ast >0 $ and $C_\ast > 0$ such that
\begin{align*}
    \|\mathbf{S}_\Theta(\tau)  (\mathbf{I} - \mathbf{P}_\Theta)\mathbf{f}\|_{\Hil^k} &\leq C_\ast e^{-\eps_\ast \tau} \|(\mathbf{I} - \mathbf{P}_\Theta)\mathbf{f}\|_{\Hil^k}\\
    \mathbf{S}_\Theta(\tau) \mathbf{P}_{0,\Theta} = \mathbf{P}_{0,\Theta} \mathbf{S}_{\Theta}(\tau) &= \mathbf{P}_{0,\Theta}\\
    \mathbf{S}_\Theta(\tau) \mathbf{P}_{1,\Theta} = \mathbf{P}_{1,\Theta} \mathbf{S}_\Theta(\tau) &= e^\tau \mathbf{P}_{1, \Theta}
\end{align*}
for all $ \tau \geq 0 $, $\mathbf{f} \in \Hil^k$ and $\Theta \in \overline{\B_{M_0}^{p(d)}}$, where $\mathbf{P}_\Theta := \mathbf{P}_{0,\Theta} + \mathbf{P}_{1,\Theta}$ and $\mathbf{P}_{0,\Theta}$ and $\mathbf{P}_{1,\Theta}$ are the Riesz projections of $\mathbf{L}_\Theta$ associated to the eigenvalue $0$ and $1$, respectively.
\end{thm}
\begin{proof}
We see that $\mathbf{P}_{\lambda,\Theta}$ commutes with $\mathbf{S}_\Theta(\tau)$ and $\mathbf{S}_\Theta(\tau) \mathbf{P}_{\lambda, \Theta} = e^{\lambda \tau} \mathbf{P}_{\lambda, \Theta}$ for all $\tau \geq 0$ and $\lambda \in \{0,1\}$ as in the proof of \autoref{stable unperturbed evolution}.

From \eqref{pert resol set away from 01} we know that there exists some $\Tilde{C } > 0$ such that
\begin{align*}
    \|\mathbf{R}_{\mathbf{L}_\Theta}(\lambda)\|_{\Hil^k \to \Hil^k} \leq \Tilde{C}
\end{align*}
for all $\lambda \in \overline{\Half_{-\frac{\eps_0}{2}}} \setminus (\disk_{\frac{\eps_0}{4}}(0) \cup \disk_{\frac{\eps_0}{4}}(1) )$ and $\Theta \in \overline{\B_{M_0}^{p(d)}}$. On the other hand, since $\overline{\Half_{-\frac{\eps_0}{2}}} \subseteq \rho((\mathbf{L}_\Theta)_{\rg(\mathbf{I} - \mathbf{P}_\Theta)})$ by \autoref{perturbed spectrum}, we know that $\lambda \mapsto \mathbf{R}_{(\mathbf{L}_\Theta)_{\rg(\mathbf{I} - \mathbf{P}_\Theta)}}(\lambda)$ is analytic on $\Half_{-\frac{\eps_0}{2}}$. Hence the maximum principle for analytic functions implies that 
\begin{align*}
    &\sup_{\lambda \in \disk_{\frac{\eps_0}{4}}(0) \cup \disk_{\frac{\eps_0}{4}}(1)}\|\mathbf{R}_{(\mathbf{L}_\Theta)_{\rg(\mathbf{I} - \mathbf{P}_\Theta)}}(\lambda)\|_{\rg(\mathbf{I} - \mathbf{P}_\Theta) \to \rg(\mathbf{I} - \mathbf{P}_\Theta)} \\
    &\leq \sup_{\lambda \in \pd(\disk_{\frac{\eps_0}{4}}(0) \cup \disk_{\frac{\eps_0}{4}}(1))}\|\mathbf{R}_{(\mathbf{L}_\Theta)_{\rg(\mathbf{I} - \mathbf{P}_\Theta)}}(\lambda)\|_{\rg(\mathbf{I} - \mathbf{P}_\Theta) \to \rg(\mathbf{I} - \mathbf{P}_\Theta)}.
\end{align*}
Since $\pd(\disk_{\frac{\eps_0}{4}}(0) \cup \disk_{\frac{\eps_0}{4}}(1)) \subseteq \overline{\Half_{-\frac{\eps_0}{2}}} \setminus (\disk_{\frac{\eps_0}{4}}(0) \cup \disk_{\frac{\eps_0}{4}}(1) )$ we have 
\begin{align*}
    \|\mathbf{R}_{(\mathbf{L}_\Theta)_{\rg(\mathbf{I} - \mathbf{P}_\Theta)}}(\lambda)\|_{\rg(\mathbf{I} - \mathbf{P}_\Theta) \to \rg(\mathbf{I} - \mathbf{P}_\Theta)} \leq \|\mathbf{R}_{\mathbf{L}_\Theta}(\lambda)\|_{\Hil^k \to \Hil^k} \leq \Tilde{C}
\end{align*}
for all $\lambda \in \pd(\disk_{\frac{\eps_0}{4}}(0) \cup \disk_{\frac{\eps_0}{4}}(1))$. Analogously we get 
\begin{align*}
    \|\mathbf{R}_{(\mathbf{L}_\Theta)_{\rg(\mathbf{I} - \mathbf{P}_\Theta)}}(\lambda)\|_{\rg(\mathbf{I} - \mathbf{P}_\Theta) \to \rg(\mathbf{I} - \mathbf{P}_\Theta)} \leq \Tilde{C}
\end{align*}
for all $\lambda \in \overline{\Half_{-\frac{\eps_0}{2}}} \setminus (\disk_{\frac{\eps_0}{4}}(0) \cup \disk_{\frac{\eps_0}{4}}(1) )$. Hence we conclude 
\begin{align}\label{reduced resol of pert unif bd}
    \sup_{\lambda \in \overline{\Half_{-\frac{\eps_0}{2}}}} \|\mathbf{R}_{(\mathbf{L}_\Theta)_{\rg(\mathbf{I} - \mathbf{P}_\Theta)}}(\lambda)\|_{\rg(\mathbf{I} - \mathbf{P}_\Theta) \to \rg(\mathbf{I} - \mathbf{P}_\Theta)} \leq \Tilde{C}
\end{align}
for all $\Theta \in \overline{\B_{M_0}^{p(d)}}$. Hence the Gearhart-Prüss-Greiner Theorem \cite[p.~302, Theorem 1.11]{EngNag00} implies the existence of $C_\ast > 0$ and $\eps_\ast > \frac{\eps_0}{2} > 0$ such that
\begin{align}\label{reduced pert growth bd}
    \|\mathbf{S}_{\Theta}(\tau)_{\rg(\mathbf{I} - \mathbf{P}_\Theta)}  \|_{\rg(\mathbf{I} - \mathbf{P}_\Theta) \to \rg(\mathbf{I} - \mathbf{P}_\Theta)} \leq C_\ast e^{-\eps_\ast \tau}
\end{align}
for all $\tau \geq 0$. In fact, since the constant $\Tilde{C}$ in \eqref{reduced resol of pert unif bd} does not depend on $\Theta \in \overline{\B_{M_0}^{p(d)}}$, one can choose $C_\ast,\eps_\ast$ independently of $\Theta \in \overline{\B_{M_0}^{p(d)}}$ such that \eqref{reduced pert growth bd} holds for all $\Theta \in \overline{\B_{M_0}^{p(d)}}$, see \cite[p.~38, Theorem A.1]{Ost24}.

From \eqref{reduced pert growth bd} one directly deduces the last statement of the theorem.
\end{proof}

\section{Nonlinear stability}
Now we turn to the nonlinear problem
\begin{align}\label{nonlin first order w eq}
    \pd_\tau \mathbf{u}(\tau) = \mathbf{L}_\Theta \mathbf{u}(\tau) + \mathbf{N}_\Theta(\mathbf{u}(\tau)).
\end{align}
The associated Cauchy problem is 
\begin{align}\label{u Cauchy}
    \begin{cases}
        \pd_\tau \mathbf{u}(\tau) = \mathbf{L}_\Theta \mathbf{u}(\tau) + \mathbf{N}_\Theta(\mathbf{u}(\tau))\\
        \mathbf{u}(0) = \mathbf{\mathbf{f}}^{T,X} + \mathbf{v}_0^{T,X} - \mathbf{v}_\Theta
    \end{cases},
\end{align}
where we define for given $\mathbf{f} = (f_1,f_2)$ the function $\mathbf{f}^{T,X}$ via
\begin{align*}
    (\mathbf{f}^{T,X})_1(\xi) &= f_1(X + T \xi)\\
    (\mathbf{f}^{T,X})_2(\xi) &= T f_2(X + T \xi) 
\end{align*}
and accordingly
\begin{align*}
    \mathbf{v}_0^{T,X}(\xi) &= \Vector{\frac{1}{\sqrt{d-2}}(X + T \xi)}{\frac{T}{\sqrt{d-2}}(X + T \xi)}
\end{align*}
and
\begin{align*}
    \mathbf{v}_\Theta &= \Vector{v_\Theta}{\Lambda v_\Theta}.
\end{align*}
Eq.~\eqref{u Cauchy} in Duhamel form is given by
\begin{align}\label{w Duhamel}
    \mathbf{u}(\tau) = \mathbf{S}_\Theta \mathbf{u}(0) + \int_{0}^\tau \mathbf{S}_\Theta(\tau - \tau')\mathbf{N}_\Theta(\mathbf{u}(\tau')) d\tau'. 
\end{align}
\subsection{Properties of the nonlinearity}
In order to efficiently deal with the nonlinearity $\mathbf{N}_\Theta$, we collect two useful formulae in the next lemma.
\begin{lem}\label{formula nonlinearity general}
Let $m \geq 1$ and $F \in C^\infty(\C^m)$. For $y_0 \in \C^m$ consider the function $N \in C^\infty(\C^m)$ defined by 
\begin{align*}
    N(y) = F(y_0 + y) - F(y_0) -  D F(y_0)y.
\end{align*}
Then 
\begin{align}\label{N quadratic}
    N(y)  = y^a y^b \int_0^1 (1 - t) \pd_a \pd_b F(y_0 + t y) dt
\end{align}
and
\begin{align}\label{N difference identity}
    N(y_1) - N(y_2) = (y_1^a - y_2^a) \int_0^1 (t y_1^b + (1-t) y_2^b)\int_0^1 \pd_a \pd_b F(y_0 + s(y_2 + t(y_1 - y_2))) ds dt
\end{align}
hold for all $y,y_1,y_2 \in \C^m$.
\end{lem}
\begin{proof}
Eq.~\eqref{N quadratic} follows from a second order Taylor expansion around $y_0$. For Eq.~\eqref{N difference identity} we compute
\begin{align*}
    N(y_1) - N(y_2) &= F(y_0 + y_1) - F(y_0 + y_2) - D F(y_0)(y_1 - y_2)\\
    &= \int_0^1 \pd_t F(y_0  + y_2  + t(y_1 - y_2)) dt - \pd_a F(y_0)(y_1^a - y_2^a)\\
    &= \int_0^1 \pd_a F(y_0 + y_2 + t(y_1 - y_2))(y_1^a - y_2^a) dt - \pd_a F(y_0)(y_1^a - y_2^a)\\
    &= (y_1^a - y_2^a)\int_0^1 [\pd_a F(y_0 + y_2 + t(y_1 - y_2)) - \pd_a F(y_0)] dt\\
    &= (y_1^a - y_2^a)\int_0^1 \int_0^1 \pd_s \pd_a F(y_0 + s(y_2 + t(y_1 - y_2)))ds dt\\
    &= (y_1^a - y_2^a) \int_0^1 \int_0^1 \pd_a \pd_b F(y_0 + s (y_2 + t(y_1 - y_2)))(y_2^b + t(y_1^b - y_2^b)  ) ds dt\\
    &= (y_1^a - y_2^a) \int_0^1 (t y_1^b + (1-t)y_2^b) \int_0^1 \pd_a \pd_b F(y_0 + s (y_2 + t(y_1 - y_2))) ds dt
\end{align*}
where we applied the fundamental theorem of calculus twice.
\end{proof}
This allows us to extract the relevant mapping properties of $\mathbf{N}_\Theta$.
\begin{prop}
Let $d \geq 3$ and $ k \geq k_0$. For each $\Theta \in \overline{\B_{M_0}^{p(d)}}$ the map $\mathbf{N}_\Theta$ can be uniquely extended from $C^\infty(\overline{\B^d},\C^d)^2$ to 
\begin{align*}
    \mathbf{N}_\Theta:\Hil^k \to \Hil^k.
\end{align*}
Further, we have the estimates
\begin{align}\label{Ntheta quadratic}
    \|\mathbf{N}_\Theta(\mathbf{f})\|_{\Hil^k} \lesssim \|\mathbf{f}\|_{\Hil^k}^2\\
\end{align}
and 
\begin{align}\label{Ntheta lipschitz}
    \|\mathbf{N}_\Theta(\mathbf{f}) - \mathbf{N}_\Theta(\mathbf{g})\|_{\Hil^k} \lesssim \|\mathbf{f} - \mathbf{g}\|_{\Hil^k}(\|\mathbf{f}\|_{\Hil^k} + \|\mathbf{g}\|_{\Hil^k} )
\end{align}
for all $\Theta \in \overline{\B_{M_0}^{p(d)}}$, $0 < \delta \leq 1$ and $\mathbf{f},\mathbf{g} \in \Hil^k_\delta$.
\end{prop}
\begin{proof}
Since $k \geq k_0 > \frac{d}{2} + 1$, we have that $H^{k-1}(\B^d)$ is a Sobolev algebra. Recall that for $\mathbf{f} \in C^\infty(\overline{\B^d},\C^d)^2$
\begin{align*}
    \mathbf{N}_\Theta(\mathbf{f}) = \Vector{0}{ - N_\Theta(f_1,Df_1,f_2)} \in C^\infty(\overline{\B^d},\C^d)^2
\end{align*}
and $N_\Theta$ arises from $F$ exactly as in \autoref{formula nonlinearity general} (where $((v_{\Theta})_1,D (v_{\Theta})_1,(v_{\Theta})_2)$ corresponds to $y_0$). Note that all components of $f_1,Df_1,f_2$ belong to $C^\infty(\overline{\B^d})$. Hence Eq.~\eqref{N quadratic} together with the fact that $H^{k-1}(\B^d)$ is a Sobolev algebra yields
\begin{align*}
    \|\mathbf{N}_\Theta(\mathbf{f})\|_{\Hil^k} \lesssim \|\mathbf{f}\|_{\Hil^k}^2 \sup_{a,b}\int_0^1 \|\pd_a \pd_b F(\mathbf{v}_\Theta + t \mathbf{f})\|_{H^{k-1}(\B^d,\C^d)} dt.
\end{align*}
Recall that 
\begin{align*}
    F^n(\mathbf{f})  &= \Gamma_{i j}^n(f_1)(f_2^i f_2^j - \pd^m f_1^i \pd_m f_1^j)
\end{align*}
with
\begin{align*}
    \Gamma_{i j}^n(z) = - \frac{2}{1 + |z|^2}(z_i \delta_j^n + z_j \delta_i^n - z^n \delta_{i j}).
\end{align*}
$F^n$ is cubic in the sense that $\pd_a \pd_b F^n(\mathbf{0}) = 0$ for all $a,b$. Using a Schauder estimate argument, see for example \cite[p.~339 ff.]{Tao06}, we infer the existence of a continuous non-decreasing function $\gamma_{a,b}^n:[0,\infty) \to [0,\infty)$ such that
\begin{align*}
    \|\pd_a \pd_b F^n(\mathbf{f})\|_{H^{k-1}(\B^d)} \leq \gamma_{a,b}^n(\|\mathbf{f}\|_{\Hil^k})
\end{align*}
for all $\mathbf{f} \in C^\infty(\overline{\B^d},\C^d)^2$. Since $\|\mathbf{v}_\Theta\|_{\Hil^k} \lesssim 1$ for all $\Theta \in \overline{\B_{M_0}^{p(d)}}$ we thus conclude 
\begin{align*}
    \sup_{a, b} \int_0^1 \|\pd_a \pd_b F(\mathbf{v}_\Theta + t \mathbf{f})\|_{H^{k-1}(\B^d,\C^d)} dt \lesssim 1 
\end{align*}
for all $\Theta \in \overline{\B_{M_0}^{p(d)}}$ and $\mathbf{f} \in C^\infty(\overline{\B^d},\C^d)^2$ with $\|\mathbf{f}\|_{\Hil^k} \leq \delta \leq 1$. Hence \eqref{Ntheta quadratic} holds for $\mathbf{f} \in C^\infty(\overline{\B^d},\C^d)^2 \cap \Hil^k_\delta$. Analogously one checks with Eq.~\eqref{N difference identity} that \eqref{Ntheta lipschitz} holds for $\mathbf{f},\mathbf{g} \in C^\infty(\overline{\B^d},\C^d)^2 \cap \Hil^k_\delta$.

Since $C^\infty(\overline{\B^d},\C^d)^2 \cap \Hil^k_\delta$ is dense in $\Hil_\delta^k$, the inequality \eqref{Ntheta lipschitz} shows that $\mathbf{N}_\Theta$ can be extended uniquely to a map from $\Hil^k_\delta$ to $\Hil^k$ and the inequalities \eqref{Ntheta quadratic} and \eqref{Ntheta lipschitz} extend to $\Hil^k_\delta$.

Finally, observe that for any $\delta > 0$ the same argument shows that \eqref{Ntheta quadratic} and \eqref{Ntheta lipschitz} hold for all $\mathbf{f}, \mathbf{g} \in C^\infty(\overline{\B^d},\C^d)^2 \cap \Hil^k_\delta$, where the implicit constants now depend on $\delta $. Hence we can again uniquely extend $\mathbf{N}_\Theta$ to $\Hil^k_\delta$. Since this holds for any $\delta > 0$, in total this shows that $\mathbf{N}_\Theta$ can be extended to all of $\Hil^k$.
\end{proof}
\subsection{Stabilizing the nonlinear evolution}
Now we study Eq.~\eqref{w Duhamel} in more detail. The first difficulty will be to obtain global-in-time solutions. This is achieved by introducing a suitable correction term in Eq.~\eqref{w Duhamel}. First we introduce the function space we will work in.
\begin{defin}
Let $d \geq 3$ and $k \geq k_0$. Let $\eps_\ast > 0$ be as in \autoref{stable perturbed evolution}. We define the space
\begin{align*}
    \X := \{ \mathbf{u} \in C([0,\infty),\Hil^k): \sup_{\tau \geq 0} e^{\eps_\ast \tau}\|\mathbf{u}(\tau)\|_{\Hil^k} < \infty  \} 
\end{align*}
with associated norm 
\begin{align*}
    \|\mathbf{u}\|_{\X} = \sup_{\tau \geq 0} e^{\eps_\ast \tau}\|\mathbf{u}(\tau)\|_{\Hil^k}.
\end{align*}
\end{defin}
Observe that $(\X,\|\cdot\|_\X)$ is a Banach space. This space is the ideal setting to study solutions of Eq.~\eqref{w Duhamel}, since it has a decay rate encoded, i.e., a function $\mathbf{u} \in C([0,\infty),\Hil^k)$ belongs to $\X$ if and only if 
\begin{align*}
    \|\mathbf{u}(\tau)\|_{\Hil^k} \lesssim e^{- \eps_\ast \tau}
\end{align*}
for all $\tau \geq 0$.

Now we can define the correction term.
\begin{defin}
Let $d \geq 3$, $k \geq k_0$, $\Theta \in \overline{\B_{M_0}^{p(d)}}$ and $\mathbf{P}_{0,\Theta},\mathbf{P}_{1,\Theta},\mathbf{P}_\Theta$ be the spectral projections on $\Hil^k$ from \autoref{stable perturbed evolution}. We define for $0 < \delta \leq 1$ the map
\begin{align*}
    \mathbf{C}_\Theta:\Hil^k \times \X_\delta \to \Hil^k
\end{align*}
via
\begin{align*}
    \mathbf{C}_\Theta(\mathbf{f},\mathbf{u}) &= \mathbf{P}_\Theta\mathbf{f} + \mathbf{P}_{0,\Theta}\int_0^\infty \mathbf{N}_\Theta(\mathbf{u}(\tau')) d\tau' + \mathbf{P}_{1,\Theta}\int_0^\infty e^{-\tau'} \mathbf{N}_\Theta(\mathbf{u}(\tau')) d\tau'.
\end{align*}
\end{defin}
This is well-defined since the appearing integrals converge absolutely by \eqref{Ntheta quadratic}. With this, we can consider for $\mathbf{f} \in \Hil^k$ the map $\mathbf{K}_{\Theta,\mathbf{f}}:\X_\delta \to\X$ defined by
\begin{align*}
    \mathbf{K}_{\Theta,\mathbf{f}}(\mathbf{u})(\tau) := \mathbf{S}_\Theta(\tau)(\mathbf{f} - \mathbf{C}_\Theta(\mathbf{f},\mathbf{u})) + \int_0^\tau \mathbf{S}_\Theta(\tau- \tau') \mathbf{N}_\Theta(\mathbf{u}(\tau')) d\tau'.
\end{align*}
Now we study fixed points of $\mathbf{K}_{\Theta,\mathbf{f}}$.
\begin{prop}\label{stabilized by correction}
Let $d \geq 3$, $k \geq k_0$ and $\Theta \in \overline{\B_{M_0}^{p(d)}}$. There exist $0 < c_0,\delta_0 \leq 1$ such that for all $0 < \delta \leq \delta_0$ and $\mathbf{f} \in \Hil^k$ with $\|\mathbf{f}\|_{\Hil^k} \leq c_0 \delta $ there exists a unique $\mathbf{u}_\Theta \in\X_\delta$ such that 
\begin{align*}
    \mathbf{u}_\Theta = \mathbf{K}_{\Theta,\mathbf{f}}(\mathbf{u}_\Theta).
\end{align*}
Furthermore, the data-to-solution map $\mathbf{f} \mapsto \mathbf{u}_\Theta$ is Lipschitz-continuous as a map from $\Hil^k_\delta$ to $\X$.
\end{prop}
\begin{proof}
For $ 0 < \delta \leq 1 $ let $\mathbf{f} \in \Hil^k$ and $\mathbf{u} \in\X_\delta$.
Using the explicit formulae available of how $\mathbf{S}_\Theta(\tau)$ acts on $\rg(\mathbf{P}_\Theta)$ from \autoref{stable perturbed evolution}, one computes
\begin{align*}
    \mathbf{K}_{\Theta,\mathbf{f}}(\mathbf{u})(\tau) &=\mathbf{S}_\Theta( \tau)(\mathbf{I} - \mathbf{P}_\Theta)\mathbf{f} + \int_0^\tau \mathbf{S}_\Theta(\tau - \tau') (\mathbf{I} - \mathbf{P}_\Theta) \mathbf{N}_\Theta(\mathbf{u}(\tau'))d \tau'\\
    &- \mathbf{P}_{0,\Theta} \int_{\tau}^\infty \mathbf{N}_\Theta(\mathbf{u}(\tau')) d \tau' - \mathbf{P}_{1,\Theta} \int_\tau^\infty e^{\tau-\tau'} \mathbf{N}_\Theta(\mathbf{u}(\tau')) d \tau'.
\end{align*}
Hence from \autoref{stable perturbed evolution} we conclude
\begin{align*}
    \|\mathbf{K}_{\Theta,\mathbf{f}}(\mathbf{u})(\tau)\|_{\Hil^k} &\lesssim e^{-\eps_\ast \tau} \|\mathbf{f}\|_{\Hil^k} + \int_0^\tau e^{-\eps_\ast(\tau - \tau')} \|\mathbf{N}_\Theta(\mathbf{u}(\tau'))\|_{\Hil^k} d\tau'\\
    &+ \int_\tau^\infty \|\mathbf{N}_\Theta(\mathbf{u}(\tau'))\|_{\Hil^k} d \tau' + \int_\tau^\infty e^{\tau - \tau'}\|\mathbf{N}_\Theta(\mathbf{u}(\tau'))\|_{\Hil^k} d \tau' \\
    &\overset{\eqref{Ntheta quadratic}}{\lesssim} e^{-\eps_\ast \tau} \|\mathbf{f}\|_{\Hil^k} + \int_0^\tau e^{-\eps_\ast(\tau- \tau')} \|\mathbf{u}(\tau')\|_{\Hil^k}^2 d \tau'+ \int_\tau^\infty \|\mathbf{u}(\tau')\|_{\Hil^k}^2 d \tau'\\
    &\leq e^{-\eps_\ast \tau} \|\mathbf{f}\|_{\Hil^k} + \delta^2 \int_0^\tau e^{-\eps_\ast(\tau- \tau')} e^{-2 \eps_\ast \tau'} d \tau'+ \delta^2 \int_\tau^\infty e^{-2 \eps_\ast \tau'} d \tau'\\
    &= e^{-\eps_\ast \tau} \|\mathbf{f}\|_{\Hil^k} + \delta^2 e^{-\eps_\ast \tau} \int_0^\tau e^{-\eps_\ast \tau'} d \tau' + \delta^2 \frac{e^{-2 \eps_\ast  \tau}}{2 \eps_\ast}\\
    &\lesssim e^{-\eps_\ast \tau} \|\mathbf{f}\|_{\Hil^k} + \delta^2 e^{-\eps_\ast \tau}
\end{align*}
and hence
\begin{align*}
    \|\mathbf{K}_{\Theta, \mathbf{f}}(\mathbf{u})\|_{\X} \lesssim \|\mathbf{f}\|_{\Hil^k}  + \delta^2.
\end{align*}
Next for $\mathbf{u}_1,\mathbf{u}_2 \in\X_\delta$ we have
\begin{align*}
    &\|\mathbf{K}_{\Theta,\mathbf{f}}(\mathbf{u}_1)(\tau) -  \mathbf{K}_{\Theta,\mathbf{f}}(\mathbf{u}_2)(\tau) \|_{\Hil^k} \\
    & \leq \int_0^\tau \|\mathbf{S}_\Theta(\tau-\tau') (\mathbf{I} - \mathbf{P}_\Theta) [\mathbf{N}_\Theta(\mathbf{u}_1(\tau'))  - \mathbf{N}_\Theta(\mathbf{u}_2(\tau'))  ]\|_{\Hil^k} d \tau'\\
    &+ \int_\tau^\infty \|\mathbf{N}_\Theta(\mathbf{u}_1(\tau')) -  \mathbf{N}_\Theta(\mathbf{u}_2(\tau'))   \|_{\Hil^k} d \tau' \\
    &+ \int_\tau^\infty e^{\tau- \tau'}\|\mathbf{N}_\Theta(\mathbf{u}_1(\tau')) -  \mathbf{N}_\Theta(\mathbf{u}_2(\tau'))   \|_{\Hil^k} d \tau'\\
    &\lesssim \int_0^\tau e^{-\eps_\ast (\tau - \tau')} \|\mathbf{N}_\Theta(\mathbf{u}_1(\tau'))  - \mathbf{N}_\Theta(\mathbf{u}_2(\tau'))  \|_{\Hil^k} d \tau'\\
    &+ \int_\tau^\infty \|\mathbf{N}_\Theta(\mathbf{u}_1(\tau')) -  \mathbf{N}_\Theta(\mathbf{u}_2(\tau'))   \|_{\Hil^k} d \tau'\\
    &\overset{\eqref{Ntheta lipschitz}}{\lesssim} \int_0^\tau e^{-\eps_\ast(\tau- \tau')}\|\mathbf{u}_1(\tau') - \mathbf{u}_2(\tau')\|_{\Hil^k}(\|\mathbf{u}_1(\tau')\|_{\Hil^k} + \| \mathbf{u}_2(\tau')\|_{\Hil^k})d \tau'\\
    &+\int_\tau^\infty \|\mathbf{u}_1(\tau') - \mathbf{u}_2(\tau')\|_{\Hil^k}(\|\mathbf{u}_1(\tau')\|_{\Hil^k} + \| \mathbf{u}_2(\tau')\|_{\Hil^k})d \tau'\\
    &\leq \int_0^\tau e^{-\eps_\ast(\tau - \tau')} e^{-2 \eps_\ast \tau'} \|\mathbf{u}_1  - \mathbf{u}_2\|_{\X}(\|\mathbf{u}_1\|_{\X} + \|\mathbf{u}_2\|_{\X}  ) d \tau'\\
    &+ \int_\tau^\infty e^{-2 \eps_\ast \tau'} \|\mathbf{u}_1  - \mathbf{u}_2\|_{\X}(\|\mathbf{u}_1\|_{\X} + \|\mathbf{u}_2\|_{\X}  ) d\tau'\\
    &= e^{-\eps_\ast \tau} \|\mathbf{u}_1  - \mathbf{u}_2\|_{\X}(\|\mathbf{u}_1\|_{\X} + \|\mathbf{u}_2\|_{\X}  ) \int_0^\tau e^{-\eps_\ast \tau'} d \tau'\\
    &+ \|\mathbf{u}_1  - \mathbf{u}_2\|_{\X}(\|\mathbf{u}_1\|_{\X} + \|\mathbf{u}_2\|_{\X}  )\int_\tau^\infty e^{-\eps_\ast \tau'} d \tau'\\
    &\lesssim  e^{-\eps_\ast \tau}\|\mathbf{u}_1  - \mathbf{u}_2\|_{\X}(\|\mathbf{u}_1\|_{\X} + \|\mathbf{u}_2\|_{\X}  )
\end{align*}
and hence
\begin{align*}
    \|\mathbf{K}_{\Theta,\mathbf{f}}(\mathbf{u}_1) -  \mathbf{K}_{\Theta,\mathbf{f}}(\mathbf{u}_2)\|_{\X} &\lesssim \|\mathbf{u}_1  - \mathbf{u}_2\|_{\X}(\|\mathbf{u}_1\|_{\X} + \|\mathbf{u}_2\|_{\X}  )\\
    &\lesssim \delta \|\mathbf{u}_1  - \mathbf{u}_2\|_{\X}.
\end{align*}
Hence we can choose $0 < \delta_0,c_0 \leq 1$ sufficiently small such that 
\begin{align*}
    \|\mathbf{K}_{\Theta,\mathbf{f}}(\mathbf{u})\|_{\X} &\leq \delta\\
    \|\mathbf{K}_{\Theta,\mathbf{f}}(\mathbf{u}_1 ) -\mathbf{K}_{\Theta,\mathbf{f}}(\mathbf{u}_2 ) \|_{\X} &\leq \frac{1}{2} \|\mathbf{u}_1 - \mathbf{u}_2\|_{\X}
\end{align*}
holds for all $0 < \delta \leq \delta_0$, $\mathbf{f}\in\Hil^k$ with $\|\mathbf{f}\|_{\Hil^k} \leq c_0 \delta $ and $\mathbf{u},\mathbf{u}_1,\mathbf{u}_2 \in\X_\delta$. Hence $\mathbf{K}_{\Theta,\mathbf{f}}$ is a contraction on $\X_\delta$ and by Banach's fixed point theorem there exists a unique $\mathbf{u}_\Theta \in\X_\delta$ such that $\mathbf{u}_\Theta = \mathbf{K}_{\Theta,\mathbf{f}}(\mathbf{u}_\Theta)$.

Next, let $\mathbf{f},\mathbf{g} \in \Hil^k$ with $\|\mathbf{f}\|_{\Hil^k},\|\mathbf{g}\|_{\Hil^k} \leq c_0 \delta $ and let $\mathbf{u}_\Theta,\Tilde{\mathbf{u}}_\Theta \in\X_\delta$ such that 
\begin{align*}
    \mathbf{u}_\Theta &= \mathbf{K}_{\Theta,\mathbf{f}}(\mathbf{u}_\Theta)\\
    \Tilde{\mathbf{u}}_\Theta &= \mathbf{K}_{\Theta,\mathbf{g}}(\Tilde{\mathbf{u}}_\Theta).
\end{align*}
We compute
\begin{align*}
    \|\mathbf{u}_\Theta(\tau) - \Tilde{\mathbf{u}}_\Theta(\tau)\|_{\Hil^k} &= \| \mathbf{K}_{\Theta,\mathbf{f}}(\mathbf{u}_\Theta)(\tau) - \mathbf{K}_{\Theta,\mathbf{g}}(\Tilde{\mathbf{u}}_\Theta)(\tau)\|_{\Hil^k}\\
    &\leq \| \mathbf{K}_{\Theta,\mathbf{f}}(\mathbf{u}_\Theta)(\tau) - \mathbf{K}_{\Theta,\mathbf{f}}(\Tilde{\mathbf{u}}_\Theta)(\tau)\|_{\Hil^k} + \| \mathbf{K}_{\Theta,\mathbf{f}}(\Tilde{\mathbf{u}}_\Theta)(\tau) - \mathbf{K}_{\Theta,\mathbf{g}}(\Tilde{\mathbf{u}}_\Theta)(\tau)\|_{\Hil^k}\\
    &\leq \frac{1}{2} e^{-\eps_\ast \tau}\|\mathbf{u}_\Theta - \Tilde{\mathbf{u}}_\Theta\|_{\X} + \| \mathbf{K}_{\Theta,\mathbf{f}}(\Tilde{\mathbf{u}}_\Theta)(\tau) - \mathbf{K}_{\Theta,\mathbf{g}}(\Tilde{\mathbf{u}}_\Theta)(\tau)\|_{\Hil^k},
\end{align*}
since $\mathbf{K}_{\Theta,\mathbf{f}}$ is a contraction on $\X_\delta$. Next, one computes
\begin{align*}
    \mathbf{K}_{\Theta,\mathbf{f}}(\Tilde{\mathbf{u}}_\Theta)(\tau) - \mathbf{K}_{\Theta,\mathbf{g}}(\Tilde{\mathbf{u}}_\Theta)(\tau) = \mathbf{S}_\Theta(\tau)(\mathbf{I} - \mathbf{P}_\Theta) (\mathbf{f} - \mathbf{g}),
\end{align*}
which yields
\begin{align*}
    \|\mathbf{K}_{\Theta,\mathbf{f}}(\Tilde{\mathbf{u}}_\Theta)(\tau) - \mathbf{K}_{\Theta,\mathbf{g}}(\Tilde{\mathbf{u}}_\Theta)(\tau)\|_{\Hil^k} \lesssim e^{-\eps_\ast \tau}\|\mathbf{f} - \mathbf{g}\|_{\Hil^k}.
\end{align*}
Together with the previous we conclude
\begin{align*}
    \|\mathbf{u}_\Theta - \Tilde{\mathbf{u}}_\Theta\|_{\X} \leq 2 \|\mathbf{K}_{\Theta,\mathbf{f}}(\Tilde{\mathbf{u}}_\Theta) -\mathbf{K}_{\Theta,\mathbf{g}}(\Tilde{\mathbf{u}}_\Theta) \|_{\X} \lesssim \|\mathbf{f} - \mathbf{g}\|_{\Hil^k},
\end{align*}
which is the claimed Lipschitz-continuity and this finishes the proof.
\end{proof}
\subsection{Evolution near the blowup solution}
In order to ensure the smallness of the initial data, we analyze the initial data operator. 
\begin{defin}
Let $d \geq 3$, $k \geq k_0$ and $(T,X,\Theta) \in [\frac{1}{2},\frac{3}{2}]\times \overline{\B_\frac{1}{2}^d}\times \overline{\B_{M_0}^{p(d)}} $. We define the map $\mathbf{U}_{T,X,\Theta}:C^\infty(\overline{\B^d_{T}(X)},\C^{d})^2 \to \Hil^k$ via
\begin{align*}
    \mathbf{U}_{T,X,\Theta}(\mathbf{f}) = \mathbf{f}^{T,X} + \mathbf{v}_0^{T,X} - \mathbf{v}_\Theta.
\end{align*}
Now we introduce some more notation. We define 
\begin{align*}
    \mathbf{f}_{0,\Theta}^1 &:= \pd_T\Big|_{T= 1} \left( \mathbf{v}_0^{T,0} - \mathbf{v}_\Theta  \right)\\
    \mathbf{f}_{j,\Theta}^1 &:= \pd_{X^j}\Big|_{X = 0} \left( \mathbf{v}_0^{1,X} - \mathbf{v}_\Theta  \right)\\
    \mathbf{f}_{a,\Theta}^0 &:= \pd_{\Theta^a}\left( \mathbf{v}_0^{1,0} - \mathbf{v}_\Theta  \right)
\end{align*}
for $\Theta \in \overline{\B_{M_0}^{p(d)}}$, $j = 1,\ldots,d$ and $a =1,\ldots,p(d)$. Even though $\mathbf{f}_{0,\Theta}^1,\mathbf{f}_{j,\Theta}^1$ do not depend on $\Theta$, we keep this notation because of the following. The action of the symmetries were defined precisely to ensure
\begin{align*}
    \mathbf{f}_{0,\Theta}^1,\mathbf{f}_{j,\Theta}^1 &\in \rg(\mathbf{P}_{1,\Theta})\\
    \mathbf{f}_{a,\Theta}^0 &\in \rg(\mathbf{P}_{0,\Theta}).
\end{align*}

\end{defin}
\begin{lem}\label{initial data map small}
Let $d \geq 3$, $k \geq k_0$ and $(T,X,\Theta) \in [\frac{1}{2},\frac{3}{2}]\times \overline{\B_\frac{1}{2}^d}\times \overline{\B_{M_0}^{p(d)}}$. We have
\begin{align*}
    \mathbf{U}_{T,X,\Theta}(\mathbf{f}) &= \mathbf{f}^{T,X} + (T-1)\mathbf{f}_{0,\Theta}^1 + X^j \mathbf{f}_{j,\Theta}^1 + \Theta^a \mathbf{f}^0_{a,\Theta} + \mathbf{r}_{T,X,\Theta},
\end{align*}
for $\mathbf{f} \in C^\infty(\overline{\B_T^d(X)},\C^{d})^2$ with
\begin{align*}
    \|\mathbf{r}_{T,X,\Theta}\|_{\Hil^k} \lesssim (T -1)^2 + |X|^2 + |\Theta|^2
\end{align*}
for all $(T,X,\Theta) \in [\frac{1}{2},\frac{3}{2}] \times \overline{\B_\frac{1}{2}^d}\times \overline{\B_{M_0}^{p(d)}}$.
\end{lem}
\begin{proof}
We fix $\xi \in \overline{\B^d}$ and consider the map $[\frac{1}{2},\frac{3}{2}] \times \overline{\B_\frac{1}{2}^d}\times \overline{\B_{M_0}^{p(d)}} \to (\C^{d})^2$ defined by $(T,X,\Theta) \mapsto \mathbf{v}_0^{T,X}(\xi) - \mathbf{v}_\Theta(\xi)$. Since this map vanishes in $(1,0,0)$, we have by Taylor's theorem
\begin{align*}
    \mathbf{v}_0^{T,X}(\xi) - \mathbf{v}_\Theta(\xi) &= (T- 1) \mathbf{f}_{0,0}^1(\xi) + X^j \mathbf{f}_{j,0}^1(\xi) + \Theta^a \mathbf{f}_{a,0}^0(\xi) + \Tilde{\mathbf{r}}_{T,X,\Theta}(\xi),
\end{align*}
with remainder
\begin{align*}
    \Tilde{\mathbf{r}}_{T,X,\Theta}(\xi) &= (T-1)^2\int_0^1 (1-s)  \pd_{T'}^2\Big|_{T' = 1 + s(T-1)} \mathbf{v}_0^{T',s X}(\xi) ds\\
    &+ 2(T-1)X^j \int_0^1 (1-s) \pd_{T'}\pd_{(X')^j} \Big|_{T'=1 + s(T-1),X' = s X} \mathbf{v}_0^{T',X'}(\xi) d s\\
    &+ X^iX^j \int_0^1 (1-s) \pd_{(X')^i} \pd_{(X')^j} \Big|_{X' = s X} \mathbf{v}_0^{1 + s(T-1),X'}(\xi) d s\\
    &- \Theta^a \Theta^{a'} \int_0^1 (1-s) \pd_{(\Theta')^a}\pd_{(\Theta')^{a'}} \Big|_{\Theta' = s \Theta}\mathbf{v}_{\Theta'}(\xi) ds.
\end{align*}
Hence we can write
\begin{align*}
    \mathbf{v}_0^{T,X} - \mathbf{v}_\Theta &= (T - 1)\mathbf{f}_{0,\Theta}^1 + X^j \mathbf{f}_{j,\Theta}^1 + \Theta^a \mathbf{f}_{a,\Theta}^0 + \mathbf{r}_{T,X,\Theta},
\end{align*}
where
\begin{align*}
    \mathbf{r}_{T,X,\Theta} = \Tilde{\mathbf{r}}_{T,X,\Theta} - (T-1) (\mathbf{f}_{0,\Theta}^1 - \mathbf{f}_{0,0}^1) - X^j (\mathbf{f}_{j,\Theta}^1 - \mathbf{f}_{j,0}^1) - \Theta^a (\mathbf{f}_{a,\Theta}^0 - \mathbf{f}_{a,0}^0).
\end{align*}
Since $\mathbf{v}_0^{T,X}(\xi) - \mathbf{v}_\Theta(\xi)$ is jointly smooth in $(T,X,\Theta,\xi)$, we see that all of its derivatives remain bounded. Hence the explicit form of $ \Tilde{\mathbf{r}}_{T,X,\Theta}$ together with the fundamental theorem of calculus applied to the last three terms yields 
\begin{align*}
    \|\mathbf{r}_{T,X,\Theta}\|_{\Hil^k} \lesssim (T-1)^2 + |X|^2 + |\Theta|^2
\end{align*}
for all $(T,X,\Theta) \in [\frac{1}{2},\frac{3}{2}] \times \overline{\B_\frac{1}{2}^d}\times \overline{\B_{M_0}^{p(d)}}$.
\end{proof}
Next, we show that the correction term can be removed for the solutions obtained in \autoref{stabilized by correction} for small data.
\begin{prop}\label{removing corr}
Let $d \geq 3$ and $k \geq k_0$. There exist $0 < c_\ast,\delta_\ast \leq 1$ such that for all $\delta \in (0 , \delta_\ast]$ and real-valued $\mathbf{f} \in C^\infty(\overline{\B_2^d},\R^{d})^2$ with $\|\mathbf{f}\|_{H^k \times H^{k-1}(\B_2^d,\R^{d})} \leq c_\ast^2\delta $ there are parameters $T_\ast \in [1-c_\ast \delta,1+c_\ast \delta]$, $X_\ast \in \overline{\B_{c_\ast \delta}^d}$, $\Theta_\ast \in \overline{\B^{p(d)}_{c_\ast \delta}}$ and a unique $\mathbf{u}_{T_\ast,X_\ast,\Theta_\ast} \in C([0,\infty),\Hil^k)$ such that $\|\mathbf{u}_{T_\ast,X_\ast,\Theta_\ast}\|_{\X} \leq \delta $ and 
\begin{align*}
    \mathbf{u}_{T_\ast,X_\ast,\Theta_\ast}(\tau) = \mathbf{S}_{\Theta_\ast}(\tau) \mathbf{U}_{T_\ast,X_\ast,\Theta_\ast}(\mathbf{f}) + \int_0^\tau \mathbf{S}_{\Theta_\ast}(\tau-\tau') \mathbf{N}_{\Theta_\ast}(\mathbf{u}_{T_\ast,X_\ast,\Theta_\ast}(\tau')) d \tau'
\end{align*}
holds for all $\tau \geq 0$.
\end{prop}
\begin{proof}
Let $0 < c_0,\delta_0 \leq 1$ be as in \autoref{stabilized by correction}. We let $0 < \delta_\ast \leq \delta_0$ and $ 0 < c_\ast \leq c_0$ and fix their values later. Let $\delta \in (0,\delta_\ast]$ and $\mathbf{f} \in C^\infty(\overline{\B_2^d},\R^d)^2$ with $\|\mathbf{f}\|_{H^k \times H^{k-1}(\B_2^d,\R^d)} \leq c_\ast^2 \delta$. For $c_\ast$ sufficiently small (in particular such that $ c_\ast \delta_0 \leq M_0$) we have by \autoref{initial data map small} 
\begin{align*}
    \|\mathbf{U}_{T,X,\Theta}(\mathbf{f})\|_{\Hil^k} &\lesssim \|\mathbf{f}^{T,X}\|_{\Hil^k} + |T - 1| \|\mathbf{f}_{0,\Theta}^1\|_{\Hil^k} + |X|\sum_j \|\mathbf{f}_{j,\Theta}^1\|_{\Hil^k} + |\Theta| \sum_a \|\mathbf{f}_{a,\Theta}^0\|_{\Hil^k} + \|\mathbf{r}_{T,X,\Theta}\|_{\Hil^k}\\
    &\lesssim \|\mathbf{f}\|_{H^k \times H^{k-1}(\B_2^d,\R^d)} + c_\ast \delta \leq  c_\ast^2 \delta+c_\ast \delta
\end{align*}
for all $T \in [1 - c_\ast \delta,1 + c_\ast \delta]$, $X \in \overline{\B^d_{c_\ast \delta}}$ and $\Theta \in \overline{\B^{p(d)}_{c_\ast \delta}}$. Hence we can now choose $c_\ast$ small enough such that $\|\mathbf{U}_{T,X,\Theta}(\mathbf{f})\|_{\Hil^k} \leq  c_0 \delta $ and hence $\mathbf{U}_{T,X,\Theta}(\mathbf{f}) \in \Hil^k$ satisfies the assumptions of the initial data in \autoref{stabilized by correction}. Hence there exists $\mathbf{u}_{T,X,\Theta} \in \X$ with $\|\mathbf{u}_{T,X,\Theta}\|_{\X} \leq \delta$ such that 
\begin{align*}
    \mathbf{u}_{T,X,\Theta}(\tau) = \mathbf{S}_\Theta(\tau)(\mathbf{U}_{T,X,\Theta}(\mathbf{f})  - \mathbf{C}_{\Theta}(\mathbf{U}_{T,X,\Theta}(\mathbf{f}) ,\mathbf{u}_{T,X,\Theta}(\tau))) + \int_0^\tau \mathbf{S}_\Theta(\tau-\tau') \mathbf{N}_\Theta(\mathbf{u}_{T,X,\Theta}(\tau')) d \tau'.
\end{align*}
If we are able to find parameters $(T,X,\Theta)$ such that $\mathbf{C}_{\Theta}(\mathbf{U}_{T,X,\Theta}(\mathbf{f}) ,\mathbf{u}_{T,X,\Theta}(\tau)) = 0$, then we are done. Since the correction term lies in the unstable subspace $\rg(\mathbf{P}_\Theta)$, which is finite-dimensional, it suffices to prove that for some parameters the functional $\ell_{T,X,\Theta} : \rg(\mathbf{P}_\Theta) \to \C$ defined by
\begin{align*}
    \ell_{T,X,\Theta}(\mathbf{g}) = \Big(\mathbf{C}_{\Theta}(\mathbf{U}_{T,X,\Theta}(\mathbf{f}) ,\mathbf{u}_{T,X,\Theta}(\tau)) \Big| \mathbf{g}  \Big)_{\Hil^k}
\end{align*}
vanishes on a basis of $\rg(\mathbf{P}_\Theta)$. By \autoref{initial data map small} we can write
\begin{align*}
    \ell_{T,X,\Theta}(\mathbf{g}) &= (\mathbf{P}_\Theta \mathbf{f}^{T,X} | \mathbf{g} )_{\Hil^k} + (T - 1)(\mathbf{f}_{0,\Theta}^1|\mathbf{g})_{\Hil^k} + X^j (\mathbf{f}_{j,\Theta}^1|\mathbf{g})_{\Hil^k} +\Theta^a (\mathbf{f}_{a,\Theta}^0|\mathbf{g})_{\Hil^k} + (\mathbf{P}_\Theta \mathbf{r}_{T,X,\Theta}|\mathbf{g})_{\Hil^k}\\
    &+ \Big( \mathbf{P}_{0,\Theta} \int_0^\infty \mathbf{N}_\Theta(\mathbf{u}_{T,X,\Theta}(\tau')) d\tau' \Big| \mathbf{g} \Big)_{\Hil^k} + \Big( \mathbf{P}_{1,\Theta} \int_0^\infty e^{-\tau'} \mathbf{N}_\Theta(\mathbf{u}_{T,X,\Theta}(\tau')) d\tau' \Big| \mathbf{g} \Big)_{\Hil^k}.
\end{align*}
Now we consider the Gram matrix $\Gamma(\Theta)$ associated to $\{\mathbf{f}_{0,\Theta}^1\}, \{\mathbf{f}_{j,\Theta}^1\}, \{\mathbf{f}_{a,\Theta}^0\}$ which is given by
\begin{align*}
    \Gamma(\Theta)_{0 0} &= (\mathbf{f}_{0,\Theta}^1 |\mathbf{f}_{0,\Theta}^1 )_{\Hil^k}\\
    \Gamma(\Theta)_{i 0} &= (\mathbf{f}_{i,\Theta}^1 |\mathbf{f}_{0,\Theta}^1 )_{\Hil^k}\\
    \Gamma(\Theta)_{0 j} &= (\mathbf{f}_{0,\Theta}^1 |\mathbf{f}_{j,\Theta}^1 )_{\Hil^k}\\
    \Gamma(\Theta)_{a0} &= (\mathbf{f}_{a,\Theta}^0 |\mathbf{f}_{0,\Theta}^1 )_{\Hil^k}\\
    \Gamma(\Theta)_{0b} &= (\mathbf{f}_{0,\Theta}^1 |\mathbf{f}_{b,\Theta}^0 )_{\Hil^k}\\
    \Gamma(\Theta)_{i j} &= (\mathbf{f}_{i,\Theta}^1 |\mathbf{f}_{j,\Theta}^1 )_{\Hil^k}\\
    \Gamma(\Theta)_{a j} &= (\mathbf{f}_{a,\Theta}^0 |\mathbf{f}_{j,\Theta}^1 )_{\Hil^k}\\
    \Gamma(\Theta)_{i  b} &= (\mathbf{f}_{i,\Theta}^1 |\mathbf{f}_{b,\Theta}^0 )_{\Hil^k}\\
    \Gamma(\Theta)_{a  b} &= (\mathbf{f}_{a,\Theta}^0 |\mathbf{f}_{b,\Theta}^0 )_{\Hil^k}.
\end{align*}
Denote the entries of the inverse matrix by $\Gamma(\Theta)^{a b}$. Then we define 
\begin{align*}
    \mathbf{g}_\Theta^{1,0} &= \Gamma(\Theta)^{0 0} \mathbf{f}_{0,\Theta}^1 + \Gamma(\Theta)^{0 j} \mathbf{f}_{j,\Theta}^1 + \Gamma(\Theta)^{0 a} \mathbf{f}_{a,\Theta}^0\\
    \mathbf{g}_\Theta^{1,i} &= \Gamma(\Theta)^{i 0} \mathbf{f}_{0,\Theta}^1 + \Gamma(\Theta)^{i j} \mathbf{f}_{j,\Theta}^1 + \Gamma(\Theta)^{i a} \mathbf{f}_{a,\Theta}^0\\
    \mathbf{g}_\Theta^{1,b} &= \Gamma(\Theta)^{b 0} \mathbf{f}_{0,\Theta}^1 + \Gamma(\Theta)^{b j} \mathbf{f}_{j,\Theta}^1 + \Gamma(\Theta)^{b a} \mathbf{f}_{a,\Theta}^0.
\end{align*}
Then one checks that this defines a dual basis to $\{\mathbf{f}_{0,\Theta}^1\}, \{\mathbf{f}_{j,\Theta}^1\}, \{\mathbf{f}_{a,\Theta}^0\}$, i.e.,
\begin{align*}
    (\mathbf{f}_{0,\Theta}^1|\mathbf{g}_{\Theta}^{1,0})_{\Hil^k} &= 1\\
    (\mathbf{f}_{i,\Theta}^1|\mathbf{g}_{\Theta}^{1,j})_{\Hil^k} &= \delta_{i }^j\\
    (\mathbf{f}_{a,\Theta}^0|\mathbf{g}_{\Theta}^{0,b})_{\Hil^k} &= \delta_{a}^b\\
    (\mathbf{f}_{0,\Theta}^1|\mathbf{g}_{\Theta}^{1,j})_{\Hil^k} &= 0\\
    (\mathbf{f}_{0,\Theta}^1|\mathbf{g}_{\Theta}^{0,b})_{\Hil^k} &= 0\\
    (\mathbf{f}_{i,\Theta}^1|\mathbf{g}_{\Theta}^{1,0})_{\Hil^k} &= 0\\
    (\mathbf{f}_{i,\Theta}^1|\mathbf{g}_{\Theta}^{0,b})_{\Hil^k} &= 0\\
    (\mathbf{f}_{a,\Theta}^0|\mathbf{g}_{\Theta}^{1,0})_{\Hil^k} &= 0\\
    (\mathbf{f}_{a,\Theta}^0|\mathbf{g}_{\Theta}^{1,j})_{\Hil^k} &= 0.
\end{align*}
Since $\{\mathbf{f}_{0,\Theta}^1\},\{\mathbf{f}_{j,\Theta}^1\},\{\mathbf{f}_{a,\Theta}^0\}$ depend smoothly on $\Theta$, so do $\{\mathbf{g}_\Theta^{1,0}\},\{\mathbf{g}_\Theta^{1,j}\},\{\mathbf{g}_{\Theta}^{0,a}\}$ by Cramer's rule. Now we consider the function $F:[1 - c_\ast \delta,1 + c_\ast \delta] \times \overline{\B_{c_\ast \delta}^d} \times \overline{\B_{c_\ast \delta}^{p(d)}} \to \R^{1 + d + p(d)}$
with components
\begin{align*}
    F^0(T,X,\Theta) &=1  - (\mathbf{P}_\Theta \mathbf{f}^{T,X}| \mathbf{g}_\Theta^{1,0})_{\Hil^k} - (\mathbf{P}_\Theta \mathbf{r}_{T,X,\Theta}|\mathbf{g}_\Theta^{1,0})_{\Hil^k}\\
    &-\Big( \mathbf{P}_{0,\Theta} \int_0^\infty \mathbf{N}_\Theta(\mathbf{u}_{T,X,\Theta}(\tau')) d\tau' \Big| \mathbf{g}_\Theta^{1,0} \Big)_{\Hil^k} \\
    &- \Big( \mathbf{P}_{1,\Theta} \int_0^\infty e^{-\tau'} \mathbf{N}_\Theta(\mathbf{u}_{T,X,\Theta}(\tau')) d\tau' \Big| \mathbf{g}_\Theta^{1,0} \Big)_{\Hil^k}\\
    F^j(T,X,\Theta) &=  - (\mathbf{P}_\Theta \mathbf{f}^{T,X}| \mathbf{g}_\Theta^{1,j})_{\Hil^k} - (\mathbf{P}_\Theta \mathbf{r}_{T,X,\Theta}|\mathbf{g}_\Theta^{1,j})_{\Hil^k}\\
    &-\Big( \mathbf{P}_{0,\Theta} \int_0^\infty \mathbf{N}_\Theta(\mathbf{u}_{T,X,\Theta}(\tau')) d\tau' \Big| \mathbf{g}_\Theta^{1,j} \Big)_{\Hil^k} \\
    &- \Big( \mathbf{P}_{1,\Theta} \int_0^\infty e^{-\tau'} \mathbf{N}_\Theta(\mathbf{u}_{T,X,\Theta}(\tau')) d\tau' \Big| \mathbf{g}_\Theta^{1,j} \Big)_{\Hil^k}\\
    F^a(T,X,\Theta) &=  - (\mathbf{P}_\Theta \mathbf{f}^{T,X}| \mathbf{g}_\Theta^{0,a})_{\Hil^k} - (\mathbf{P}_\Theta \mathbf{r}_{T,X,\Theta}|\mathbf{g}_\Theta^{0,a})_{\Hil^k}\\
    &-\Big( \mathbf{P}_{0,\Theta} \int_0^\infty \mathbf{N}_\Theta(\mathbf{u}_{T,X,\Theta}(\tau')) d\tau' \Big| \mathbf{g}_\Theta^{0,a} \Big)_{\Hil^k} \\
    &- \Big( \mathbf{P}_{1,\Theta} \int_0^\infty e^{-\tau'} \mathbf{N}_\Theta(\mathbf{u}_{T,X,\Theta}(\tau')) d\tau' \Big| \mathbf{g}_\Theta^{0,a} \Big)_{\Hil^k}.
\end{align*}
Note that the components of $F$ indeed take values in $\R$. This can be verified by first checking that $\mathbf{U}_{T,X,\Theta}(\mathbf{f})$ being real-valued implies that $\mathbf{u}_{T,X,\Theta}(\tau)$ is real-valued for all $\tau \geq 0$ and secondly that all of $\mathbf{P}_{0,\Theta},\mathbf{P}_{1,\Theta}, \mathbf{N}_\Theta$ preserve the property of being real-valued.

Next we claim that $F$ is continuous. We have seen already that $\{\mathbf{g}_{\Theta}^{1,0}\},\{\mathbf{g}_{\Theta}^{1,j}\},\{\mathbf{g}_{\Theta}^{0,a}\}$ depend smoothly and in particular continuously on $\Theta$. In the proof of \autoref{perturbed spectrum} we have seen that $\mathbf{P}_{0,\Theta},\mathbf{P}_{1,\Theta}$ also depend continuously on $\Theta$. Next one checks, similarly to \eqref{Ntheta lipschitz}, that 
\begin{align*}
    \|\mathbf{N}_{\Theta_1}(\mathbf{g}_1) - \mathbf{N}_{\Theta_2}(\mathbf{g}_2)\|_{\Hil^k} \lesssim |\Theta_1 - \Theta_2| + \|\mathbf{g}_1 - \mathbf{g}_2\|_{\Hil^k}
\end{align*}
for all $\Theta_1,\Theta_2 \in \overline{\B_{c_\ast \delta}^{p(d)}}$ and $\mathbf{g}_1,\mathbf{g}_2 \in \Hil_\delta^k$. For fixed $\mathbf{f} \in C^\infty(\overline{\B_2^d},\R^d)^2$ we have that $(T,X,\Theta) \mapsto \mathbf{U}_{T,X,\Theta}(\mathbf{f})$ is continuous as a map from $[1 - c_\ast \delta,1 + c_\ast \delta] \times \overline{\B_{c_\ast \delta}^d} \times \overline{\B_{c_\ast \delta}^{p(d)}}$ to $\Hil^k$. By \autoref{stabilized by correction} we know that $\mathbf{u}_{T,X,\Theta}$ depends Lipschitz-continuously on its initial data $\mathbf{U}_{T,X,\Theta}(\mathbf{f})$ and hence also $\mathbf{u}_{T,X,\Theta}$ depends continuously on $(T,X,\Theta)$. With this one concludes that $F$ is continuous.

Now we can use the estimates from \ref{Ntheta quadratic} and \autoref{initial data map small} to conclude
\begin{align*}
    |F^0(T,X,\Theta)  -1  | &\lesssim \|\mathbf{f}^{T,X}\|_{\Hil^k} + \|\mathbf{r}_{T,X,\Theta}\|_{\Hil^k} + \|\mathbf{u}_{T,X,\Theta}\|_{\X}^2\\
    &\lesssim c_\ast^2 \delta  + c_\ast^2 \delta^2  + \delta^2, 
\end{align*}
and similarly 
\begin{align*}
    |F^j(T,X,\Theta)| &\lesssim c_\ast^2 \delta  +  c_\ast^2 \delta^2 + \delta^2\\
    |F^a(T,X,\Theta)| &\lesssim c_\ast^2 \delta  +  c_\ast^2 \delta^2 + \delta^2.
\end{align*}
Choosing $c_\ast $ and $\delta_\ast$ sufficiently small, we conclude that $F$ is a continuous self-map on $[1 - c_\ast \delta,1 + c_\ast \delta] \times \overline{\B^d_{c_\ast \delta}} \times \overline{\B^{p(d)}_{c_\ast \delta}}$ for all $\delta \in (0,\delta_\ast)$.
Brouwer's fixed point theorem thus yields the existence of a fixed point, i.e., $(T_\ast,X_\ast,\Theta_\ast) \in [1 - c_\ast \delta,1 + c_\ast \delta] \times \overline{\B^d_{c_\ast \delta}} \times \overline{\B^{p(d)}_{c_\ast \delta}}$ such that $F(T_\ast,X_\ast,\Theta_\ast) = (T_\ast,X_\ast,\Theta_\ast)$. With this we just compute
\begin{align*}
    \ell_{T_\ast,X_\ast,\Theta_\ast}(\mathbf{g}_{\Theta_\ast}^{1,0}) &= T_\ast - F_{0}(T_\ast,X_\ast,\Theta_\ast) = T_\ast - T_\ast = 0\\
    \ell_{T_\ast,X_\ast,\Theta_\ast}(\mathbf{g}_{\Theta_\ast}^{1,j}) &= X_\ast^j - F^j(T_\ast,X_\ast,\Theta_\ast) = X_\ast^j - X_\ast^j = 0\\
    \ell_{T_\ast,X_\ast,\Theta_\ast}(\mathbf{g}_{\Theta_\ast}^{0,a}) &= \Theta_\ast^a - F^a(T_\ast,X_\ast,\Theta_\ast) = \Theta_\ast^a - \Theta_\ast^a = 0,
\end{align*}
as claimed.

This shows that $\ell_{T_\ast,X_\ast,\Theta_\ast} = 0$ and thus $\mathbf{C}_{\Theta_\ast}(\mathbf{U}_{T_\ast,X_\ast,\Theta_\ast}(\mathbf{f}),\mathbf{u}_{T_\ast,X_\ast,\Theta_\ast}) = 0$, which finishes the proof.
\end{proof}
\subsection{Proof of the main result}
\begin{proof}[Proof of \autoref{main result}]
First, combining the fact that $U_\ast(0,\cdot)$ maps $\overline{\B_2^d}$ into a subset of $\S^d$ that has a positive distance to the south pole and that we have the Sobolev embedding $H^k(\B_2^d,\R^{d+1}) \hookrightarrow C(\overline{\B_2^d},\R^{d+1})$, we find some $\delta_\ast > 0$ such that $\|F -U_\ast(0,\cdot) \|_{H^k(\B_2^d,\R^{d+1})} \leq \delta_\ast $ implies that $F(\overline{\B_2^d})$ has a positive distance to the south pole for any $F \in C^\infty(\overline{\B_2^d},\R^{d+1})$. Now choose $0 < c_\ast \leq 1 $ as in \autoref{removing corr}. We observe that the stereographic projection $\Psi : \S^{d}\setminus \{S\} \to \R^d$ has a smooth extension to the half-space $\{(\Tilde{y},y^{d+1}) \in \R^{d+1}: y^{d+1} > -1\}$ given by $(\Tilde{y}, y^{d+1})\mapsto \frac{1}{1 + y^{d+1}}\Tilde{y}$. By using smooth cut-offs, we can hence construct $\Tilde{\Psi} \in C^\infty(\R^{d+1},\R^d)$ that coincides with $\Psi$ away from a neighborhood of $S$ on $\S^d$. Hence by applying Schauder estimates to $\Tilde{\Psi}$ (and $D \Tilde{\Psi} : \R^{d+1}\times \R^{d+1} \to \R^d$), we conclude, by shrinking $c_\ast$ further if necessary, that
\begin{align}\label{FG close to U}
    \|F - U_\ast(0,\cdot)\|_{H^k(\B_2^d,\R^{d+1})} + \|G - \pd_0 U_\ast(0,\cdot)\|_{H^{k-1}(\B_2^d,\R^{d+1})} \leq c_\ast^3 \delta_\ast
\end{align}
implies
\begin{align}\label{Psi FG close to Psi U}
    \|(\Psi(F),D \Psi(F)G) - (\Psi(U_\ast(0,\cdot)),D \Psi(U_\ast(0,\cdot))\pd_0 U_\ast(0,\cdot))\|_{H^k \times H^{k-1}(\B_2^d,\R^d)} \leq c_\ast^2 \delta_\ast
\end{align}
for all $F,G \in C^\infty(\overline{\B_2^d},\R^{d+1})$.

Now we leave $c_\ast,\delta_\ast$ fixed and assume that $F,G \in C^\infty(\overline{\B_2^d},\R^{d+1})$ with $F \cdot G = 0$ and such that \eqref{FG close to U} holds. We define
\begin{align*}
    \mathbf{f} := \Vector{\Psi(F) - \Psi(U_\ast(0,\cdot))}{D \Psi(F)G -D \Psi(U_\ast(0,\cdot))\pd_0 U_\ast(0,\cdot) }.
\end{align*}
Then $\mathbf{f} \in C^\infty(\overline{\B_2^d},\R^d)$ and $\|\mathbf{f}\|_{H^k \times H^{k-1}(\B_2^d,\R^d)} \leq c_\ast^2 \delta_\ast$ by \eqref{Psi FG close to Psi U}. \autoref{removing corr} hence yields the existence of $T_\ast \in [1-c_\ast \delta_\ast,1 + c_\ast \delta_\ast]$, $X_\ast \in \overline{\B_{c_\ast \delta_\ast}^d}$, $\Theta_\ast \in \overline{\B^{p(d)}_{c_\ast \delta_\ast}}$ and $\mathbf{u} \in C([0,\infty),\Hil^k)$ such that
\begin{align}\label{u decays}
    \|\mathbf{u}(\tau)\|_{\Hil^k} \lesssim e^{-\eps_\ast \tau}
\end{align}
and
\begin{align*}
    \mathbf{u}(\tau) = \mathbf{S}_{\Theta_\ast}(\tau) \mathbf{U}_{T_\ast,X_\ast,\Theta_\ast}(\mathbf{f}) + \int_0^\tau \mathbf{S}_{\Theta_\ast}(\tau - \tau') \mathbf{N}_{\Theta_\ast}(\mathbf{u}(\tau')) d \tau'
\end{align*}
for all $\tau \geq 0$. We have
\begin{align*}
    \mathbf{U}_{T_\ast,X_\ast,\Theta_\ast}(\mathbf{f}) = \mathbf{f}^{T_\ast,X_\ast} + \mathbf{v}_0^{T_\ast,X_\ast} - \mathbf{v}_{\Theta_\ast} \in C^\infty(\overline{\B^d},\R^d)^2 \subseteq \D(\mathbf{L}_{\Theta_\ast}).
\end{align*}
Also, since $\mathbf{N}_{\Theta_\ast}$ is locally Lipschitz one can conclude from a Gronwall argument that also $\tau \mapsto \mathbf{N}_{\Theta_\ast}(\mathbf{u}(\tau))$ is Lipschitz as a map $[0,\infty) \to \Hil^k$. Now we apply \cite[Corollary 4.2.11, p.~109]{Paz83} to conclude that we have $\mathbf{u} \in C^1((0,\infty),\Hil^k)$, $\mathbf{u}(\tau) \in \D(\mathbf{L}_{\Theta_\ast})$ for all $\tau \geq 0$ and $\mathbf{u}$ satisfies
\begin{align}\label{first order u solves Cauchy}
    \begin{cases}
        \pd_\tau \mathbf{u}(\tau) = \mathbf{L}_{\Theta_\ast} \mathbf{u}(\tau) + \mathbf{N}_{\Theta_\ast}(\mathbf{u}(\tau)), \quad \tau > 0\\
        \mathbf{u}(0) = \mathbf{U}_{T_\ast,X_\ast,\Theta_\ast}(\mathbf{f})
    \end{cases}.
\end{align}
Now note that since $k$ is sufficiently large, Eq.~\eqref{first order u solves Cauchy} holds pointwise also in the spatial variable $\xi$ as an equation involving functions that are continuous in $\xi$.
We now set $\mathbf{v}(\tau) := \mathbf{v}_{\Theta_\ast} + \mathbf{u}(\tau)$. Since $\mathbf{v}_{\Theta_\ast} \in C^\infty(\overline{\B^d},\R^d)^2$, we still have $\mathbf{v} \in C([0,\infty),\Hil^k) \cap C^1((0,\infty),\Hil^k)$ and $\mathbf{v}(\tau) \in \D(\mathbf{L}_{\Theta_\ast}) = \D(\mathbf{L})$. By definition of $\mathbf{L}_{\Theta_\ast}$ and $\mathbf{N}_{\Theta_\ast}$, we conclude that Eq.~\eqref{first order u solves Cauchy} implies $\pd_\tau \mathbf{v}(\tau) = \mathbf{L} \mathbf{v}(\tau) + \mathbf{F}(\mathbf{v}(\tau))$ for $\tau > 0$. Writing now $\mathbf{v}(\tau)(\xi) = \Vector{v_1(\tau,\xi)}{v_2(\tau,\xi)}$, we conclude that 
\begin{align*}
    \pd_\tau v_1(\tau,\xi) &= \xi^j \pd_{j} v_1(\tau,\xi) + v_2(\tau,\xi)\\
    \pd_\tau v_2(\tau,\xi) &= \Delta v_1(\tau,\xi) - \xi^j \pd_j v_2(\tau,\xi) - v_2(\tau,\xi) - F(v_1(\tau,\xi),D_\xi v_1(\tau, \xi),v_2(\tau,\xi))
\end{align*}
holds for $\tau > 0$ and $\xi \in \overline{\B^d}$. Using these equations and the fact that the initial data $\mathbf{U}_{T_\ast,X_\ast,\Theta_\ast}(\mathbf{f}) + \mathbf{v}_{\Theta_\ast}$ is smooth, one concludes that in fact $v_1 \in C^2([0,\infty) \times \overline{\B^d})$ and $v_1$ satisfies Eq.~\eqref{int wm eq selfsim} for all $\tau \geq 0$ and $\xi \in \overline{\B^d}$.

Now we just have to change variables and we get our desired solution. Set 
\begin{align*}
    u(t,x) := v_1\left(-\log(T_\ast - t) + \log T_\ast,\frac{x - X_\ast}{T_\ast- t}   \right)
\end{align*}
and 
\begin{align*}
    U(t,x) := \Psi^{-1}(u(t,x)).
\end{align*}
By construction of the equations we have
\begin{align*}
    \pd^\mu \pd_\mu U(t,x) + (\pd^\mu U(t,x) \cdot \pd_\mu U(t,x)) U(t,x) = 0
\end{align*}
for all $(t,x ) \in \mathcal{C}_{T_\ast,X_\ast}$. By construction and the fact that $F \cdot G = 0$, the initial data is given by 
\begin{align*}
    U(0,\cdot) &= F\\
    \pd_0 U(0,\cdot) &= G.
\end{align*}
For the stability estimate, we compute for integers $0 \leq s \leq k $
\begin{align*}
    &(T_\ast - t)^{s - \frac{d}{2}}\|U(t,\cdot) - U_{\Theta_\ast}^{T_\ast,X_\ast}(t,\cdot)\|_{\Dot{H}^s(\B_{T_\ast - t}^d(X_\ast),\R^{d+1})} \\
    &=  \|U(t,X_\ast + (T_\ast - t)\cdot) - U_{\Theta_\ast}^{T_\ast,X_\ast}(t,X_\ast + (T_\ast - t)\cdot) \|_{\Dot{H}^s(\B^d,\R^{d+1})}\\
    &\overset{\text{Schauder}}{\lesssim} \|\Psi(U(t,X_\ast + (T_\ast - t)\cdot)) - \Psi(U_{\Theta_\ast}^{T_\ast,X_\ast}(t,X_\ast + (T_\ast - t)\cdot)) \|_{H^s(\B^d,\R^{d})}\\
    &=  \|v_1(-\log(T_\ast - t) + \log T_\ast,\cdot) - v_{\Theta_\ast}\|_{H^s(\B^d,\R^d)}\\
    &=  \|\mathbf{u}_1(-\log(T_\ast - t) + \log T_\ast)\|_{H^s(\B^d,\R^d)}\\
    &\lesssim  \|\mathbf{u}(-\log(T_\ast - t) + \log T_\ast)\|_{\Hil^k}\\
    &\overset{\eqref{u decays}}{\lesssim}  e^{- \eps_\ast(-\log(T_\ast - t) + \log T_\ast)} =  \left(\frac{T_\ast - t}{T_\ast} \right)^{\eps_\ast} \simeq (T_\ast - t)^{\eps_\ast}
\end{align*}
for all $t \in [0,T_\ast)$. From this homogeneous estimate, we get
\begin{align*}
    &(T_\ast - t)^{s - \frac{d}{2}} \|U(t,\cdot) - U_{\Theta_\ast}^{T_\ast,X_\ast}(t,\cdot)\|_{H^s(\B_{T_\ast - t}^d(X_\ast),\R^{d+1})}\\
    &\simeq \sum_{r  = 0}^s \underbrace{(T_\ast - t)^{s - \frac{d}{2}}}_{\lesssim (T_\ast - t)^{r - \frac{d}{2}}}\|U(t,\cdot) - U_{\Theta_\ast}^{T_\ast,X_\ast}(t,\cdot)\|_{\Dot{H}^r(\B_{T_\ast - t}^d(X_\ast),\R^{d+1})}\\
    &\lesssim (T_\ast - t)^{\eps_\ast}
\end{align*}
for all $t \in [0,T_\ast)$, which is exactly \eqref{U close to blowup}. Analogously one checks \eqref{pd0U close to blowup}, where the additional factor of $(T_\ast - t)$ appears due to taking time derivatives after changing to similarity coordinates.

Finally, uniqueness of $C^2$ solutions follows from a standard Gronwall argument.

Redefining parameters via
\begin{align*}
    \delta_\ast &\mapsto c_\ast \delta_\ast\\
    c_\ast & \mapsto c_\ast^2
\end{align*}
then yields the parameters as claimed in \autoref{main result} and this finishes the proof.
\end{proof}

\appendix
\section{Some results on Sobolev spaces}
We require some specific inequalities relating various Sobolev norms.
\subsection{Standard results}
\begin{prop}\label{standard results prop}
Let $d,m,k\geq 1$ be integers. Then the following statements hold.
\begin{enumerate}[(i)]
    \item The embedding
    \begin{align*}
        H^k(\B^d,\C^m) \hookrightarrow H^{k-1}(\B^d,\C^m)
    \end{align*}
    is compact.
    \item For any $\eps > 0$, there exists a constant $C > 0$ such that 
    \begin{align}\label{weak Ehrling ineq}
        \|f\|_{H^{k-1}(\B^d,\C^m)} \leq \eps \|f\|_{H^k(\B^d,\C^m)} + C \|f\|_{L^2(\B^d,\C^m)},\quad \forall f \in H^k(\B^d,\C^m).
    \end{align}
    \item The equivalence
    \begin{align}\label{inhomog is sum of homog and L2}
        \|f\|_{H^k(\B^d,\C^m)} \simeq \|f\|_{\Dot{H}^k(\B^d,\C^m)} + \|f\|_{L^2(\B^d,\C^m)},\quad \forall f \in H^k(\B^d,\C^m)
    \end{align}
    holds.
\end{enumerate}
\end{prop}
\begin{proof}
\underline{$(i)$:} This follows from applying compactness of $H^k(\B^d)\hookrightarrow H^{k-1}(\B^d)$ (which is a consequence of the Rellich-Kondrachov theorem, see for example \cite[p.~285, Theorem 9.16]{Bre11}) to each component.

\underline{$(ii)$:} This is a special case of \emph{Lions' Lemma} (sometimes called \emph{Ehrling's Lemma}), see \cite[p.~173]{Bre11}. For the convenience of the reader we reproduce its proof in our setting.

Fix $\eps > 0$ and consider for $n \geq 1$ the sets
\begin{align*}
    U_n := \{f \in H^{k-1}(\B^d,\C^m): \|f\|_{H^{k-1}(\B^d,\C^m)} < \eps + n\|f\|_{L^2(\B^d,\C^m)} \}.
\end{align*}
The collection $\{U_n\}_{n \geq 1}$ forms an open covering of $H^{k-1}(\B^d,\C^m)$. Since $H^k(\B^d,\C^m)$ embeds compactly into $H^{k-1}(\B^d,\C^m)$, we conclude that the closure of
\begin{align*}
    \{f \in H^k(\B^d,\C^m): \|f\|_{H^k(\B^d,\C^m)} = 1\}
\end{align*}
with respect to $\|\cdot\|_{H^{k-1}(\B^d,\C^m)}$ is compact in $H^{k-1}(\B^d,\C^m)$. Since this set is covered by $\cup_{n \geq 1}U_n$, by compactness there must exist some $N \geq 1$ such that in particular
\begin{align*}
    \{f \in H^k(\B^d,\C^m): \|f\|_{H^k(\B^d,\C^m)} = 1\} \subseteq U_N.
\end{align*}
Then we compute for $f \in H^k(\B^d,\C^m)\setminus \{0\}$
\begin{align*}
    \|f\|_{H^{k-1}(\B^d,\C^m)} &= \|f\|_{H^{k}(\B^d,\C^m)} \left\| \frac{f}{ \|f\|_{H^{k}(\B^d,\C^m)}} \right\|_{H^{k-1}(\B^d,\C^m)}\\
    &\leq \|f\|_{H^{k}(\B^d,\C^m)}\left(\eps + N \left\| \frac{f}{ \|f\|_{H^{k}(\B^d,\C^m)}} \right\|_{L^2(\B^d,\C^m)}\right)\\
    &= \eps \|f\|_{H^{k}(\B^d,\C^m)} + N \|f\|_{L^2(\B^d,\C^m)},
\end{align*}
since $\frac{f}{ \|f\|_{H^{k}(\B^d,\C^m)}} \in U_N$. Hence \eqref{weak Ehrling ineq} holds with $C = N$.

\underline{$(iii)$:} Clearly 
\begin{align*}
    \|f\|_{H^k(\B^d,\C^m)} \gtrsim \|f\|_{\Dot{H}^k(\B^d,\C^m)} + \|f\|_{L^2(\B^d,\C^m)},\quad \forall f \in H^k(\B^d,\C^m).
\end{align*}
For the other direction, we compute
\begin{align*}
    \|f\|_{H^k(\B^d,\C^m)} &= \left( \sum_{|\alpha| \leq k} \|\pd^\alpha f\|_{L^2(\B^d,\C^m)}^2 \right)^\frac{1}{2} = \left( \sum_{|\alpha| \leq k-1} \|\pd^\alpha f\|_{L^2(\B^d,\C^m)}^2 + \sum_{|\alpha| = k} \|\pd^\alpha f\|_{L^2(\B^d,\C^m)}^2 \right)^\frac{1}{2}\\
    &= (\|f\|_{H^{k-1}(\B^d,\C^m)}^2 + \|f\|_{\Dot{H}^k(\B^d,\C^m)}^2 )^\frac{1}{2} \leq \|f\|_{H^{k-1}(\B^d,\C^m)} + \|f\|_{\Dot{H}^k(\B^d,\C^m)}.
\end{align*}
Applying \eqref{weak Ehrling ineq} for $\eps = \frac{1}{2}$ yields
\begin{align*}
    \|f\|_{H^{k-1}(\B^d,\C^m)} \leq \frac{1}{2 } \|f\|_{H^k(\B^d,\C^m)} + C \|f\|_{L^2(\B^d,\C^m)}
\end{align*}
for some $C > 0$. Inserting this into the above inequality yields
\begin{align*}
    \|f\|_{H^k(\B^d,\C^m)} \leq \frac{1}{2} \|f\|_{H^k(\B^d,\C^m)} + C \|f\|_{L^2(\B^d,\C^m)} + \|f\|_{\Dot{H}^k(\B^d,\C^m)}
\end{align*}
and hence by subtracting $\frac{1}{2} \|f\|_{H^k(\B^d,\C^m)}$ and multiplying by $2$ we obtain
\begin{align*}
    \|f\|_{H^k(\B^d,\C^m)} \leq 2\|f\|_{\Dot{H}^k(\B^d,\C^m)} + 2C \|f\|_{L^2(\B^d,\C^m)} 
\end{align*}
for all $f \in H^k(\B^d,\C^m)$, which is the other bound of \eqref{inhomog is sum of homog and L2} and this finishes the proof.
\end{proof}
Combining \eqref{weak Ehrling ineq} and \eqref{inhomog is sum of homog and L2} immediately implies the following corollary.
\begin{cor}
Let $d,m,k \geq 1$ be integers. For each $\eps > 0$ there exists some $C > 0$ such that
\begin{align}\label{strong Ehrling ineq}
    \|f\|_{H^{k-1}(\B^d,\C^m)} \leq \eps \|f\|_{\Dot{H}^k(\B^d,\C^m)} + C \|f\|_{L^2(\B^d,\C^m)}, \quad \forall f \in H^k(\B^d,\C^m).
\end{align}
\end{cor}
\subsection{Controlling inhomogeneous norms with homogeneous norms}
At times it is more convenient to prove results for homogeneous norms instead of the inhomogeneous counterparts. The following result shows on which subspaces of $H^k(\B^d,\C^m)$ only controlling homogeneous norms is sufficient.
\begin{prop}\label{prop control inhomog by homog}
Let $d,m,k \geq 1$ be integers and define
\begin{align*}
    \mathcal{P}_{k,d} := \{p \in \C[\xi_1,\ldots,\xi_d]: \deg(p) \leq k-1\}
\end{align*}
and
\begin{align*}
    \mathcal{P}_{k,d,m} := \{(p_1,\ldots,p_m): p_j \in \mathcal{P}_{k,d} \text{ for }j=1,\ldots,m\}.
\end{align*}
Let $X$ be a closed subspace of $H^k(\B^d,\C^m)$. Then the following are equivalent.
\begin{enumerate}[(i)]
    \item $X \cap \mathcal{P}_{k,d,m} = \{0\}$.
    \item The inequality 
    \begin{align*}
        \|f\|_{L^2(\B^d,\C^m)} \lesssim \|f\|_{\Dot{H}^k(\B^d,\C^m)}
    \end{align*}
    holds for all $ f \in X$.
    \item The inequality 
    \begin{align}\label{control inhomog by homog}
        \|f\|_{H^k(\B^d,\C^m)} \lesssim \|f\|_{\Dot{H}^k(\B^d,\C^m)}
    \end{align}
    holds for all $ f \in X$.
\end{enumerate}
\end{prop}
\begin{proof}
\underline{$(ii) \Rightarrow (iii)$:} This follows from \eqref{inhomog is sum of homog and L2}.

\underline{$(iii) \Rightarrow (i)$:} Let $f \in X \cap \mathcal{P}_{k,d,m}$. Then $\pd^\alpha f = 0$ for all $\alpha \in \N_0^d$ with $|\alpha| = k$. Hence
\begin{align*}
    \|f\|_{H^k(\B^d,\C^m)} \overset{\eqref{control inhomog by homog}}{\lesssim} \|f\|_{\Dot{H}^k(\B^d,\C^m)} = \left( \sum_{|\alpha| = k} \|\pd^\alpha f\|_{L^2(\B^d,\C^m)}^2 \right)^\frac{1}{2} = 0,
\end{align*}
which implies $f = 0$.

\underline{$(i) \Rightarrow (ii)$:} We assume by contradiction that $(ii)$ does not hold. Then there exists a sequence $(f_n)_{n \geq 1} \subseteq X$ such that 
\begin{align*}
    \|f_n\|_{L^2(\B^d,\C^m)} > n \|f_n\|_{\Dot{H}^k(\B^d,\C^m)}
\end{align*}
for all $ n \geq 1$. In particular $f_n \not=0$ and we can define the sequence
\begin{align*}
    g_n := \frac{f_n}{\|f_n\|_{L^2(\B^d,\C^m)}}.
\end{align*}
Then by construction $\|g_n\|_{L^2(\B^d,\C^m)} = 1$ and 
\begin{align*}
    \|g_n\|_{\Dot{H}^k(\B^d,\C^m)} = \frac{\|f_n\|_{\Dot{H}^k(\B^d,\C^m)}}{\|f_n\|_{L^2(\B^d,\C^m)}} \leq \frac{1}{n}.
\end{align*}
Hence $(g_n)_{n \geq 1}$ is a bounded sequence in $H^k(\B^d,\C^m)$ and from the compactness of the embedding $H^k(\B^d,\C^m) \hookrightarrow L^2(\B^d,\C^m)$ there exists a subsequence, which is without loss of generality $(g_n)_{n \geq 1}$ itself, that converges in $L^2(\B^d,\C^m)$ to some $g \in L^2(\B^d,\C^m)$. We have
\begin{align*}
    \|g_n - g_\ell\|_{H^k(\B^d,\C^m)} &\overset{\eqref{inhomog is sum of homog and L2}}{\simeq } \|g_n - g_\ell\|_{L^2(\B^d,\C^m)} + \|g_n - g_\ell\|_{\Dot{H}^k(\B^d,\C^m)}\\
    &\leq \|g_n - g_\ell\|_{L^2(\B^d,\C^m)} +\frac{1}{n} + \frac{1}{\ell},
\end{align*}
which shows that $(g_n)_{n \geq 1}$ is Cauchy in $H^k(\B^d,\C^m)$. Thus $g \in H^k(\B^d,\C^m)$ and
\begin{align*}
    \lim_{n \to \infty} \|g_n - g\|_{H^k(\B^d,\C^m)} = 0.
\end{align*}
By continuity of the (semi-)norms we have
\begin{align*}
    \|g\|_{L^2(\B^d,\C^m)} &= \lim_{n \to \infty} \|g_n\|_{L^2(\B^d,\C^m)} = 1\\
    \|g\|_{\Dot{H}^k(\B^d,\C^m) }&= \lim_{n \to \infty} \|g_n\|_{\Dot{H}^k(\B^d,\C^m) } = 0.
\end{align*}
The latter implies that $\pd^\alpha g = 0$ for all $\alpha \in \N_0^d$ with $|\alpha| = k$. This is only possible for $g \in \mathcal{P}_{k,d,m}$. Now we note that since $(g_n)_{n \geq 1} \subseteq X$ and $X$ is closed in $H^k(\B^d,\C^m)$, we also have $g \in X$. Hence by assumption $g = 0$. This is a contradiction to $\|g\|_{L^2(\B^d,\C^m)} = 1$.
\end{proof}
An analogous statement holds for the spaces $\Hil^k$.
\subsection{Smallness from additional regularity}
For the proof of \autoref{full gen is diss up to per} we need to be able to trade one derivative for smallness. One possibility is \eqref{strong Ehrling ineq}. But here one has a potentially big term with a full $L^2(\B^d,\C^m)$-norm. The next result, which is an adaptation of \cite[Lemma 3.9]{MerRapRodSze22}, shows that one can alternatively pay with a finite-rank projection term instead of a full $L^2(\B^d,\C^m)$-term.

\begin{prop}\label{subcoer prop}
Let $d,m,k \geq 1$ be integers. Let $(\cdot|\cdot)$ be an inner product on $H^k(\B^d,\C^m)$ such that its induced norm $\|\cdot\|$ is equivalent to $\|\cdot\|_{H^k(\B^d,\C^m)}$.

For any $\eps > 0$ there exist a constant $C > 0$ and a finite set $\{f_n\}_{n=1}^N \subseteq H^k(\B^d,\C^m)$ that is orthonormal with respect to $(\cdot|\cdot)$ such that 
\begin{align}\label{subcoer ineq}
    \|f\|_{H^{k-1}(\B^d,\C^m)}^2 \leq \eps \|f\|_{\Dot{H}^k(\B^d,\C^m)}^2 + C \sum_{n=1}^N |(f|f_n)|^2
\end{align}
holds for all $f \in H^k(\B^d,\C^m)$.
\end{prop}
\begin{proof}
The Riesz representation theorem yields that for any $f \in L^2(\B^d,\C^m)$ there exists a unique $\ell(f) \in H^k(\B^d,\C^m)$ such that 
\begin{align*}
    (f | g)_{L^2(\B^d,\C^m)} = (\ell(f) | g)
\end{align*}
for all $g \in H^k(\B^d,\C^m)$. One checks that $\ell : L^2(\B^d,\C^m) \to H^k(\B^d,\C^m)$ is linear and since 
\begin{align*}
    \|\ell(f)\|^2 = (\ell(f) | \ell(f)) &= (f | \ell(f))_{L^2(\B^d,\C^m)} \leq \|f\|_{L^2(\B^d,\C^m)} \|\ell(f)\|_{L^2(\B^d,\C^m)}\\
    &\lesssim  \|f\|_{L^2(\B^d,\C^m)} \|\ell(f)\|,
\end{align*}
we have $\|\ell(f)\| \lesssim \|f\|_{L^2(\B^d,\C^m)} $, i.e., $\ell$ is bounded.

Denote by $\iota : H^k(\B^d,\C^m) \to L^2(\B^d,\C^m)$ the natural embedding and consider the map $\iota \circ \ell : L^2(\B^d,\C^m) \to L^2(\B^d,\C^m)$. We compute
\begin{align*}
    ((\iota \circ \ell)(f) | g)_{L^2(\B^d,\C^m)} &= (\ell(f) | g)_{L^2(\B^d,\C^m)} = \overline{(g | \ell(f))_{L^2(\B^d,\C^m)}} = \overline{(\ell(g) | \ell(f))}\\
    &= (\ell(f) | \ell(g)) = (f | (\iota \circ \ell)(g))_{L^2(\B^d,\C^m)}
\end{align*}
for all $f,g \in L^2(\B^d,\C^m)$, which shows that $\iota \circ \ell$ is self-adjoint. This calculation also shows 
\begin{align*}
    ((\iota \circ \ell)(f) | f)_{L^2(\B^d,\C^m)} = (\ell(f) | \ell(f)) \geq 0
\end{align*}
and hence $\iota \circ \ell$ is non-negative. Since $\ell$ is injective, we conclude that $\iota \circ \ell$ is positive. Since $\iota$ is compact, we also have that $\iota \circ \ell$ is compact. The spectral theorem of positive compact operators yields the existence of an orthonormal basis of $L^2(\B^d,\C^m)$ consisting of eigenvectors of $\iota \circ \ell$. More precisely, denote by $(\lambda_j)_{j \geq 1}$ the sequence of eigenvalues with
\begin{align*}
    \lambda_1 > \lambda_2 > \ldots \lambda_j> \ldots > 0
\end{align*}
and $\lim_{j \to \infty} \lambda_j = 0$. For $j \geq 1$ let $\{f_{j,i}\}_{i=1}^{I(j)}$ with $I(j)\in \N $ be an orthonormal basis of the eigenspace of $\iota \circ \ell$ corresponding to the eigenvalue $\lambda_j$. Since all $\lambda_j \not=0$ we have
\begin{align*}
    f_{j,i} = \frac{1}{\lambda_j} (\iota\circ \ell)(f_{j,i}) = \frac{1}{\lambda_j}\ell(f_{j,i}) \in H^k(\B^d,\C^m)
\end{align*}
for all $j \geq 1$ and $1 \leq i \leq I(j)$. Consider for $n \geq 1$ the set
\begin{align*}
    \A_n := \{f \in H^k(\B^d,\C^m) : \|f\|_{L^2(\B^d,\C^m)} = 1, (f | f_{j,i})_{L^2(\B^d,\C^m)} = 0 \text{ for all }1 \leq j \leq n, 1 \leq i \leq I(j)\}.
\end{align*}
We set 
\begin{align*}
    I_n := \inf_{f \in \A_n} \|f\|^2.
\end{align*}
Lower semi-continuity of $(\cdot|\cdot)$ and compactness of $\iota$ yield that this infimum is attained for some $g \in \A_n$, i.e., $\|g\|^2 = I_n$.

A Lagrange multiplier argument yields the existence of $\mu,\mu_{j,i} \in \C$ such that
\begin{align}\label{Lagrange multiplier identity}
    (g | f) &= \sum_{j=1}^n \sum_{i=1}^{I(j)} \mu_{j,i} (f_{j,i}|f)_{L^2(\B^d,\C^m)} + \mu (g|f)_{L^2(\B^d,\C^m)}
\end{align}
for all $f \in H^k(\B^d,\C^m)$. Since $f_{j,i}$ is an eigenvector of $\iota \circ \ell$, we have
\begin{align*}
    (g | f_{j,i} ) = \frac{1}{\lambda_j} (g | \ell(f_{j,i})) = \frac{1}{\lambda_j} (g | f_{j,i})_{L^2(\B^d,\C^m)} = 0
\end{align*}
for all $ 1 \leq j \leq n,1 \leq i \leq I(j)$, since $g \in \A_n$. Hence inserting $f = f_{j,i}$ in Eq.~\eqref{Lagrange multiplier identity} yields
\begin{align*}
    0 = (g | f_{j,i}) = \mu_{j,i}(f_{j,i}|f_{j,i})_{L^2(\B^d,\C^m)} + \mu (g|f_{j,i})_{L^2(\B^d,\C^m)}  = \mu_{j,i}
\end{align*}
and hence \eqref{Lagrange multiplier identity} reduces to 
\begin{align*}
    (g | f) = \mu (g | f)_{L^2(\B^d,\C^m)}
\end{align*}
for all $f \in H^k(\B^d,\C^m)$. By definition of $\ell$ this shows that $g$ is an eigenvector of $\iota \circ \ell$ with eigenvalue $\mu^{-1}$. Since $g \in \A_n$ this means that $\mu^{-1} \leq \lambda_{n+1}$. Also we have
\begin{align*}
    I_n = (g|g) =  \mu (g | g)_{L^2(\B^d,\C^m)} = \mu \geq \frac{1}{\lambda_{n+1}}.
\end{align*}
Now let $\eps > 0$ be given. Using \eqref{weak Ehrling ineq} we find a constant $C' > 0$ such that
\begin{align*}
    \|f\|_{H^{k-1}(\B^d,\C^m)}^2 \leq \frac{\eps}{16} \|f\|_{\Dot{H}^k(\B^d,\C^m)}^2 + C' \|f\|_{L^2(\B^d,\C^m)}^2 
\end{align*}
holds for all $ f \in H^k(\B^d,\C^m)$. For $f \in \A_n$ we compute
\begin{align*}
    \|f\|_{H^{k-1}(\B^d,\C^m)}^2 &\leq \frac{\eps}{16} \|f\|_{\Dot{H}^k(\B^d,\C^m)}^2 + C' \|f\|_{L^2(\B^d,\C^m)}^2 = \frac{\eps}{16} \|f\|_{\Dot{H}^k(\B^d,\C^m)}^2 + C'\\
    &= \frac{\eps}{16} \|f\|_{\Dot{H}^k(\B^d,\C^m)}^2 + C' \underbrace{\frac{1}{I_n}}_{\leq \lambda_{n+1}} \underbrace{I_n}_{\leq \|f\|^2}\\
    &\leq \frac{\eps}{16} \|f\|_{\Dot{H}^k(\B^d,\C^m)}^2 + C' \lambda_{n+1} \|f\|^2.
\end{align*}
Let $C'' > 0$ such that $\|h\|^2 \leq C'' \|h\|_{H^k(\B^d,\C^m)}^2 $ holds for all $h\in H^k(\B^d,\C^m) $. We choose $n$ large enough such that $\lambda_{n+1} \leq \max\left\{ \frac{\eps}{16 C' C''},\frac{1}{2 C' C''} \right\}$. Inserting into the above yields
\begin{align*}
    \|f\|_{H^{k-1}(\B^d,\C^m)}^2 &\leq \frac{\eps}{16} \|f\|_{\Dot{H}^k(\B^d,\C^m)}^2 + C' C'' \lambda_{n+1} \|f\|_{H^{k}(\B^d,\C^m)}^2\\
    &= \frac{\eps}{16} \|f\|_{\Dot{H}^k(\B^d,\C^m)}^2 + C' C'' \lambda_{n+1} \|f\|_{\Dot{H}^{k}(\B^d,\C^m)}^2+ C' C'' \lambda_{n+1} \|f\|_{H^{k-1}(\B^d,\C^m)}^2\\
    &\leq \frac{\eps}{8}\|f\|_{\Dot{H}^k(\B^d,\C^m)}^2 + \frac{1}{2} \|f\|_{H^{k-1}(\B^d,\C^m)}^2,
\end{align*}
which we can rewrite as
\begin{align*}
    \|f\|_{H^{k-1}(\B^d,\C^m)}^2 \leq \frac{\eps}{4} \|f\|_{\Dot{H}^k(\B^d,\C^m)}^2.
\end{align*}
By scaling we conclude that the above inequality holds for all $f \in H^k(\B^d,\C^m)$ such that 
\begin{align*}
    (f | f_{j,i})_{L^2(\B^d,\C^m)} = 0
\end{align*}
for all $1 \leq j \leq n, 1 \leq i \leq I(j) $. Now consider the operator $P : H^k(\B^d,\C^m) \to H^k(\B^d,\C^m)$
\begin{align*}
    P f := \sum_{j=1}^n \sum_{i=1}^{I(j)} (f | f_{j,i})_{L^2(\B^d,\C^m)} f_{j,i}.
\end{align*}
Since our previous inequality can, by construction, be applied to $f - Pf$, we have
\begin{align*}
    \|f\|_{H^{k-1}(\B^d,\C^m)}^2  &\leq  \left( \|f- Pf\|_{H^{k-1}(\B^d,\C^m)} + \|P f\|_{H^{k-1}(\B^d,\C^m)}\right)^2\\
    &\leq 2 \|f- Pf\|_{H^{k-1}(\B^d,\C^m)}^2 + 2 \|P f\|_{H^{k-1}(\B^d,\C^m)}^2\\
    &\leq \frac{\eps}{2} \|f- Pf\|_{\Dot{H}^{k}(\B^d,\C^m)}^2 + 2 \|P f\|_{H^{k-1}(\B^d,\C^m)}^2\\
    &\leq \eps \|f\|_{\Dot{H}^{k}(\B^d,\C^m)}^2 + \eps\| Pf\|_{\Dot{H}^{k}(\B^d,\C^m)}^2 + 2 \|P f\|_{H^{k-1}(\B^d,\C^m)}^2\\
    &\leq \eps \|f\|_{\Dot{H}^{k}(\B^d,\C^m)}^2 + \max\{\eps,2\} \|P f\|_{H^{k}(\B^d,\C^m)}^2
\end{align*}
for all $f \in H^k(\B^d,\C^m)$. Let $C''' > 0$ be such that $\|h\|_{H^k(\B^d,\C^m)}^2 \leq C'''\|h\|^2$ for all $h \in H^k(\B^d,\C^m)$. Then the above yields in total
\begin{align*}
    \|f\|_{H^{k-1}(\B^d,\C^m)}^2  &\leq \eps \|f\|_{\Dot{H}^{k}(\B^d,\C^m)}^2 + C''' \max\{\eps,2\} \|P f\|^2.
\end{align*}
Using that the $\{f_{j,i}\}$ are orthonormal in $L^2(\B^d,\C^m)$ together with the fact that they are eigenvectors of $\iota \circ \ell$ allows to compute
\begin{align*}
    \|P f\|^2 = \sum_{j=1}^{n} \sum_{i=1}^{I(j)} \lambda_j |(f | f_{j,i})|^2.
\end{align*}
Hence if we now write $\{f_\ell\}_{\ell=1}^N = \left\{ \sqrt{\lambda_j}f_{j,i}: 1 \leq j \leq n, 1 \leq i \leq I(j) \right\}$ with $N = \sum_{j=1}^n I(j)  $ one checks that $\{ f_l \}_{\ell=1}^N$ is orthonormal with respect to $(\cdot | \cdot)$ and the above inequality becomes
\begin{align*}
    \|f\|_{H^{k-1}(\B^d,\C^m)}^2  &\leq \eps \|f\|_{\Dot{H}^k(\B^d,\C^m)}^2 + C \sum_{\ell=1}^N |(f | f_\ell)|^2,
\end{align*}
where $C = \max\{\eps,2\} C'''$, for all $f \in H^k(\B^d,\C^m)$, which is exactly the claim.
\end{proof}

\printbibliography

\end{document}